\pdfoutput=1
\documentclass[11pt]{amsart}
\usepackage{amsmath}
\usepackage{amssymb}
\usepackage{mathtools}

\usepackage{enumerate}
\usepackage{slashed}
\usepackage{caption}
\usepackage{subcaption}
\usepackage{newlfont}
\usepackage{amsrefs}
\usepackage{comment}
\usepackage[normalem]{ulem}
\usepackage{accents}
\usepackage{leftidx}

\usepackage{harpoon}
\usepackage[pdftex]{graphicx}
\usepackage{esint}
\usepackage{relsize}
\usepackage[latin1]{inputenc}
\usepackage{xcolor}

\definecolor{brass}{rgb}{0.71, 0.65, 0.36}

\usepackage{tikz-cd}
\theoremstyle{plain}
\newtheorem{theorem}{Theorem}[section]
\newtheorem{lemma}[theorem]{Lemma}

\newtheorem{corollary}[theorem]{Corollary}

\newtheorem{proposition}[theorem]{Proposition}

\theoremstyle{definition}
\newtheorem{remark}[theorem]{Remark}
\newtheorem*{rem}{Remark}
\newtheorem{definition}[theorem]{Definition}

\numberwithin{equation}{section}

\def\Ric{\operatorname{Ric}}

\def\E{\mathcal{E}}

\def\sup{\operatorname{sup}}
\def\inf{\operatorname{inf}}

\def\sup{\operatorname{sup}}
\def\supp{\operatorname{supp}}

\def\bb{\mathbb}

\def\th{\theta}
\def\S{\Sigma}
\def\G{\Gamma}
\def\a{\alpha}

\def\p{\partial}

\theoremstyle{plain}

\numberwithin{equation}{section}

\allowdisplaybreaks

\begin{document}

\title[Mass Lower Bounds for ALF Manifolds]{Mass Lower Bounds for Asymptotically Locally Flat Manifolds}


\author[Khuri]{Marcus Khuri}
\address{Department of Mathematics\\
Stony Brook University\\
Stony Brook, NY 11794, USA}
\email{marcus.khuri@stonybrook.edu}

\author[Wang]{Jian Wang}
\address{State Key Laboratory of Mathematical Sciences,\\
Academy of Mathematics and Systems Science, Chinese Academy of Sciences, \\
Beijing 100190, China}
\email{jian.wang.4@amss.ac.cn}



\thanks{M. Khuri acknowledges the support of NSF Grants DMS-2104229 and DMS-2405045.}

\begin{abstract}
We establish positive mass type theorems for asymptotically locally flat (ALF) manifolds, which have asymptotic ends 
modeled on circle bundles over a Euclidean base with fibers of constant length. In particular for dimensions
$n\leq 7$, the mass of AF manifolds is shown to be nonnegative under the assumption of nonnegative scalar curvature if a codimension-two coordinate sphere in the asymptotic end is trivial in homology, with zero mass achieved only for the product $\mathbb{R}^{n-1}\times S^1$. The same conclusions are obtained in dimension four for ALF manifolds admitting an almost free $U(1)$ action. Moreover, in this setting the mass is shown to be bounded below by a multiple of the degree of the circle bundle at infinity. This is the first such result illustrating how nontrivial topology of the end contributes to the mass.
\end{abstract}

\maketitle

\section{Introduction} 
\label{sec1} \setcounter{equation}{0}
\setcounter{section}{1}

A landmark result in the study of scalar curvature and in mathematical relativity is the positive mass theorem, proved originally for asymptotically Euclidean (AE) manifolds of nonnegative scalar curvature by Schoen-Yau \cite{schoen-yau1979} and Witten \cite{Witten}. Versions of this theorem have been successfully extended in the asymptotically hyperbolic regime by Andersson-Cai-Galloway \cite{ACG}, Chru\'{s}ciel-Herzlich \cite{ChruscielHerzlich}, Wang \cite{XWang}, and Zhang \cite{XZhang}, to the asymptotically locally hyperbolic realm by Alaee-Hung-Khuri \cite{AHK}, Brendle-Hung \cite{BrendleHung}, and Lee-Neves \cite{LeeNeves},  as well as to the complex hyperbolic setting by Herzlich \cites{Herzlich1,Herzlich2}.  Motivated by considerations in quantum gravity, this result was also conjectured to hold for asymptotically locally Euclidean (ALE) manifolds, however counterexamples where found by LeBrun \cite{LeBrun}. Nevertheless, the positive mass theorem in the ALE case has been established for certain K\"{a}hler manifolds or under a spin structure matching condition through the work of Hein-LeBrun \cite{HeinLeBrun} and Dahl \cite{Dahl}, respectively. Furthermore, with inspiration coming from the study of gravitational instantons and Kaluza-Klein theory, Minerbe \cite{minerbe} has obtained positivity of mass for the asymptotically flat\footnote{It should be noted that the terminology of an asymptotically flat manifold has taken on two inequivalent meanings within the context of the positive mass theorem, one arising from mathematical relativity and another from the study of gravitational instantons. In this article we will use the latter notion.} (AF) and asymptotically locally flat (ALF) settings, and Liu-Shi-Zhu \cite{LSZ} as well as Chen-Liu-Shi-Zhu \cite{CLSZ} have proved incarnations for the AF and other cases; related results were additionally found by Dai \cite{Dai}, Dai-Sun \cite{DaiSun}, and Barzegar-Chru\'{s}ciel-H\"{o}rzinger \cite{BCH}. However, the AF and ALF settings remain the least well understood.

\begin{definition}\label{AF}
A connected Riemannian manifold $(M^n, g)$ of dimension $n\geq 3$ will be called \textit{asymptotically flat} (AF) if there is a compact subset $\mathcal{C}\subset M^n$, such that $M^n\setminus \mathcal{C}$ has a finite number of components (referred to as ends) each of which is diffeomorphic to $(\mathbb{R}^{n-1}\setminus B_1)\times S^1$, and after pullback the metric asymptotes to a flat product with decay
\begin{equation}
|\mathring{\nabla}^l (g-g_0)|_{g_0}=O(r^{-q-l}),\quad\quad l=0,1,2,
\end{equation}
where $q>(n-3)/2$ and $\mathring{\nabla}$ denotes the Levi-Civita connection of $g_0$. Here $g_0 =dr^2 +r^2 g_{S^{n-2}}+\ell^2 d\theta^2$ is the product metric on $\mathbb{R}^{n-1}\times S^1$, where $\theta\in [0,2\pi)$ parameterizes the $S^1$ factor and $\ell>0$ is a constant. Moreover, the scalar curvature of $g$ is required to be integrable, $R_g \in L^1(M^n)$.
\end{definition}

To each end of an AF manifold there is a well-defined notion of mass, see \eqref{massdef} below. The mass may be viewed as a geometric invariant which connects scalar curvature with the global geometry and topology of the manifold. The Euclidean Reissner-Nordstr\"{o}m metrics on $\mathbb{R}^2 \times S^2$ are AF, complete, and scalar flat, but they can have negative mass for certain choices of parameters. In \cite{minerbe}*{Theorem 1} Minerbe sought to explain this phenomenon by asserting that mass is nonnegative for ALF manifolds under the assumption of nonnegative Ricci curvature, which Reissner-Nordstr\"{o}m does not satisfy; note that Minerbe's definition of mass in this result does not coincide with the standard one. In a different direction, Liu-Shi-Zhu \cite{LSZ}*{Theorem 1.2} show that for AF manifolds of dimensions less than 8 the positive mass theorem holds if the circle at infinity is homotopically nontrivial. Our first result is related to that of Liu-Shi-Zhu, in that topological aspects of the asymptotic end lead to positivity of mass. The proof is based on stable minimal hypersurfaces. On an end $\mathcal{E}$, we will refer to the intersection of the $r$ and $\theta$ level sets as a \textit{coordinate sphere} and will denote it by $\mathcal{S}_{r,\theta}^{n-2}$. Throughout the paper all manifolds will be assumed to be smooth and orientable for convenience, however the main theorems continue to hold in the nonorientable case by passing to the orientation cover.

\begin{theorem}\label{A} 
Let $(M^n, g)$ be a complete AF manifold with nonnegative scalar curvature and $4\leq n\leq 7$. If some coordinate sphere $\mathcal{S}^{n-2}_{r,\theta}$ of an end $\mathcal{E}$ is trivial in homology $H_{n-2}(M^n;\mathbb{Z})$, then the mass is nonnegative in this end. Moreover, the mass of $\mathcal{E}$ vanishes if and only if $(M^n,g)$ is isometric to $\mathbb{R}^{n-1}\times S^1$.
\end{theorem}

In \cite{minerbe}*{Theorem 2} Minerbe considered AF manifolds with nonnegative scalar curvature and a matching condition for the spin structure at infinity, to establish a positive mass theorem. However, the restriction to AF asymptotics does not allow for an application to examples such as the (multi)-Taub-NUT geometries, which have nontrivial fibrations in the asymptotic end. To the authors' knowledge, there are no known positive mass theorems in the literature for general ALF manifolds with nonnegative scalar curvature. The two theorems below aim to fill this gap.  We first recall the notion of an ALF 4-manifold inspired by Biquard-Gauduchon-LeBrun \cite{BGL}*{Definition 1}. For a single asymptotic end, the definition in \cite{BGL} is more general than that presented here in terms of the topology of the asymptotic end, and properties of the Killing field $V$ which we require to have closed orbits. Thus, while the (multi)-Taub-NUT examples are covered by our definition, the Euclidean Kerr instanton lies outside its scope due to the requirement of closed orbits. In what follows $\iota$ and $\mathcal{L}$ will denote interior product and Lie derivative, respectively.

\begin{definition}\label{ALF} 
A connected Riemannian 4-manifold $(M^4, g)$ will be called \textit{asymptotically locally flat} (ALF) if the following conditions are satisfied.
\begin{enumerate}[(i)]
\item There is a compact subset $\mathcal{C}\subset M^4$ such that $M^4\setminus \mathcal{C}$ has a finite number of components, and any such component $\mathcal{E}$ (referred to as an end) is diffeomorphic to $\mathbb{R}_+\times \mathcal{S}^3$, where $\mathcal{S}^3$ is a principal $U(1)$ bundle over $S^2$. 

\item $\mathcal{S}^3$ is equipped with a connection 1-form $\tau$, and a vector field $V$ which serves as the infinitesimal generator of the $U(1)$ action. These objects satisfy $\iota_V \tau=1$ and $\mathcal{L}_{V}\tau=0$.


\item $\mathbb{R}_+ \times\mathcal{S}^3$ is equipped with a model metric
\begin{equation}\label{def-model}
g_0 =dr^2 +r^2 g_{S^2} + \ell^2 \tau^2,
\end{equation}
where $r$ parameterizes $\mathbb{R}_+$, $g_{S^2}$ is a pullback of the unit round metric from the base space of $\mathcal{S}^3$, and $\ell>0$ is a constant.  

\item After pulling back via the diffeomorphism $\mathbb{R}_+\times \mathcal{S}^3\rightarrow \mathcal{E}$, metric $g$ asymptotes to the model with decay
\begin{equation}\label{ALF-metric-decay}
|\mathring{\nabla}^l (g-g_0)|_{g_0}=O(r^{-q-l}), \quad\quad l=0,1,2,
\end{equation}
where $q>1/2$ and $\mathring{\nabla}$ denotes the Levi-Civita connection of $g_0$.

\item The scalar curvature of $g$ is integrable, $R_g \in L^1(M^4)$.
\end{enumerate}
\end{definition}

Observe that an AF 4-manifold is the special case of an ALF manifold when the connection 1-form is trivial. Since the classifying space is $BU(1)=\mathbb{CP}^{\infty}$ and the relevant homotopy classes of maps satisfies
$[S^2,\mathbb{CP}^{\infty}]=H^2(S^2;\mathbb{Z})=\mathbb{Z}$,
the bundles $\mathcal{S}^3$ are classified by the integers. In particular, the possible topologies for $\mathcal{S}^3$ are either $S^1 \times S^2$ or the lens spaces $L(p,1)$ for an integer $p\geq 1$. Note also that the curvature 2-form $d\tau$ descends to the base, and thus may be expressed as the pullback $\pi_s^* F$ for some 2-form $F$ on $S^2$, where $\pi_s:\mathcal{S}^3 \rightarrow S^2$ is the projection map. In fact, the de Rham cohomology class $[\tfrac{1}{2\pi}F]=\mathbf{c}_1$ is the Chern class of this bundle, and hence $F[S^2]$ is independent of the connection and gives an integer multiple (Chern number) of $2\pi$ \cite{Chern}. In the context of Definition \ref{ALF} we will refer to this integer as the \textit{degree of asymptotic end} $\mathcal{E}$. Its relation to the Chern-Simons 3-form is given by
\begin{equation}\label{degree-def1}
\deg(\mathcal{\mathcal{E}}):=\frac{1}{4\pi^2}\int_{\mathcal{S}^3}\tau \wedge d\tau=\frac{1}{2\pi}\int_{S^2}F.
\end{equation}
By choosing the orientation of the circle fibers appropriately, this integer may be assumed to be nonnegative and corresponds with the lens space parameter $p$, with $p=0$ indicating the product $S^1 \times S^2$. Another feature of ALF asymptotics concers the scalar curvature of the model metric, which decays but is not zero in general. In particular, using O'Neill's formula \cite{Besse}*{Corollary 9.37} we find
\begin{equation}\label{decay-scal}
R_{g_0}=-\frac{1}{4\ell^2}|d\tau|^2=O(r^{-4})
\end{equation}
as $r\rightarrow \infty$. 

To each end of an AF or ALF manifold there is a well-defined notion of mass. In analogy with the AE setting, this is expressed in an asymptotically Cartesian coordinate system. More precisely, each end $\mathcal{E}$ is a $U(1)$ bundle over $\mathbb{R}^{n-1} \setminus B_1$, and after pulling back with a local trivialization while changing from polar to Cartesian coordinates $x=(x^1,\dots,x^{n-1})$ on the base, the model metric takes the form 
\begin{equation}\label{connection-form-metric}
g_0 =\delta_{ij}dx^i dx^j +\ell^2 (d\theta +A_i dx^i)^2
\end{equation}
where $A_i dx^i=\sigma^* \tau$ is the local connection 1-form on $\mathbb{R}^{n-1} \setminus B_1$ obtained with the help of a local section $\sigma$. The \textit{mass} of the end $\mathcal{E}$ is then defined by
\begin{equation}\label{massdef}
m=\lim_{r\rightarrow \infty}\frac{1}{2\pi\ell \omega_{n-2}}\int_{\mathcal{S}^{n-1}_r} \star_{g_0}\left(\text{div}_{g_0}g-d \text{tr}_{g_0}g \right),
\end{equation}
where the integrand is computed with respect to the coordinates $(x,\theta)$. Here $\mathcal{S}_{r}^{n-1}$ denotes level sets of the radial coordinate $r=|x|$, the Hodge star operator with respect to $g_0$ is labeled $\star_{g_0}$, and $\omega_{n-2}$ is the volume of the unit $(n-2)$-sphere. Note that the mass of the model metric $g_0$ is zero. Moreover in local coordinates $(\text{div}_{g_0}g)(\partial_b)=\mathring{\nabla}^a g_{ab}$, so this mass expression formally agrees up to a positive multiple with the ADM mass in the asymptotically Euclidean context. The fact that the mass \eqref{massdef} is a geometric invariant independent of the choice of ALF structure is established in Proposition \ref{mass-well-def} below. Previously such a statement was claimed by Minerbe \cite{minerbe}*{pg. 952} for AF asymptotics; in \cite{minerbe}*{Proposition 6} an analogous result was established for the so called `Gauss-Bonnet mass' (tailored to Ricci curvature) in the ALF case. Note that in the AF setting, the mass \eqref{massdef} agrees with, up to a positive multiple, those of \cites{BCH,CLSZ,Dai,LSZ} and \cite{minerbe}*{Theorem 2}.

Toric and $U(1)$ symmetries play an important role in the study of gravitational instantons \cite{AADN}, \cite{MingyangLi}. In fact, in a dramatic recent development Li-Sun \cite{LiSun} have found toric instantons on infinitely many new diffeomorphism types of 4-manifolds, which are not locally Hermitian. Therefore, it is natural to consider positive mass theorems in the ALF setting which assume such symmetries. We will restrict attention to the case in which the circle action has a finite number of fixed points. In order to state the next definition note that the infinitesimal generator $V$, for the $U(1)$ action on $\mathcal{S}^3$ in the model geometry of an end $\mathcal{E}$, may be extended trivially to all of $\mathbb{R}_+ \times \mathcal{S}^3$; the same notation $V$ will be used for this extended vector field. 

\begin{definition}\label{def-alf-mfd}  
A $U(1)$ action on an ALF manifold $(M^4, g)$ with designated end $\mathcal{E}$ is called \emph{almost free with respect to $\mathcal{E}$} if the following properties are satisfied.
\begin{enumerate}[(i)]
\item $U(1)$ acts on $(M^4,g)$ by isometries.

\item There are finitely many points $\{p_1, \ldots, p_k\}$ such that the action of $U(1)$ on $M^4\setminus\{p_1, \ldots, p_k\}$ is free. 

\item The $U(1)$ action asymptotes to that of the model geometry associated with $\mathcal{E}$. More precisely, if $T$ denotes the infinitesimal generator of the action on $(M^4, g)$ then
\begin{equation}\label{inf-gen-asy}
|\mathring{\nabla}^l (T-V)|_{g_0}=O(r^{-q-l}), \quad\quad l=0,1,2,3,
\end{equation}
where $q>1/2$ and $V$ are as in Definition \ref{ALF}. Here the notation $T$ is used without change for its pushforward to the model geometry via the asymptotic diffeomorphism.
\end{enumerate}
\end{definition}

Roughly speaking, condition $(\text{iii})$ is included to ensure that the quotient $\mathcal{E}/U(1)$ is an asymptotically Euclidean end, as will be shown in Proposition \ref{AEend-decay}. Furthermore, O'Neill's formula for the scalar curvature of Riemannian submersions shows that if the fibers are (scalar) flat, the difference of the base and ambient scalar curvatures is weakly nonnegative. Based on this observation, we are able to perform a reduction argument to the 
3-dimensional AE setting in order to obtain mass positivity in the presence of an almost free action.

\begin{theorem}\label{B} 
Let $(M^4, g)$ be a complete ALF manifold with nonnegative scalar curvature. If there is an almost free $U(1)$ action with respect to end $\mathcal{E}$, then the mass is nonnegative in this end. Moreover, the mass vanishes if and only if $(M^4,g)$ is isometric to $\mathbb{R}^3\times S^1$. 
\end{theorem}

\begin{remark}
As a consequence of the rigidity proof, we obtain a positive mass theorem for ALF gravitational instantons. More precisely, if the nonnegative scalar curvature hypothesis is replaced by the condition of Ricci flatness (or more generally nonnegative Ricci curvature with strong decay), then the conclusions of Theorem \ref{B} continue to hold as an immediate corollary of Lemma \ref{harmonic}. This should be compared with Minerbe's result \cite{minerbe}*{Theorem 1} for a different definition of mass and a weaker assumption on the decay of Ricci curvature.
\end{remark}

Generalizations that include positive lower bounds for the mass in terms of horizon area, charge, and angular momentum are well-known in the AE context, and are referred to as Penrose-type inequalities \cite{Mars}. With the next result we establish the first Penrose-type inequality in the ALF regime, providing a mass lower bound in terms of the degree of the asymptotic end. Heuristically, such a statement may be anticipated by recalling that the model metric $g_0$ of Definition \ref{ALF} has zero mass and nonpositive scalar curvature \eqref{decay-scal} depending on the bundle curvature. Thus, the degree of the end seemingly contributes positively to the mass by canceling negative effects from the scalar curvature. Although the constant $\ell/16$ appearing in this theorem is not sharp (see Remark \ref{notsharp}), we speculate that with an optimal constant rigidity should be obtained only for the (multi)-Taub-NUT manifolds.

\begin{theorem}\label{C} 
Let $(M^4,g)$ be a complete ALF manifold of nonnegative scalar curvature, with mass $m$ in end $\mathcal{E}$. If there is an almost free $U(1)$ action with respect to $\mathcal{E}$, then 
\begin{equation}\label{foainfinapfinpqinp}
m\geq \frac{\ell}{16} |\deg(\mathcal{E})|,
\end{equation}
where $2\pi\ell$ is the asymptotic length of the circle fibers.
\end{theorem}
 
This paper is organized as follows. In the next section the mass \eqref{massdef} is shown to be a geometric invariant, and other preliminary facts are recorded. The local structure of singularities for almost free $U(1)$ actions is studied in Section \ref{sec3}, while a density result for ALF manifolds is discussed in Section \ref{sec4}. Furthermore, Sections \ref{sec5} and \ref{sec6} are dedicated to the proof of Theorem \ref{B}, and the proof of Theorem \ref{C} is provided in Section \ref{sec7}. Lastly, the proof of Theorem \ref{A} is given in Sections \ref{sec8} and \ref{sec9}.


\section{Asymptotically Euclidean Quotients and Geometric Invariance of the Mass}
\label{sec2} \setcounter{equation}{0}
\setcounter{section}{2}

In this section we will collect several observations to be used later in the paper. These include showing that the quotient space of an ALF end is asymptotically Euclidean, that the mass \eqref{massdef} is a geometric invariant, and recording 
an alternate expression for the degree of an asymptotic end. 

\subsection{The degree of an asymptotic end.} Consider an ALF manifold $(M^4,g,\E)$ having an almost free $U(1)$ action with respect to a designated end $\E$. Let $T$ be the infinitesimal generator of the action, and denote its dual 1-form by $\eta=g(T,\cdot)$. Similarly, the dual 1-form for the infinitesimal generator $V$ of the action on the model end $(\mathbb{R}_+ \times \mathcal{S}^3,g_0)$ will be denoted by $\eta_0=g_0(V,\cdot)$. Note that since the $\eta_0 (V)=\ell^2$ the connection 1-form may be expressed as $\tau=\frac{\eta_0}{|\eta_0|^2}$, and thus the degree of the end becomes
\begin{equation}\label{degree-def2}
\mathrm{deg}(\mathcal{E})=\frac{1}{4\pi^2}\int_{\mathcal{S}^3_r}\frac{\eta_0}{|\eta_0|^2}\wedge d\left(\frac{\eta_0}{|\eta_0|^2}\right),
\end{equation}
for any $r$. Moreover, Definition \ref{def-alf-mfd} part $(\mathrm{iii})$ implies that
\begin{equation}\label{dual-form-asy}
|\mathring{\nabla}^l (\eta-\eta_0)|_{g_0}=O(r^{-q-l}), \quad\quad l=0,1,2.
\end{equation}
These properties allow for a computation of the degree in terms of $\eta$.

\begin{proposition}\label{degree-eta} Let $(M^4, g, \mathcal{E})$ be an ALF manifold with an almost free $U(1)$ action. Then the degree of the designated asymptotic end is given by 
\begin{equation}
\mathrm{deg}(\E)=\lim_{r\rightarrow\infty}\frac{1}{4\pi^2}\int_{\mathcal{S}^3_r}\frac{\eta}{|\eta|^2}\wedge d\left(\frac{\eta}{|\eta|^2}\right). 
\end{equation}
\end{proposition}

\begin{proof} 
Observe that basic manipulations yield
\begin{align}
\begin{split}
\frac{\eta}{|\eta|^2}\wedge d\left(\frac{\eta}{|\eta|^2}\right)&=\frac{\eta_0}{|\eta_0|^2}\wedge d\left(\frac{\eta_0}{|\eta_0|^2}\right)+\frac{\eta_0}{|\eta_0|^2}\wedge d\left(\frac{\eta}{|\eta|^2}-\frac{\eta_0}{|\eta_0|^2}\right)
+\left(\frac{\eta}{|\eta|^2}-\frac{\eta_0}{|\eta_0|^2}\right)\wedge d\left(\frac{\eta}{|\eta|^2}\right)\\
&=\frac{\eta_0}{|\eta_0|^2}\wedge d\left(\frac{\eta_0}{|\eta_0|^2}\right)-d\left[\frac{\eta_0}{|\eta_0|^2}\wedge\left(\frac{\eta}{|\eta|^2}-\frac{\eta_0}{|\eta_0|^2}\right)\right]\\
&+2d\left(\frac{\eta_0}{|\eta_0|^2}\right)\wedge\left(\frac{\eta}{|\eta|^2}-\frac{\eta_0}{|\eta_0|^2}\right)
+\left(\frac{\eta}{|\eta|^2}-\frac{\eta_0}{|\eta_0|^2}\right)\wedge d\left(\frac{\eta}{|\eta|^2}-\frac{\eta_0}{|\eta_0|^2}\right).
\end{split}
\end{align}
Moreover, the decay of \eqref{decay-scal} implies that $|d\eta_0|_{g_0}=O(r^{-2})$, and hence together with \eqref{dual-form-asy} we have
\begin{equation}
\Big|\frac{\eta}{|\eta|^2}-\frac{\eta_0}{|\eta_0|^2} \Big|_{g_0}=O(r^{-q}),\quad\quad
\Big|d\left(\frac{\eta}{|\eta|^2}-\frac{\eta_0}{|\eta_0|^2}\right) \Big|_{g_0}=O(r^{-q -1}).
\end{equation}
It follows that
\begin{equation}
\lim_{r\rightarrow\infty}\int_{\mathcal{S}^3_r}\frac{\eta}{|\eta|^2}\wedge d\left(\frac{\eta}{|\eta|^2}\right)
=\lim_{r\rightarrow\infty}\int_{\mathcal{S}^3_r}\frac{\eta_0}{|\eta_0|^2}\wedge d\left(\frac{\eta_0}{|\eta_0|^2}\right)
=4\pi^2 \deg(\E),
\end{equation}
where in the last equality we used \eqref{degree-def2}.
\end{proof}

\subsection{The quotient space of ALF ends} 
Consider the quotient map $\pi: M^4 \rightarrow \bar{M}^3:=M^4 /U(1)$ and the Riemannian quotient space $(\bar{M}^3,\bar{g})$ of the ALF manifold $(M^4, g)$, and observe that the metric may be expressed in Riemannian submersion format as
\begin{equation}\label{quo-form}
g=\pi^{*}\bar{g}+\frac{\eta^2}{|\eta|^2}.
\end{equation} 
Note that $\bar{M}^3$ is a manifold away from a finite number of points.
Moreover, the quotient of the model asymptotic end $(\mathbb{R}_{+}\times \mathcal{S}^3)/U(1)$ is diffeomorphic to the exterior of a ball in $\mathbb{R}^3$, and the model metric can also be written in the submersion context with a flat base 
\begin{equation}\label{model-case}
g_0=\pi_0^*g_{\bb{R}^3}+\frac{\eta_0^2}{|\eta_0|^2}, 
\end{equation}
where $\pi_0$ is the quotient map for the model end.
This suggests that the ALF manifold quotient is asymptotically Euclidean with respect to the designated end.


\begin{proposition}\label{AEend-decay}
Let $(M^4,g,\E)$ be a complete ALF manifold with an almost free $U(1)$ action and a designated end $\E$. Then there exists a subset $\E' \subset\E$ such that $\E'/U(1)\subset\bar{M}^3$ is diffeomorphic to $\mathbb{R}^3 \setminus B_{r_1}$ for some $r_1 >0$, and in the coordinates $\bar{x}=(\bar{x}^1,\bar{x}^2,\bar{x}^3)$ provided by this diffeomorphism the quotient metric satisfies the decay
\begin{equation}\label{AE-metric-decay}
|\partial^l (\bar{g}_{ij}-\delta_{ij})(\bar{x})|=O(|\bar{x}|^{-q-l}),\quad\quad\quad l=0,1,2,
\end{equation}
where $q > 1/2$. In particular, $(\bar{M}^3,\bar{g})$ is asymptotically Euclidean with respect to the designated end.
\end{proposition} 

\begin{remark}\label{Prop2.2} 
The notion of an asymptotically Euclidean end sometimes includes not only \eqref{AE-metric-decay}, but also a requirement that the scalar curvature $R_{\bar{g}}$ is integrable. Although this latter condition may not be valid in the current setting, we will nevertheless use the AE terminology. When $R_{\bar{g}}$ is integrable, the ADM mass of the designated AE end is well-defined and given by
\begin{equation}\label{massaeend}
\bar{m}=\lim_{\bar{r}\rightarrow\infty}\frac{1}{16\pi}\int_{S_{\bar{r}}}(\bar{g}_{ij,i}-\bar{g}_{ii,j})\bar{\nu}^j ,
\end{equation}
where $S_{\bar{r}}$ is a coordinate sphere with unit outer normal $\bar{\nu}$, and $\bar{r}=|\bar{x}|$.
\end{remark}

\begin{proof}
By pushing the boundary of the end away from any fixed points if necessary, we may assume that the $U(1)$ action is free on $\E$ so that $\E/U(1)$ is a manifold. Consider the local coordinates $(x,\theta)$ on $\E$ as in \eqref{connection-form-metric}, and note that the functions $x$ are globally defined and yield a corresponding map $x=(x^1,x^2,x^3): \E \rightarrow \mathbb{R}^3$. Although these functions descend to the quotient in the context of the model end, they do not necessarily have this property on $\E$. However, new related functions will be constructed which do descend and yield the desired diffeomorphism. To do this, we will make use of the flow $\varphi_{t}$ for the Killing field $T$ that generates the $U(1)$ action. Recall that for $\mathbf{x}\in \E$ the following properties are satisfied
\begin{equation}\label{FF}
\varphi_0 (\mathbf{x})=\mathbf{x},\quad\quad \varphi_{t+s}(\mathbf{x})=\varphi_t(\varphi_s(\mathbf{x})), \quad\quad \partial_t \varphi_t (\mathbf{x})=T(\varphi_t (\mathbf{x})).
\end{equation}

Let $\mathcal{E}_{r_0}\subset\E$ denote the portion of the end for which $|x|>r_0$. 
It will be assumed that $r_0$ is sufficiently large, so that for each $t\in\mathbb{R}$ we have $\varphi_t :\E_{r_0} \rightarrow \E$. In local coordinates we may write 
\begin{equation}
\varphi_t(\mathbf{x})=(y^1(t,\mathbf{x}),y^2(t,\mathbf{x}),y^3(t,\mathbf{x}),\vartheta(t,\mathbf{x})),
\end{equation}
and observe that if $\mathbf{x}=(x,\theta)$ then
\begin{equation}
y^i(t,\mathbf{x}) -x^i =\int_{0}^{t} \partial_s y^i(s,\mathbf{x}) ds, \quad\quad i=1,2,3,
\end{equation}
and hence
\begin{equation}
\frac{\partial y^i}{\partial x^j}-\delta^i_j =\int_{0}^{t}\partial_{x^j} T^i (s,\mathbf{x})ds, \quad\quad i,j=1,2,3,
\end{equation}
where $T^i =T(dx^i)$ are components of $T$ in local coordinates. Notice that $V=\partial_{\theta}$ and thus
\begin{equation}\label{aonfoinh}
\partial_{x^j} T^i =\partial_{x^j}\left(T -V\right)^i =\mathring{\nabla}_j\left(T-V\right)^i -\mathring{\Gamma}_{ja}^i \left(T-V\right)^a ,
\end{equation}
where $\mathring{\Gamma}_{ja}^i$ are Christoffel symbols for $g_0$. Next observe that by comparing the expressions in \eqref{def-model} and \eqref{connection-form-metric} we find $\partial^l A_i=O(r^{-1-l})$ for all $l$, which implies that $\partial^l \mathring{\Gamma}_{ja}^i=O(r^{-2-l})$. Here $\partial^l$ represents derivatives with respect to $x$ of order $l$, and below the notation $O_k$ represents corresponding fall-off for derivatives up to order $k$. Combining these observations with \eqref{inf-gen-asy} and \eqref{aonfoinh} produces $\partial_{x^j} T^i =O_2(r^{-q-1})$, and thus
\begin{equation}\label{fpoapoiniohh}
\frac{\partial y^i}{\partial x^j}=\delta^i_j +O_2(r^{-q-1}).
\end{equation}

Preliminary coordinate functions will be obtained by averaging along the flow of $T$. Since $T$ generates a $U(1)$ action,  a complete orbit is achieved after the flow parameter passes $2\pi$. 
We then define functions on $\E_{r_0}$ by
\begin{equation}
\tilde{x}^i(\mathbf{x})=\frac{1}{2\pi}\int_{0}^{2\pi}y^i (t,\mathbf{x})dt, \quad\quad i=1,2,3,
\end{equation}
and note that \eqref{fpoapoiniohh} implies that
\begin{equation}\label{fpoapoiniohh1}
\frac{\partial \tilde{x}^i}{\partial x^j}=\delta^i_j +O_2(r^{-q}).
\end{equation}
The collection $(\tilde{x},\theta)$ then form local coordinates on $\E_{r_0}$, for $r_0$ sufficiently large.
Moreover, by \eqref{FF} we have $y^i(t,\varphi_{s}(\mathbf{x}))=y^i(t+s,\mathbf{x})$, from which it follows that the $\tilde{x}^i$ are constant on $T$-orbits, namely $T(\tilde{x}^i)=0$ for $i=1,2,3$. Therefore, these descend to smooth functions on the orbit space $\bar{x}=(\bar{x}^1 ,\bar{x}^2, \bar{x}^3):\E_{r_0}/U(1)\rightarrow \mathbb{R}^3$,
and estimate \eqref{fpoapoiniohh1} implies 
\begin{equation}\label{aonfoiniknhj}
\pi^* \bar{x}^i = x^i +O_2(r^{1-q}). 
\end{equation}
Let $\E'_{r_1}\subset\E$ be such that $\E'_{r_1}/U(1)=\bar{x}^{-1}(\mathbb{R}^3 \setminus B_{r_1})$. We will show that 
\begin{equation}\label{goanoinijjj}
\bar{x}: \E'_{r_1}/U(1)\rightarrow \mathbb{R}^3 \setminus B_{r_1}
\end{equation}
is a diffeomorphism for $r_1$ sufficiently large.

To show that the map of \eqref{goanoinijjj} is surjective let $\bar{x}_0 \in\mathbb{R}^3 \setminus B_{r_1}$, and observe that for fixed $\theta_0$ we may solve $\bar{x}(\pi(x_0,\theta_0))=\bar{x}_0$ for $x_0$ close to $\bar{x}_0$ if $r_1$ is large, by the inverse function theorem and \eqref{aonfoiniknhj}. The map is also a local diffeomorphism for large $r_1$, as can be seen from similar arguments used to show that $(\tilde{x},\theta)$ form local coordinates. Moreover, the map is proper since if $\mathcal{C}\subset\mathbb{R}^3 \setminus B_{r_1}$ is compact then $\mathcal{C}':=(\pi^* \bar{x})^{-1}(\mathcal{C})\subset\E$ is compact, as this set is clearly closed and is bounded by \eqref{fpoapoiniohh1}. Since $\pi$ is continuous we have that $\pi(\mathcal{C}')=\bar{x}^{-1}(\mathcal{C})\subset\E'_{r_1}/U(1)$ is compact. Thus, $\bar{x}$ of \eqref{goanoinijjj} is a covering map, and since $\mathbb{R}^3 \setminus B_{r_1}$ is simply connected it follows that it is injective. In particular, it is a diffeomorphism.

We will write $\E'$ instead of $\E'_{r_1}$ for convenience. Consider the quotient metric $\bar{g}$ on $\E'/U(1)$, which in the $\bar{x}$-coordinates is given by
\begin{equation}
\bar{g}_{ij}=\bar{g}(\partial_{\bar{x}^i},\partial_{\bar{x}^j})=g(\partial_{\tilde{x}^i},\partial_{\tilde{x}^j})
-\frac{\eta(\partial_{\tilde{x}^i})\eta(\partial_{\tilde{x}^j})}{|\eta|^2}.
\end{equation}
According to \eqref{fpoapoiniohh1}, changing to $(x,\theta)$ coordinates produces
\begin{equation}
\tilde{g}_{ij}:=g(\partial_{\tilde{x}^i},\partial_{\tilde{x}^j})=g(\partial_{x^i},\partial_{x^j})+O_2(r^{-q})=\delta_{ij}+O_2(r^{-q}).
\end{equation}
Furthermore observe that
\begin{equation}
\eta(\partial_{\tilde{x}^i})=\eta\left(\partial_{\tilde{x}^i}-\nabla\tilde{x}^i\right)
=\eta\left((\tilde{g}_{ij}-\delta_{ij})\nabla\tilde{x}^j\right)=O_2 (r^{-q}),
\end{equation}
where we have used
\begin{equation}
\eta(\nabla \tilde{x}^i)=g(T,\nabla\tilde{x}^i)=0,\quad\quad\quad
\tilde{g}_{ij}\nabla\tilde{x}^j =\tilde{g}_{ij}\left(\tilde{g}^{ab}\partial_{\tilde{x}^b}\tilde{x}^j \right)\partial_{\tilde{x}^a}=\partial_{\tilde{x}^i}.
\end{equation}
Combining these facts with $r=|x|=|\bar{x}|+O_2(|\bar{x}|^{1-q})$ yields the desired decay \eqref{AE-metric-decay}.
\end{proof}

\subsection{Geometric invariance of the mass}
The ADM mass of an asymptotically Euclidean end is well-known (\cites{bartnik1986,Chrusciel}) to be independent of the choice of asymptotic coordinates as well as the choice of limiting surfaces, and is thus well-defined. An analogous statement has been obtained by Minerbe \cite{minerbe}*{Proposition 6} for the so called `Gauss-Bonnet mass' (tailored to Ricci curvature) in the ALF case, and was claimed  \cite{minerbe}*{pg. 952} in the setting of AF asymptotics for mass \eqref{massdef}. In this subsection we will establish such a result for mass \eqref{massdef} in the ALF setting.

\begin{definition} 
An ALF end $(\E, g)$ is said to have \textit{ALF structure} $(\Psi,\mathcal{S}^3, \pi_{s},\tau,\ell)$ if there is a subset $\E_\Psi\subset\E$ such that $\E\setminus \E_\Psi$ is precompact and $\Psi:[r_0, \infty)\times \mathcal{S}^3 \rightarrow \E_{\Psi}$ is a diffeomorphism for some $r_0 >0$. Additionally, if $\pi_s: \mathcal{S}^3\rightarrow S^2$ is the projection map for principal $U(1)$ bundle $\mathcal{S}^3$ and $\pi_s^* g_{S^2}$ is extended trivially in the radial direction, then the model metric on $[r_0, \infty)\times \mathcal{S}^3$ is given by
\begin{equation}\label{model}
g_0=dr^2+r^2\pi^*_s g_{S^2}+\ell^2\tau^{2},
\end{equation}
and serves as the decay limit for the pullback metric
\begin{equation}\label{def-ALF-str-decay}
\Psi^* g-g_0=O_2(r^{-q}),
\end{equation}
where $q>1/2$. 
\end{definition}

Let $d:\E \rightarrow \mathbb{R}_+$ denote the distance function from $\partial\E_{\Psi}$. Due to the fall-off 
of \eqref{def-ALF-str-decay}, we observe that this function is comparable to the radial coordinate $r$.

\begin{lemma}\label{dist} 
Let $(\Psi, \mathcal{S}^3, \pi_s, \tau, \ell)$ be an ALF structure on $(\E, g)$, and consider the distance function $d$ from $\partial\E_\Psi$. There is a constant $c>0$ such that on $\E_{\Psi}$ the following inequality holds
\begin{equation}\label{dist-comparison}
\frac{d}{2}- c \leq r\leq 2d+c.
\end{equation}
\end{lemma}

\begin{proof}
Let $\mathbf{x}=\Psi(r_1,v_1)$ be a point of $\E_{\Psi}$, and consider a curve $\alpha(t)=\Psi(t,v_1)$ connecting $\partial\E_{\Psi}$ to $\mathbf{x}$, where $t\in[r_0,r_1]$. Using \eqref{def-ALF-str-decay} we then have
\begin{equation}\label{upper-dist}
d(\mathbf{x})\leq \int^{r_1}_{r_0} \sqrt{ g(\dot{\alpha}, \dot{\alpha})}dt \leq \int^{r_1}_{r_0} \sqrt{ g_0(\partial_r, \partial_r) + c_0{t^{-q}}} dt\leq 2r_1 +c_1=2r(\mathbf{x}) +c_1,
\end{equation}
for some positive constants $c_0$ and $c_1$. Next consider a minimizing geodesic $\gamma(t)=\Psi(\gamma_0(t))$
connecting $\partial\E_{\Psi}$ to $\mathbf{x}$, which is parameterized by arclength and where $\gamma_0(t)=(r(t),v(t))$. According to \eqref{def-ALF-str-decay} it follows that
\begin{equation}
1=|\dot{\gamma}(t)|_{g}^2 \geq (1-c_1 r(t)^{-q})|\dot{\gamma}_0(t)|^2_{g_0}\geq (1-c_0 r(t)^{-q})\dot{r}(t)^2\geq \frac{\dot{r}(t)^2}{4},
\end{equation}
for $r(t)$ sufficiently large. We then have
\begin{equation}
r(\mathbf{x})-r_0 =\int_{0}^{d(\mathbf{x})}\dot{r}(t)dt \leq 2 d(\mathbf{x}) +c_1.
\end{equation}
\end{proof}

 
Although a given ALF end will have many different ALF structures, they must satisfy strong relations which we now discuss. Let $\mathcal{F}=(\Psi, \mathcal{S}^3, \pi_s, \tau,\ell)$ and $\mathcal{F}'=(\Psi', \mathcal{S}'^{3}, \pi'_s, \tau', \ell')$ be two ALF structures on a single ALF end $(\E, g)$. Since $\E\setminus \E_{\Psi}$ is precompact, there are two constant $r_1$ and $r'_1$ such that  $\Psi([r_1, \infty)\times \mathcal{S}^3)\subset \Psi'([r'_1, \infty)\times \mathcal{S}'^3)\subset \E_{\Psi}$, which gives the following homomorphisms between fundamental groups induced by the inclusion maps
\begin{equation}
\pmb{\pi_1}(\Psi([r_1, \infty)\times \mathcal{S}^3))\rightarrow \pmb{\pi_1}(\Psi'([r'_1,\infty)\times \mathcal{S}'^3))\rightarrow \pmb{\pi_1}(\E_\Psi)
\end{equation}
The composition of these two homomorphisms is the identity, and thus there is an injection $\pmb{\pi_1}(\mathcal{S}^3)\rightarrow \pmb{\pi_1}(\mathcal{S}'^3)$ along with a surjection $\pmb{\pi_1}(\mathcal{S}'^3)\rightarrow \pmb{\pi_1}(\mathcal{S}^3)$. By reversing the roles of $\mathcal{S}^3$ and $\mathcal{S}'^3$, it follows that $\pmb{\pi_1}(\mathcal{S}^3)$ and $\pmb{\pi_1}(\mathcal{S}'^3)$ are isomorphic. Therefore, since the total space of a principal $U(1)$ bundle over $S^2$ is completely determined by its fundamental group \cite{Chern}, we find that $\mathcal{S}^3$ and $\mathcal{S}'^3$ are diffeomorphic. Next note that by pushing the boundary $\partial \E_{\Psi'}$ away from $\E\setminus \E_\Psi$, it may be assumed that $\E_{\Psi'} \subset \E_{\Psi}\subset \E$. The end $\E_{\Psi'}$ can then be equipped with two model metrics from the two ALF structures, namely
\begin{equation}\label{2model}
g_0=dr^2+ r^2 \pi_s^*g_{S^2}+ \ell^2 \tau^2 ,\quad\quad\quad
g'_0={dr'}^2+ {r'}^2 {\pi'}_s^* g_{S^2}+ {\ell'}^2 {\tau'}^2 .
\end{equation} 
Furthermore, observe that Lemma \ref{dist} together with the fact that both $g_0$ and $g'_0$ asymptote to $g$ implies the following decay between the two model metrics
\begin{equation}\label{model-decay}
|\mathring{\nabla}^l\left((\Psi'^{-1}\circ \Psi)^*g'_0-g_0\right)|_{g_0}=O(r^{-q_0-l}), \qquad l=0,1,2,
\end{equation}
where $\mathring{\nabla}$ is the Levi-Civita connection of $g_0$ and $q_0=\min\{q,q'\}$.  


\begin{lemma}\label{vector-asy} 
Let $V$ and $V'$ be the infinitesimal generators of the $U(1)$ action for the model geometries of the two ALF structures $\mathcal{F}$ and $\mathcal{F}'$. If $\ell$ and $\ell'$ denote the respective lengths of these two vectors with respect to  $g_0$ and $g'_0$ then
\begin{equation}\label{vector-model-decay}
\ell=\ell', \quad\quad\quad |\mathring{\nabla}^l \left((\Psi^{-1}\circ \Psi')_{*}V' -V\right)|_{g_0}=O(r^{-q_1 -l}), 
\end{equation}
where $q_1=\min\{1, q_0\}$ and $l=0,1,2$.
\end{lemma}

\begin{proof}
We will first establish a weak version of \eqref{vector-model-decay} for $l=0$.
Assume by way of contradiction that there exists a sequence $\{\mathbf{p}_k \}_{k=1}^{\infty} \subset \Psi^{-1}(\E_{\Psi'})$ such that $r(\mathbf{p}_k)\rightarrow\infty$ along with 
\begin{equation}\label{aofjoianihh}
|(\Psi^{-1}\circ \Psi')_{*}V' -V|_{g_0}(\mathbf{p}_k)\geq c>0,
\end{equation}
for some constant $c$ and all $k$.
Recall that the model end may be viewed as a $U(1)$ bundle over a Euclidean base $\pi_0 :[r_0,\infty)\times\mathcal{S}^3 \rightarrow \mathbb{R}^3 \setminus B_{r_0}$, and the metric $g_0$ may be expressed by \eqref{connection-form-metric} in this context. Let $\mathbf{B}_{r_k}(\mathbf{p}_k)$ be the preimage under $\pi_0$ of the Euclidean ball of radius $r_k=\frac{1}{2}\left(r(\mathbf{p}_k)-r_0\right)$ centered at $\pi_0(\mathbf{p}_k)$. As in the proof of Proposition \ref{AEend-decay} we have decay of the metric coefficients $|\partial^l A_i|=O(r^{-1-l})$ for all $l$, and thus
\begin{equation}\label{aoofinoinghhhh}
\left(\mathbf{B}_{r_k}(\mathbf{p}_k),g_0,\mathbf{p}_k \right)\rightarrow\left(\mathbb{R}^3 \times S^1 , \delta +\ell^2 d\vartheta^2,\mathbf{p}_{\infty}\right)
\end{equation}
in the pointed $C^\infty$-topology, where $\mathbf{p}_{\infty}$ is a designated point of the limit manifold. Moreover, through the pushforward by diffeomorphisms relating the sequence and limit manifolds we find that (with an abuse of notation) $V(\mathbf{p}_k)\rightarrow \partial_{\vartheta}$. Similarly, if $\mathbf{p}'_k =\Psi'^{-1} \circ\Psi(\mathbf{p}_k)$ then in the model geometry for ALF structure $\mathcal{F}'$ it follows that
\begin{equation}\label{foioigjioqnigh}
\left(\mathbf{B}'_{1}(\mathbf{p}'_k),g'_0,\mathbf{p}'_k \right)\rightarrow\left(B_1 \times S^1 , \delta +\ell'^2 d\vartheta'^2,\mathbf{p}'_{\infty}\right),
\end{equation}
and $V'(\mathbf{p}'_k)\rightarrow \partial_{\vartheta'}$. On the other hand, by \eqref{model-decay} we find that
\begin{equation}\label{foioigjioqnigh1}
\left((\Psi'^{-1} \circ\Psi)^{-1}(\mathbf{B}'_1(\mathbf{p}'_k)),(\Psi'^{-1} \circ\Psi)^{*} g'_0,\mathbf{p}_k\right)
\rightarrow\left(B_1 \times S^1 , \delta +\ell^2 d\vartheta^2,\mathbf{p}_{\infty}\right).
\end{equation}
These last two limits must be isometric, and therefore $\ell=\ell'$; note that this conclusion does not rely on the hypothesis \eqref{aofjoianihh}, and thus it holds independently. 

Consider the sequence of vectors $W_k =(\Psi^{-1}\circ \Psi')_{*}V'(\mathbf{p}'_k)$, and observe that \eqref{model-decay} shows this forms a Cauchy sequence in $T_{\mathbf{p}_{\infty}}(\mathbb{R}^3 \times S^1)$ through the pushforward by diffeomorphisms associated with the convergence \eqref{aoofinoinghhhh}. Hence, we may write $W_k \rightarrow W_{\infty}$ with $|W_{\infty}|=\ell$, and additionally $|W_{\infty}-\partial_{\vartheta}|\geq c$ in light of \eqref{aofjoianihh}. Using the isometry $\Psi^{-1}\circ \Psi'$ between the sequences of \eqref{foioigjioqnigh} and \eqref{foioigjioqnigh1}, it follows that
the geodesics $\gamma_k(t)=\exp_{\mathbf{p}'_k}(tV'(\mathbf{p}'_k))$, $t\in[0,2\pi]$ within $\mathbf{B}'_{1}(\mathbf{p}'_k)$ remain uniformly away from being closed. However, due to \eqref{foioigjioqnigh} these curves asymptote to a closed geodesic
as $k\rightarrow\infty$, yielding a contradiction. We conclude that 
\begin{equation}\label{ofioaininhhhh}
\lim_{r\rightarrow\infty}|(\Psi^{-1}\circ \Psi')_{*}V' -V|_{g_0}=0.
\end{equation}

We will now obtain the full statement of \eqref{vector-model-decay}. The Christoffel symbol decay from the proof of Proposition \ref{AEend-decay} implies that $|\mathring{\nabla}^l V|_{g_0}=O(r^{-1-l})$ for $l\geq 1$, with a corresponding estimate for the derivatives of $V'$. Furthermore, if $\psi:=\Psi'^{-1}\circ\Psi$ then \eqref{model-decay} produces a comparison between the model and pullback connections
\begin{equation}
|(\mathring{\nabla}^l-\psi^* \mathring{\nabla}'^l ) \psi^{-1}_{*} V'|_{g_0}=|((\psi^{-1})^* \mathring{\nabla}^l-\mathring{\nabla}'^l ) V'|_{(\psi^{-1})^* g_0}=O(r^{-q_0 -l}),\quad\quad l=1,2.
\end{equation}
It follows that
\begin{equation}\label{onainfoinaoihh}
|\mathring{\nabla}^l(V-\psi^{-1}_* V')|_{g_0}\leq |\mathring{\nabla}^l V|_{g_0}
+|(\mathring{\nabla}^l-\psi^* \mathring{\nabla}'^l ) \psi^{-1}_{*} V'|_{g_0}
+|(\psi^* \mathring{\nabla}'^l ) \psi^{-1}_{*} V'|_{g_0}=O(r^{-q_1 -l}),
\end{equation}
where $q_1 =\min\{1,q_0\}$. Let $\mathbf{p}=(r(\mathbf{p}),v(\mathbf{p}))\in \Psi^{-1}(\E_{\Psi'})$ and consider a curve $\alpha(t)=(t,v(\mathbf{p}))$, $t\in[r(\mathbf{p}),\infty)$ connecting $\mathbf{p}$ to infinity. Then applying \eqref{ofioaininhhhh} and \eqref{onainfoinaoihh} produces
\begin{align}\label{0r8y30q498hi}
\begin{split}
|V-\psi^{-1}_{*}V'|_{g_0}(\mathbf{p})
&=-\int_{r(\mathbf{p})}^{\infty}\frac{d}{dt}|V-\psi^{-1}_{*}V'|_{g_0}(\alpha(t))dt\\
&\leq \int_{r(\mathbf{p})}^{\infty}\ell|\mathring{\nabla}(V-\psi^{-1}_{*}V')|_{g_0}(\alpha(t))dt\\
&\leq Cr(\mathbf{p})^{-q_1},
\end{split}
\end{align}
for some constant $C$. Together, \eqref{onainfoinaoihh} and \eqref{0r8y30q498hi} give the desired result.
\end{proof}

Let $(x,\theta)$ denote the local coordinates on $\Psi^{-1}(\E_{\Psi'})$ in which the model metric $g_{0}$ takes the form \eqref{connection-form-metric}. A $g_0$-orthonormal frame on this model end is then given by
\begin{equation}\label{def-frame}
X_i=\partial_{x^i}-A_{i}\partial_{\theta},\quad i=1,2,3,\quad\quad X_4 =\ell^{-1}V.
\end{equation}
In an analogous way, we may construct a $g'_0$-orthonormal frame $\{X'_1, X'_2, X'_3, X'_4\}$ on $\Psi'^{-1}(\E_{\Psi'})$. If $\psi=\Psi'^{-1}\circ \Psi$, then the two sets of frames are related by a transition matrix $Q=(Q_a^b)$ such that
\begin{equation}\label{orthnormal-relation}
\psi^{-1}_*X'_a=Q^b_{a}(\mathbf{p})X_b ,\quad\quad\quad Q^b_a(\mathbf{p})= g_0(\psi^{-1}_*X'_a, X_b),
\end{equation} 
for $\mathbf{p}\in \Psi^{-1}(\E_{\Psi'})$.
Furthermore, Lemma \ref{vector-asy} and \eqref{model-decay} imply that 
\begin{equation}\label{formula-vector-asy}
Q_a^b-\delta_{a}^{b}=O_2(r^{-q_1}) \quad \text{ for  } a=4 \text{ or } b=4.
\end{equation}
In general, the transition matrix asymptotes to an orthogonal matrix.

\begin{lemma}\label{orthnormal} 
There exists a constant $4\times 4$ orthogonal matrix $\mathbf{O}=(\mathbf{O}^b_a)$ such that
\begin{equation}\label{transf-estimate}
Q_a^b-\mathbf{O}_a^b=O_2(r^{-q_1}),
\end{equation}
with $\mathbf{O}_a^{b}=\delta_{a}^{b}$ for $a=4$ or $b=4$. 
\end{lemma}

\begin{proof} 
We will first establish the estimate for derivatives in \eqref{transf-estimate}. Recall that the Christoffel symbol decay from the proof of Proposition \ref{AEend-decay} implies that $|\mathring{\nabla}^l X_a|_{g_0}=O(r^{-1-l})$ for $l\geq 1$, with a corresponding estimate for the derivatives of $X'_a$. Observe that \eqref{model-decay} yields a comparison between the model and pullback connections to produce
\begin{equation}
|(\mathring{\nabla}^l-\psi^* \mathring{\nabla}'^l ) \psi^{-1}_*X'_a|_{g_0}=|((\psi^{-1})^* \mathring{\nabla}^l-\mathring{\nabla}'^l ) X'_a|_{(\psi^{-1})^* g_0}=O(r^{-q_0 -l}),\quad\quad l=1,2.
\end{equation}
It follows that 
\begin{equation}
|\mathring{\nabla}^l \psi^{-1}_* X'_a|_{g_0}\leq 
|(\mathring{\nabla}^l-\psi^* \mathring{\nabla}'^l ) \psi^{-1}_{*} X'_a|_{g_0}
+|(\psi^* \mathring{\nabla}'^l ) \psi^{-1}_{*} X'_a|_{g_0}=O(r^{-q_1 -l}),
\end{equation}
and hence for $l=1,2$ we have
\begin{equation}\label{transfer-decay}
|\mathring{\nabla}^l Q^b_a|_{g_0}\leq \sum^l_{k=0}2|\mathring{\nabla}^{l-k} \psi^{-1}_*X'_a|_{g_0}|\mathring{\nabla}^k X_b|_{g_0}=O(r^{-q_1 -l}).
\end{equation}

We will now address the existence of an asymptotic limit for $Q$. Let $\mathbf{p} =(r,v)$ and $\bar{\mathbf{p}}=(\bar{r},\bar{v})$ be points of $\Psi^{-1}(\E_{\Psi'})$ with $\bar{r}\geq r$. Consider the curve $\alpha(t)=(t,v)$, $t\in[r,\bar{r}]$ connecting $\mathbf{p}$ to $\tilde{\mathbf{p}}=(\bar{r},v)$, and a minimizing geodesic $\gamma\subset \mathcal{S}_{\bar{r}}^3$ parameterized by arclength which connects $\tilde{\mathbf{p}}$ to $\bar{\mathbf{p}}$; note that the length of this geodesic satisfies $L=O(\bar{r})$.
With the help of \eqref{transfer-decay} we then find
\begin{align}\label{f0-9qu09hjy}
\begin{split}
|Q^b_a(\mathbf{p})-Q^b_a(\bar{\mathbf{p}})|&\leq|Q^b_a(\mathbf{p})-Q^b_a(\tilde{\mathbf{p}})|+|Q^b_a(\tilde{\mathbf{p}})-Q^b_a(\bar{\mathbf{p}})|\\
&\leq \int_{r}^{\bar{r}}\ell|\mathring{\nabla} Q^b_a|_{g_0}(\alpha(t))dt 
+\int_{0}^{L} |\mathring{\nabla} Q^b_a|_{g_0}(\gamma(s))ds \\
&\leq Cr^{-q_1},
\end{split}
\end{align}
for some constant $C$. It follows that for each $a$, $b$ the 1-parameter family of functions $Q_a^b (r,\cdot)$ is Cauchy in $C^0(\mathcal{S}^3)$, and thus $\lim_{r\rightarrow\infty}Q_a^b (r,\cdot) =: \mathbf{O}_a^b$ exists for some $\mathbf{O}_a^b \in C^0(\mathcal{S}^3)$. Moreover, \eqref{f0-9qu09hjy} also implies that the matrix $\mathbf{O}=(\mathbf{O}_a^b)$ is constant on $\mathcal{S}^3$.

To complete the proof, notice that taking the limit as $\bar{r}\rightarrow\infty$ in \eqref{f0-9qu09hjy} produces
\begin{equation}
|Q_a^b(\mathbf{p}) -\mathbf{O}_a^b|\leq Cr^{-q_1}. 
\end{equation}
This, combined with \eqref{formula-vector-asy} and \eqref{transfer-decay} yields \eqref{transf-estimate}, along with the special values for the limit when $a=4$ or $b=4$.  
Moreover, \eqref{model-decay} and \eqref{orthnormal-relation} imply 
\begin{equation} 
\delta_{a}^b=\lim_{r\rightarrow \infty} g_0(\psi^{-1}_*X'_a, \psi^{-1}_*X'_b)=\lim_{r\rightarrow\infty} \sum_{n=1}^4 Q^n_aQ^n_b= \sum_{n=1}^{4} \mathbf{O}^n_a \mathbf{O}^n_b,
\end{equation}
showing that $\mathbf{O}$ is an orthogonal matrix. 
\end{proof}

We are now able to establish the main result of this section.

\begin{proposition}\label{mass-well-def} 
The mass \eqref{massdef} of an ALF end $(\E,g)$ is well-defined and independent of the choice of ALF structure. 
\end{proposition}

\begin{proof} 
We will first address the existence of the limit in \eqref{massdef} for a given ALF structure $\mathcal{F}=(\Psi, \mathcal{S}^3, \pi_s, \tau,\ell)$. Let $\{X_1, X_2, X_3, X_4\}$ be the $g_0$-orthonormal frame on $\Psi^{-1}(\E_{\Psi})$ given by \eqref{def-frame}, and denote components of the pullback metric by $g_{ab}=\Psi^{*}g(X_a , X_b)$. Then according to \cite{bartnik1986}*{(4.2) and (4.7)}, the scalar curvature may be expressed as
\begin{equation}\label{Scal-exp}
R_{\Psi^*g}=|\det\Psi^*g|^{-\frac{1}{2}} \partial_b (g_{ab, a}-g_{aa, b})+ O(r^{-2-2q}),\quad  R_{\Psi^*g} \star_{g_0} 1=d( g_{ab}\omega^a_c\wedge \eta^{cb})+O(r^{-2-2q}),
\end{equation} 
where $\omega^a_c$ are the Levi-Civita connection $1$-forms of $\Psi^*g$ with respect to the chosen frame, and the remaining 2-forms are given by $\eta^{ab}=\star_{g_0}(X^a\wedge X^b)$ in which an upper index is used to indicate the dual coframe. 
By setting
\begin{equation}
Z=g_{ab}\omega^a_c\wedge \eta^{cb}
\end{equation}
we then have
\begin{equation}\label{foaibfoiuabifo}
R_{\Psi^*g}\star_{g_0}1=dZ+O(r^{-2-2q}),\quad\quad
Z=\star_{g_0} (\text{div}_{g_0}(\Psi^*g)-d \text{tr}_{g_0}(\Psi^*g))+O(r^{-1-2q}),
\end{equation}
and thus for any $r>\bar{r}>r_0$ it follows that
\begin{equation}
\int_{\mathcal{S}^3_r} Z-\int_{\mathcal{S}^3_{\bar{r}}} Z=\int_{\mathbf{B}_r\setminus \mathbf{B}_{\bar{r}}} dZ\\
=\int_{\mathbf{B}_r\setminus \mathbf{B}_{\bar{r}}}\left(R_{\Psi^*g}+O(r^{-2-2q})\right)\star_{g_0}1,
\end{equation}
where the boundary surfaces are oriented with respect to the unit normal pointing towards infinity and $\mathbf{B}_r =\pi_0^{-1}(B_r)$ as in the proof of Lemma \ref{vector-asy}. Since $R_g\in L^1(\E)$ and $q>1/2$, we conclude that the limit
\begin{equation}\label{mass-formula}
\lim\limits_{r\rightarrow \infty}\int_{\mathcal{S}^3_r} \star_{g_0}(\text{div}_{g_0}(\Psi^*g)-d \text{tr}_{g_0}(\Psi^*g) ) 
\end{equation} 
exists and is finite, showing that the mass is well-defined for each ALF structure.

We will now establish the independence of the mass with respect to ALF structure. Consider any other ALF structure $\mathcal{F}'=(\Psi', \mathcal{S}'^{3}, \pi'_s, \tau', \ell')$, with $g'_0$-orthonormal frame $\{X'_a\}_{a=1}^4$ and coframe $\{X'^a\}_{a=1}^4$ on $\Psi'^{-1}(\E_{\Psi'})$. If $\psi=\Psi'^{-1}\circ \Psi$ is the diffeomorphism relating the two ALF structures, then according to Lemma \ref{orthnormal} there exists a transition matrix $Q$ relating the frames and coframes that asymptotes to a constant orthogonal matrix $\mathbf{O}$ with the following decay
\begin{equation}\label{0hqj309rjh9qihh}
\psi_{*}^{-1}X'_a =Q_{a}^b X_b,\quad\quad \psi^* X'^a=(Q^{-1})^a_b X^b,\quad\quad
Q_a^b -\mathbf{O}_a^b =O_2(r^{-q_1}).
\end{equation}
Furthermore, on $\Psi'^{-1}(\E_{\Psi'})$ we may construct the relevant differential forms
\begin{equation}
\eta'^{ab}=\star_{g'_0}(X'^a\wedge X'^b),\quad\quad\quad
Z'=g'_{ab}\omega'^a_c\wedge \eta'^{cb},
\end{equation}
where $\omega'^a_c$ is the Levi-Civita connection 1-form for ${\Psi'}^*g$ with respect to the frame $\{X'_a\}_{a=1}^4$ and $g'_{ab}=(\Psi'^*g)(X'_a,X'_b)$, so that
\begin{equation}
Z'=\star_{g'_0} (\text{div}_{g'_0}(\Psi'^*g)-d \text{tr}_{g'_0}(\Psi'^*g))+O(r'^{-1-2q'}).
\end{equation}
Next, observe that a computation utilizing \eqref{model-decay}, \eqref{0hqj309rjh9qihh}, and Lemma \ref{dist} 
yields
\begin{equation}
\psi^* Z' -Z=dQ_a^b \wedge \star_{g_0} (\psi^* X'^a \wedge g_{bc}X^c) +O(r^{-1-2q_1}).
\end{equation}
Moreover, applying \eqref{model-decay} and \eqref{0hqj309rjh9qihh} one more time produces
\begin{equation}
\star_{g_0} (\psi^* X'^a \wedge g_{bc}X^c)=\mathbf{O}_a^n \star_{g_0}(X^n \wedge X^b) +O(r^{-q_1}).
\end{equation}
It follows that
\begin{align}
\begin{split}
\psi^* Z' -Z &=d\left(Q_a^b \mathbf{O}_a^n \star_{g_0}(X^n \wedge X^b)\right)
+Q_a^b \mathbf{O}_a^n d\star_{g_0}(X^n \wedge X^b)+O(r^{-1-2q_1})\\
&=d\left(Q_a^b \mathbf{O}_a^n \star_{g_0}(X^n \wedge X^b)\right)
+O(r^{-1-2q_1})
\end{split}
\end{align}
since 
\begin{equation}
Q_a^b \mathbf{O}_a^n=\delta^{bn}+O(r^{-q_1}),\quad\quad\quad d\star_{g_0}(X^n \wedge X^b) =O(r^{-2}),
\end{equation}
where the last equation arises from the decay of Christoffel symbols as in the proof of Proposition \ref{AEend-decay}.
Hence, with $q_1 >1/2$ we find that the mass of the end $\E$ as determined by the two ALF structures must agree.
\end{proof}



\begin{corollary}\label{quo-sep}
Let $(M^4, g, \E)$ be a complete ALF manifold with an almost free $U(1)$ action and a designated end $\E$. 
Consider the quotient map $\pi: \E \rightarrow \bar{\E}:=\E /U(1)$ along with the AE Riemannian quotient space $(\bar{\E},\bar{g})$, and assume that $R_{\bar{g}}\in L^1 (\bar{\E})$.
Then the mass $m$ of the original end is related to the ADM mass $\bar{m}$ of the quotient end by
\begin{equation}\label{mass-split}
m =4\bar{m}+\lim_{r\rightarrow \infty}\frac{1}{2\pi}\int_{\bar{S}_{r} } \langle\pi_{*}\mathcal{N},\bar{\nu}\rangle 
\end{equation}
where $\bar{S}_{r}\subset\bar{\mathcal{E}}$ is a coordinate sphere with unit outer normal $\bar{\nu}$, the Euclidean inner product is denoted by $\langle\cdot,\cdot\rangle$, and $\mathcal{N}$ is the mean curvature vector of the $U(1)$ fibers within $\E$. 
\end{corollary}

\begin{proof}
Recall that the metric on $\E$ may be expressed in Riemannian submersion format 
\begin{equation}
g=\pi^* \bar{g}+\frac{\eta^2}{|\eta|^2}, 
\end{equation}
where $\eta=g(T,\cdot)$ is the dual 1-form to the $U(1)$ generator $T$. According to
Proposition \ref{AEend-decay}, there exists a coordinate system $\bar{x}=(\bar{x}^1, \bar{x}^2, \bar{x}^3)$ on the quotient end $\bar{\E}$ that yields an AE structure. By pulling these functions back to $\E$ and denoting them by $\tilde{x}^i =\pi^* \bar{x}^i$ we obtain
\begin{equation}
\pi^* \bar{g} =\bar{g}_{ij} d\tilde{x}^i d\tilde{x}^j =(\delta_{ij}+O_2(r^{-q}))d\tilde{x}^i d\tilde{x}^j,
\end{equation}
where $r=|\bar{x}|$. Combining this with the flow parameter $t$, for the vector field $T$, yields local coordinates $(\tilde{x},t)$ on $\E$ such that
\begin{equation}
\frac{\eta}{|\eta|}=\ell (1+O_2(r^{-q}))dt +\tilde{A}_i d\tilde{x}^i,\quad\quad \tilde{A}_i =O_2(r^{-q}).
\end{equation}
These observations yield an ALF structure in which the components of the metric satisfy the following fall-off conditions
\begin{align}\label{890}
\begin{split}
g_{ij}&=g(\partial_{\tilde{x}^i},\partial_{\tilde{x}^j})=\bar{g}_{ij} +\frac{\eta(\partial_{\tilde{x}^i})\eta(\partial_{\tilde{x}^j})}{|\eta|^2}=\bar{g}_{ij}+O_2(r^{-q}),
\\
g_{i4}=g(&\partial_{\tilde{x}^i},\ell^{-1}\partial_{t})=O_2(r^{-q}),\quad\quad\quad 
g_{44}=g(\ell^{-1}\partial_t ,\ell^{-1}\partial_t)=1 +O_2(r^{-q}),
\end{split}
\end{align}
for $i,j=1,2,3$. Moreover, the unit outer normal to the surface $\mathcal{S}^3_{r}$ admits the asymptotics
$\nu=\tfrac{\tilde{x}^i}{r}\partial_{\tilde{x}^i}+O(r^{-q})$. It follows that the mass flux density may be expressed as
\begin{equation}\label{fijqpj9qj39j0u-}
\sum_{a,b=1}^{4}(g_{ab,a}-g_{aa,b})\nu^b
=\sum_{i,j=1}^{3}(\bar{g}_{ij,i}-\bar{g}_{ii,j})\bar{\nu}^j +\sum_{b=1}^{4}(g_{4b,4}-g_{44,b})\nu^b +O(r^{-1-2q}),
\end{equation}
where $\bar{\nu}=\partial_{r}$. Furthermore since
\begin{align}
\begin{split}
\sum_{b=1}^{4}g_{4b,4}\nu^b &=\ell^{-2}\partial_t g(\partial_t,\nu) +O(r^{-1-2q}),\\
\sum_{b=1}^{4}g_{44,b}\nu^b &=2\ell^{-2}g(\nabla_{\partial_t} \nu,\partial_t)+O(r^{-1-2q})=2 g(\mathcal{N},\nu)+O(r^{-1-2q}),
\end{split}
\end{align}
and the $t$-derivative term integrates to zero along the flux surfaces, we find that
\begin{equation}\label{pgfjq-9jr-09j-0y}
\int_{\mathcal{S}^3_{r}}\sum_{b=1}^{4}(g_{4b,4}-g_{44,b})\nu^b 
=4\pi\ell\int_{\bar{S}_{r} } \langle\pi_{*}\mathcal{N},\bar{\nu}\rangle +O(r^{1-2q}).
\end{equation}
Since the mass of $\E$ as computed with respect to this particular ALF structure agrees with $m$ by Proposition \ref{mass-well-def}, and the ADM mass of $\bar{\E}$ is well-defined in light of the integrability of $R_{\bar{g}}$, the desired result is obtained by integrating \eqref{fijqpj9qj39j0u-} over $\mathcal{S}^3_{r}$ and passing to the limit, after applying \eqref{pgfjq-9jr-09j-0y}.
\end{proof}

\section{The Local Structure of Quotient Space Singularities}
\label{sec3} \setcounter{equation}{0}
\setcounter{section}{3}





In this section we will study the quotient space of 4-manifolds under an almost free U(1) action. Previous work on this topic has been carried out in \cites{AADN,Jang}, however little attention has been dedicated to the local geometry of the quotient space. 
Our main result in this direction provides an asymptotic expression for the quotient metric near singular points.
Note that although Definition \ref{def-alf-mfd} describes almost free $U(1)$ actions with respect to an ALF end, here these ends will play no role. Thus, when referring to such actions in this section, part (iii) of the definition may be ignored.

\begin{theorem}\label{local} 
Let $(M^{4}, g)$ be a complete Riemannian 4-manifold with an almost free $U(1)$ action. 
Then the Riemannian quotient space $(\bar{M}^3, \bar{g})$ is a smooth Riemannian 3-manifold in the compliment of finitely many points $\{\bar{p}_1 ,\dots, \bar{p}_k\}$. Moreover, for each $i\in \{1,\dots,k\}$ there exists a neighborhood $\bar{U}_i \subset \bar{M}^3$ with polar coordinates such that 
\begin{equation}
\bar{g}= d\bar{r}^2+{\bar{r}^2}g_{CP^{1}}+ \pmb{\epsilon}_{\bar{r}}    
\quad \text{ on } \bar{U}_i \setminus\{\bar{p}_i\},
\end{equation}
where $g_{CP^1}=\frac{1}{4}g_{S^2}$ is the Fubini-Study metric on the complex projective line $CP^1$ and
where $\pmb{\epsilon}_{\bar{r}}=O_{2}(\bar{r}^3)$ is a 1-parameter family of symmetric 2-tensors on $S^2$.
\end{theorem}

In order to establish this theorem, we will first observe that at any singular point $p_i \in M^4$ of the $U(1)$ action
the tangent space is a metric cone over the unit round 3-sphere $(S^3, g_{S^3})$, and the lift of the almost free action to the tangent space produces an isometric and free $U(1)$ action on these cross-sections. It will then be shown that the quotient of these spheres by the lifted action is the cross section of the tangent cone of $\bar{M}^3$ at $\bar{p}_i$. A computation of linking numbers can then be used to conclude that the tangent cone cross-section is $CP^1$.

\subsection{Preliminary observations}  
Let $\{p_1,\dots,p_k\}$ denote the singular points of the almost free $U(1)$ action on $M^4$. We first note that each $p_i$ must be a fixed point of the action, in that the isotropy group at that point is the whole group. To see this, note that if the isotropy is not the whole group then it must be a finite cyclic group, making $p_i$ an exceptional point. However, according to \cite{Fintushel1}*{Proposition 3.1} and \cite{Fintushel}*{Section 9} the set of exceptional points must be open, which contradicts the almost free assumption.  

In what follows, $p\in M^4$ will denote a generic fixed point and $\bar{p}\in\bar{M}^3$ will denote its image under the quotient map. It is clear that the quotient space $(\bar{M}^3 ,\bar{g})$ is a smooth Riemannian manifold in the compliment of the fixed point images. Consider the flow $\varphi_t$ associated with the Killing field generator of the $U(1)$ action, so that $\varphi_t \in\text{Isom}(M^4, g)$ for each $t\in\mathbb{R}$. Observe that the linearized map $(d\varphi_t)_p \in \text{Isom}(T_{p}M^4, g_{p})$ 
induces an isometric $U(1)$ action, denoted $\tilde{\varphi}_t$, on the unit sphere $S^3\subset T_{p}M^4$. Let $(\Sigma^2, g_{\Sigma})$ be the quotient space of $(S^3, g_{S^3})$ by this action, and write $\Sigma_{\bar{r}}^2 \subset\bar{M}^3$ for the geodesic sphere centered at $\bar{p}$ of $\bar{g}$-distance $\bar{r}$ to this point. We will now show that the tangent cone at $\bar{p}\in\bar{M}^3$ is a metric cone over this surface $(\Sigma^2, g_{\Sigma})$. 

\begin{proposition}\label{lim-action}
Let $(M^{4}, g)$ be a complete Riemannian 4-manifold with an almost free $U(1)$ action. Consider a singular point $p\in M^4$ of the action. Then in the Gromov-Hausdorff topology
\begin{equation}
(\Sigma_{\bar{r}}^2 , \bar{r}^{-2}\bar{g}|_{\Sigma_{\bar{r}}^2})\rightarrow (\Sigma^2 , g_{\Sigma}) \quad\text{ as }\bar{r}\rightarrow 0.
\end{equation}
Moreover, if the induced $U(1)$ action on $S^3\subset T_p M^4$ is free, there exists a neighborhood $\bar{U} \subset \bar{M}^3$ of $\bar{p}$ such that
\begin{equation}\label{fajw-9rjq093oj}
\bar{g}= d\bar{r}^2+{\bar{r}^2}g_{\Sigma}+\pmb{\epsilon}_{\bar{r}} 
\quad \text{ on } \bar{U} \setminus\{\bar{p}\},
\end{equation}
where $\pmb{\epsilon}_{\bar{r}}=O_{2}(\bar{r}^3)$ is a 1-parameter family of symmetric 2-tensors on $\Sigma^2$.
\end{proposition}
 

\begin{proof}
According to the discussion preceding this proposition, $p$ must be a fixed point of the $U(1)$ symmetry. If $r_0 >0$ is less than the injectivity radius at $p$, then we may lift the action to obtain $\hat{\varphi}_t \in  \text{Isom}(B_{r_0}(0)\subset T_{p}M^4, \exp_p^* g)$ given by
\begin{equation}\label{act-lift}
\hat{\varphi}_t (v)=\exp^{-1}_p \circ \varphi_t \circ \exp_p(v),\quad\quad\quad v\in B_{r_0}(0).
\end{equation}
Using geodesic polar coordinates and the Gauss lemma, the pullback metric may be expressed as
\begin{equation}\label{total-asy}
\exp^*_p g=dr^2+r^2g_{r},\quad\quad \quad 
g_{r}=g_{S^3}+O_2(r),
\end{equation}
for some family of metrics $g_r$ on the 3-sphere. It follows that
\begin{equation}\label{taylor-group-order}
\hat{\varphi}_t=r\tilde{\varphi}_t+O_2(r^2), 
\end{equation}
and on the sphere $(S^3,g_r)$ we obtain the isometric $U(1)$ action defined by
\begin{equation}\label{scal-action-def}
\hat{\varphi}^r_t (w):=r^{-1} \hat{\varphi}_t(rw),\quad\quad\quad w\in T_p M^4, \quad |w|=1. 
\end{equation}
Moreover, the expansion \eqref{taylor-group-order} implies that
\begin{equation} 
\lim_{r\rightarrow 0}\hat{\varphi}_t^r (w)=\tilde{\varphi}_t (w).
\end{equation}
Combining this observation with \eqref{total-asy} shows that  $(S^{3}, g_{r}, \hat{\varphi}^r_t)$ converges to $(S^{3}, g_{S^{3}}, \tilde{\varphi}_t)$ in the equivariant Gromov-Hausdorff topology (see \cite{FY}*{Definition 3.3}). The following diagram expresses the relation with quotient spaces, where the vertical arrows indicate the quotient operation and $\bar{r}$ is the descent of $r$ to the quotient space.
\begin{center}\begin{tikzcd}
(S^{3}, g_{r})\arrow[r,"r\rightarrow 0"]\arrow[d,"U(1)\text{ action} ~\hat{\varphi}^r_t"]& (S^{3}, g_{S^{3}})\arrow[d, "U(1)\text{ action}~ \tilde{\varphi}_t"]
\\ ( \Sigma^2_{\bar{r}}, \bar{r}^{-2}\bar{g}|_{\Sigma^2_{\bar{r}}})\arrow[r, "\bar{r}\rightarrow 0"]&(\Sigma^2, g_{\Sigma})
\end{tikzcd}
\end{center}
According to \cite{Fintushel1}*{Proposition 3.1} and \cite{Fintushel}*{Section 9} the orbit spaces are Riemannian manifolds possibly with boundary, away from a potentially empty collection of curves and a finite number of points.
Furthermore, \cite{FY}*{Lemma 3.4} implies that the lower horizontal arrow of the diagram holds in the Gromov-Hausdorff sense. 

Now assume that $\tilde{\varphi}_t$ acts freely on $S^3$, so that $\Sigma^2$ and $\Sigma_{\bar{r}}^2$ are smooth Riemannian manifolds. We will obtain the asymptotics of $\bar{g}$ in $\bar{U}\ni \bar{p}$, the image under the quotient map of $B_{r_0}(p)\subset M^4$. Note that since the curvature of $\Sigma^2_{\bar{r}}$ is uniformly controlled by O'Neill's formula, Theorem 0.6 of \cite{Colding97} shows that this sequence of manifolds converges to $\Sigma^2$ in the $C^{1,\alpha}$ topology; in particular, $\Sigma^2_{\bar{r}}$ is diffeomorphic to $\Sigma^2$.
Consider the Killing fields $\hat{X}$, $\hat{X}^{r}$, and $\tilde{X}$ associated with the flows $\hat{\varphi}_t$, $\hat{\varphi}^r_t$, and $\tilde{\varphi}_t$, and their dual 1-forms $\hat{\eta}$, $\hat{\eta}^r$, and $\tilde{\eta}$. Observe that \eqref{scal-action-def} implies $\hat{X}^r (w)=r^{-1}\hat{X}(rw)$, and thus comparing with the time derivative of \eqref{taylor-group-order} produces
\begin{equation}\label{f0hj-0r9hjq-0hj}
|\tilde{\nabla}^{l}(\hat{X}^r -\tilde{X})|_{g_{S^3}}=O(r) ,\quad\quad\quad |\tilde{\nabla}^l (\hat{\eta}^r -\tilde{\eta})|_{g_{S^3}}=O(r),\quad\quad l=0,1,2,
\end{equation}
where $\tilde{\nabla}$ denotes covariant differentiation with respect to $g_{S^3}$. Next, express the metrics on the 3-sphere in Riemannian submersion format
\begin{equation}
g_r=\hat{\pi}_{\bar{r}}^* \left(\bar{r}^{-2} \bar{g}|_{\Sigma^2_{\bar{r}}}\right)+\frac{(\hat{\eta}^r)^2}{|\hat{\eta}^r|_{g_r}^2},\quad\quad\quad
g_{S^3}=\tilde{\pi}^* g_{\Sigma} +\frac{\tilde{\eta}^2}{|\tilde{\eta}|^2_{g_{S^3}}},
\end{equation}
where 
\begin{equation}
\hat{\pi}_{\bar{r}}: S^3 \rightarrow \Sigma^2_{\bar{r}},\quad\quad\quad \tilde{\pi}:S^3 \rightarrow \Sigma^2 ,
\end{equation}
are quotient maps. Since the action of $\tilde{\varphi}_t$ is free it follows that $\min_{S^3}|\tilde{\eta}|_{g_{S^3}}>0$, and hence \eqref{total-asy} and \eqref{f0hj-0r9hjq-0hj} yield
\begin{equation}\label{fpjw-9h4j9poq}
\Big|\hat{\pi}_{\bar{r}}^* \left(\bar{r}^{-2}\bar{g}|_{\Sigma^2_{\bar{r}}}\right)-\tilde{\pi}^* g_{\Sigma}\Big|_{g_{S^3}}\leq|g_r -g_{S^3}|_{g_{S^3}}+\Bigg|\frac{(\hat{\eta}^r)^2}{|\hat{\eta}^r|_{g_r}^2}-\frac{\tilde{\eta}^2}{|\tilde{\eta}|^2_{g_{S^3}}}\Bigg|_{g_{S^3}}=O(r),
\end{equation}
with corresponding asymptotics for derivatives.
Moreover, with the estimate between Killing fields \eqref{f0hj-0r9hjq-0hj}, we may use their flows as in Proposition \ref{AEend-decay} to find diffeomorphisms $f_{\bar{r}} :\Sigma^2 \rightarrow \Sigma^2_{\bar{r}}$ such that the maps $f_{\bar{r}}\circ\tilde{\pi}$ and $\hat{\pi}_{\bar{r}}$ are $O(\bar{r})$-close in the $C^2(S^3,\Sigma_{\bar{r}}^2)$ topology. Therefore \eqref{fpjw-9h4j9poq} gives
\begin{equation}
\Big| \tilde{\pi}^* f_{\bar{r}}^* (\bar{r}^{-2}\bar{g}|_{\Sigma^2_{\bar{r}}})-\tilde{\pi}^* g_{\Sigma}\Big|_{g_{S^3}}
\leq \Big|\hat{\pi}^* \left(\bar{r}^{-2}\bar{g}|_{\Sigma^2_{\bar{r}}}\right)-\tilde{\pi}^* g_{\Sigma}\Big|_{g_{S^3}}+O(\bar{r}) =O(\bar{r}),
\end{equation}
with corresponding estimates for derivatives. The desired expansion \eqref{fajw-9rjq093oj} now follows.
\end{proof}

The behavior of the quotient metric near $\bar{p}\in\bar{M}^3$ is closely tied to the $U(1)$ action of $\tilde{\varphi}_t$ on $(S^3, g_{S^3})$. Since this is the action of a 1-parameter subgroup of $SO(4)$, it is given by $\tilde{\varphi}_t =e^{\tilde{\mathcal{K}}t}$ for some $\tilde{\mathcal{K}}\in so(4)$. Moreover, as $i\tilde{\mathcal{K}}$ is Hermitian and purely imaginary, it has eigenvalues $\pm \tilde{k}_1 , \pm \tilde{k}_2 \in \mathbb{R}$. Neither of these values can vanish, since the action has an isolated fixed point at $p\in M^4$.
By choosing appropriate normal coordinates and identifying $T_{p}M^4$ with $\mathbb{C}^2$, the linear transformation $\tilde{\mathcal{K}}$ may be represented by the matrix
\begin{equation}
\mathcal{K}=\begin{bmatrix}
i\tilde{k}_1& 0\\
0 &i\tilde{k}_2\\
\end{bmatrix}.
\end{equation}
Furthermore, periodicity of the action implies that $\tilde{k}_1 /\tilde{k}_2 = k_1 /k_2$ for some relatively prime $k_1 , k_2 \in \mathbb{Z}$. 
Upon rescaling the parameter $t$ if necessary (the notation will remain unchanged), the $U(1)$ action on $S^3 \subset\mathbb{C}^2$ may then be represented as
\begin{equation}\label{expforphi}
\tilde{\varphi}_t =\begin{bmatrix}
e^{i k_1 t }& 0\\
0 &e^{i k_2 t}\\
\end{bmatrix}, \quad\quad t\in[0,2\pi].
\end{equation}
The next result will be established in Section \ref{Section 3.3}, and when combined with Proposition \ref{lim-action} will complete the proof of Theorem \ref{local}.

\begin{theorem}\label{multi} 
Let the hypotheses of Proposition \ref{lim-action} hold. Then $\tilde{\varphi}_t$ is a free $U(1)$ action on the 3-sphere, and $k_j =\pm 1$ for $j=1,2$.
\end{theorem}


\subsection{Linking number} 
As we will see, the possible values of $k_j$ are closely tied to the linking number of specific knots in $S^3$. Here we will briefly recall the concept of linking number, and indicate its application to the $U(1)$ action on $M^4$. In what follows $L_1$ and $L_2$ will denote closed curves in the 3-sphere. Furthermore, a \textit{meridian} of $L_1$ is a generator for the kernel of the inclusion map $\pmb{\pi_1}(\partial B_{\varepsilon}(L_1))\rightarrow \pmb{\pi_1}(B_{\varepsilon}(L_1))$, where $B_{\varepsilon}(L_1)$ is the closed tubular neighborhood of $L_1$ in $S^3$ of radius $\varepsilon$. When additional curves are involved in a link, $\varepsilon$ will implicitly be taken suffiently small so that the tubular neighborhood is disjoint from the other curves.

\begin {definition}\label{lk-number}
Let $L_1\sqcup L_2\subset S^3$ be a link with two components. Its \textit{linking number} is the value $\mathrm{lk}(L_1, L_2)\in \mathbb{Z}$ such that
\begin{equation}
[L_2]=\mathrm{lk}(L_1, L_2)[m_1]\quad\text{ in } H_1(S^3\setminus L_1 ;\mathbb{Z}),
\end{equation} 
where $m_1$ is a meridian of $L_1$.
\end{definition}

\begin{lemma}\label{degree-lk} 
Let $L_1\sqcup L_2$ and $L'_1 \sqcup L'_2$ be two links in $S^3$, each with two components. If $L'_j \subset B_{\varepsilon_j}(L_j)$ for $j=1,2$ then
\begin{equation}
\mathrm{lk}(L'_1, L'_2)=\deg (\mathrm{Pr}_{L_1}|_{L'_1})\cdot \deg (\mathrm{Pr}_{L_2}\vert_{L'_2})\cdot \mathrm{lk}(L_1, L_2),
\end{equation}
where $\mathrm{Pr}_{L_j}: B_{\varepsilon_j}(L_j) \rightarrow L_j$ denotes the projection map.
\end{lemma}

\begin{proof}
According to the loop lemma \cite{HA}*{Theorem 3.1}, there exists an embedded disc $D_1 \subset B_{\varepsilon_1}(L_1)$ bounded by a meridian $m_1$ of $L_1$. It may be assumed without loss of generality that $L'_1$ intersects $D_1$ transversely. Thus each  component of $B_{\varepsilon'_1}(L'_1)\cap D_1$ is a disc, where $\varepsilon'_1$ is chosen small enough so that $B_{\varepsilon'_1}(L'_1)\subset B_{\varepsilon_1}(L_1)$. The boundary of each such component is homologous to the meridian $m'_1$ of $L'_1$, up to a sign determined by orientation. Using the surface $D_1 \setminus B_{\varepsilon'_1}(L'_1)$ this implies that
\begin{equation}\label{degree1}
[m_1]=\deg (\mathrm{Pr}_{L_1}|_{L'_1})[m'_1] \quad\text{ in }  H_1(B_{\varepsilon_1}(L_1)\setminus \mathring{B}_{\varepsilon'_1}(L'_1);\mathbb{Z}),
\end{equation}
where the top circle notation indicates the interior of a set. Furthermore, this equation also holds in $H_1 (S^3\setminus L'_1 ;\mathbb{Z})$. Next observe that
\begin{align}\label{degree2}
\begin{split}
 [L_2]&=\mathrm{lk}(L_1, L_2)[m_1]\quad\text{ in }H_1(S^3 \setminus \mathring{B}_{\varepsilon_1}(L_1);\mathbb{Z}),\\
[L'_2]&=\deg(\mathrm{Pr}_{L_2}|_{L'_2})[L_2] \quad\text{ in } H_1(B_{\varepsilon_2}(L_2);\mathbb{Z}),
\end{split}
\end{align}
and these equations hold in $H_1(S^3\setminus L'_1 ;\mathbb{Z})$ as well.
Consequently, in $H_1(S^3\setminus L'_1;\mathbb{Z})$ we find
\begin{align}
\begin{split}
[L'_2]&=\deg(\mathrm{Pr}_{L_2}|_{L'_2})[L_2] \\
&=\deg(\mathrm{Pr}_{L_2}|_{L'_2})\mathrm{lk}(L_1, L_2)[m_1]\\ 
&=\deg(\mathrm{Pr}_{L_2}|_{L'_2})\mathrm{lk}(L_1, L_2)\deg(\mathrm{Pr}_{L_1}|_{L'_1})[m'_1],
\end{split}
\end{align} 
from which the desired result now follows.
\end{proof}

 \begin{lemma}\label{hopf} 
If $\pi:S^3\rightarrow S^2$ be a principal $U(1)$ bundle, then it is the Hopf fibration. Moreover, for any two distinct points $p, q \in S^2$, the pre-image  $\pi^{-1}(\{p, q\})$ forms  a (Hopf) link with linking number  $\pm 1$. 
 \end{lemma}

 \begin{proof}
As described in \cite{Chern}, the isomorphism classes of principal $U(1)$ bundles over $S^2$ are in one-to-one correspondence with $\pmb{\pi_1}(SO(2))\cong\mathbb{Z}$.
%
In particular, the total space of the bundle corresponding to integer $n$ is the lens space $L(n,1)$. Thus, such a bundle with $S^3$ total space corresponds to $n=\pm 1$ and must be the Hopf fibration. It follows that $\pi^{-1}(\{p, q\})$ is a Hopf link with linking number  $\pm 1$. 
 \end{proof}

 \begin{corollary}\label{quo-hopf} 
Assume the setting of Proposition \ref{lim-action}. The $U(1)$ action of $\hat{\varphi}_t^r$ on $(S^3 ,g_r)$ is free for $r\in (0,r_0)$. Moreover, the corresponding quotient space $\Sigma_{\bar{r}}^2$ is diffeomorphic to $S^2$, and the quotient map yields the Hopf fibration.
 \end{corollary}

 \begin{proof}
According to the definition \eqref{scal-action-def} of $\hat{\varphi}_t^r$, a singular point for this action with $r>0$ produces a singular point for $\varphi_t$ away from $p$, which contradicts the almost free property of $\varphi_t$. Thus, the action of $\hat{\varphi}_t^r$ is free for $r\in (0,r_0)$. For each such $r$, the quotient map gives a principal $U(1)$ bundle over $\Sigma_{\bar{r}}^2$.
Using the long exact sequence for fibrations \cite{HA1}*{Theorem 4.41}, we find that the quotient manifold is simply connected and $\pmb{\pi_2} (\Sigma_{\bar{r}}^2)\cong \bb{Z}$, which implies that $\Sigma_{\bar{r}}^2 \cong S^2$. Furthermore, by Lemma \ref{hopf} the quotient map yields the Hopf fibration. 
 \end{proof}

\subsection{Proof of Theorem \ref{multi}}\label{Section 3.3}
Since $\tilde{\varphi}_t$ may be represented by \eqref{expforphi} with integers $k_1$ and $k_2$ that are nonzero, its action on $(S^3 ,g_{S^3})$ is free. Consider the two orbits of this action given by
\begin{equation}
\tilde{C}_1=\{(z_1, 0)\in \mathbb{C}^2 \mid |z_1|=1\}, \quad \quad\quad \tilde{C}_2=\{(0, z_2)\in\mathbb{C}^2 \mid |z_2|=1\}.
\end{equation}
Note that the points $w_1 =(1, 0)$ and $w_2=(0,1)$ lie in $\tilde{C}_1$ and $\tilde{C}_2$, respectively, and $\tilde{C}_1\sqcup \tilde{C}_2$ forms a Hopf link with $\mathrm{lk}(\tilde{C}_1, \tilde{C}_2)=\pm1$.
According to the proof of Proposition \ref{lim-action}, we have that $(S^3, g_{r}, \hat{\varphi}_t^r)$ converges to $(S^3, g_{S^3}, \tilde{\varphi}_t)$ in the equivariant Gromov-Hausdorff topology. Therefore, the orbits $\tilde{C}_j$, $j=1,2$ of $\tilde{\varphi}_t$ are approximated by the two orbits of $\hat{\varphi}_t^r$ defined as
\begin{equation}
\hat{C}_j^r =\{\hat{\varphi}_t^r (w_j) \mid t\in U(1)\}, \quad\quad j=1,2.
\end{equation}
In particular, given $\varepsilon>0$ there exists $r_{\varepsilon}\in (0,r_0)$ such that for all $r<r_{\varepsilon}$ we have $\hat{C}_j^r \subset B_{\varepsilon}(\tilde{C}_j)$ for $j=1,2$.
Observe that the integer $k_j$ is the degree of the map $\mathrm{Pr}_{\tilde{C}_j}|_{\hat{C}^r_j}$, where $\mathrm{Pr}_{\tilde{C}_j}: B_{\varepsilon}(\tilde{C}_j)\rightarrow \tilde{C}_j$ is the projection map.  Applying Lemma \ref{degree-lk} produces 
\begin{equation}
\mathrm{lk}(\hat{C}^r_1, \hat{C}^r_2)=\deg(\mathrm{Pr}_{\tilde{C}_1}|_{\hat{C}^r_1})\cdot \deg(\mathrm{Pr}_{\tilde{C}_2}|_{\hat{C}^r_2})\cdot \mathrm{lk}(\tilde{C}_1, \tilde{C}_2)=\pm k_1 k_2. 
\end{equation}
Furthermore, Corollary \ref{quo-hopf} shows that the action of $\hat{\varphi}_t^r$ induces the Hopf fibration, and hence Lemma \ref{hopf} implies that $\mathrm{lk}(\hat{C}^r_1, \hat{C}^r_2)=\pm1$. We conclude that $k_1k_2=\pm1$, which yields the desired result.

\subsection{Proof of Theorem \ref{local}}
The almost free $U(1)$ action on $(M^{4}, g)$ immediately produces a Riemannian quotient space $(\bar{M}^3, \bar{g})$ which is a smooth Riemannian 3-manifold in the compliment of finitely many points $\{\bar{p}_1 ,\dots, \bar{p}_k\}$. For each $i$, Proposition \ref{lim-action} and Theorem \ref{multi} show that
there exists a neighborhood $\bar{U}_i \subset \bar{M}^3$ of $\bar{p}_i$ such that
\begin{equation}
\bar{g}= d\bar{r}^2+{\bar{r}^2}g_{\Sigma}+\pmb{\epsilon}_{\bar{r}}   
\quad \text{ on } \bar{U}_i \setminus\{\bar{p}_i\},
\end{equation}
where $(\Sigma^2 , g_{\Sigma})$ is the quotient of $(S^3, g_{S^3})$ by the $U(1)$ action $\tilde{\varphi}_t$. Corollary \ref{quo-hopf} implies that $\Sigma^2$ is topologically a 2-sphere. Furthermore, the action $\tilde{\varphi}_t$ is given by \eqref{expforphi} and Theorem \ref{multi} shows that $k_1 k_2 =\pm 1$. It follows this quotient yields the Hopf fibration, and therefore $g_{\Sigma}=\frac{1}{4}g_{S^2}$.

\begin{remark}\label{singular-beh}
For later use, we note that the proofs of Theorems \ref{local} and \ref{multi} imply that in each neighborhood $U_i$ about a singular point $p_i$, there exists a normal coordinate system such that the Killing field generator $T$ of the $U(1)$ action and its dual 1-form $\eta$ take the form
\begin{align}\label{local-killing}
\begin{split}
T(x)&=\pm\left(x^2 \partial_{x^1} -x^1 \partial_{x^2}\right) \pm\left(x^4 \partial_{x^3} -x^3 \partial_{x^4} \right)+ O_2(|x|^2),\\
\eta(x)&=\pm\left(x^2 dx^1-x^1 dx^2 \right)\pm \left(x^4 dx^3 -x^3 dx^4 \right)+O_2(|x|^2).
\end{split}
\end{align}
In particular $|T|_g=|\eta|_g =|x| +O(|x|^2)$ and $|d\eta|_g =\sqrt{8} +O(|x|)$.
\end{remark}

\section{Density for Harmonically ALF Manifolds} 
\label{sec4} \setcounter{equation}{0}
\setcounter{section}{4}

Although Proposition \ref{AEend-decay} implies that the quotient space of a complete ALF manifold with almost free $U(1)$ action is AE, it does not guarantee that the quotient scalar curvature is integrable, see Remark \ref{Prop2.2}. This lack of integrability prevents a direct application of the mass splitting that occurs in Corollary \ref{quo-sep}. Nevertheless, as will be shown in this section, the original ALF metric may be approximated by those which are harmonically ALF. This will allow us to recover the desired mass splitting in later sections.
To accomplish this, we will employ a conformal gluing in order to replace the end of the given ALF manifold with one that is asymptotically more simple. Let
\begin{equation}
g_0= \delta_{ij}dx^i dx^j+\ell^{2}\tau^2
\end{equation}
be the model metric for an ALF end, and observe that for $m\in\mathbb{R}$ we have
\begin{equation}\label{harmonic-construct}
 \Delta_{g_0}\left(1+\frac{m}{6r}\right)=0
\end{equation}
where $r^2=\sum_{i=1}^3 (x^i)^2$. Moreover, these harmonic functions yield ALF metrics $\left(1+\frac{m}{6r}\right)^2 g_0$, having mass $m$. Note that although the new metrics do not amit nonnegative scalar curvature, they retain the same decay as in \eqref{decay-scal}. 

\begin{definition}\label{foanofinapinp}
An end $(\mathcal{E},g)$ of an ALF manifold is called \textit{harmonically ALF} if the following decay holds
\begin{equation}
\Big|\mathring{\nabla}^l \left(g-\left(1+\frac{m}{6r}\right)^2g_0 \right)\Big|_{g_0}=O(r^{-\mathring{q}-l}), \quad\quad l=0,1,2,
\end{equation}
for some $m\in\mathbb{R}$ and $\mathring{q}>1$.
\end{definition}

Observe that the mass of a harmonically ALF end is the parameter $m$ from the definition. In analogy with AE case \cite{SYEL} (see also \cite{CorvinoPollack}*{Proposition 3.3}), any ALF manifold may be approximated by one with harmonically ALF ends while preserving nonnegative scalar curvature.

\begin{theorem}\label{ALF-density} 
Let $(M^4,g,\E)$ be a complete ALF manifold having nonnegative scalar curvature, and an almost free $U(1)$ action with respect to a designated end $\E$.
For any $\varepsilon>0$, there exists a complete metric $g'$ on $M^4$ which is $\varepsilon$-close to $g$ and has the following properties.
\begin{itemize}
\item[(i)] The ALF manifold $(M^4 ,g')$ admits an almost free $U(1)$ action with respect to $\E$.
\item[(ii)] The scalar curvature of $g'$ is nonnegative, $R_{g'}\geq 0$.
\item[(iii)] The respective masses $m$ and $m'$ of $(\E,g)$ and $(\E,g')$ satisfy $|m-m'|< \varepsilon $. 
\item[(iv)] $(\E,g')$ is a harmonically ALF end. In particular, there exists a function $f\in C^{\infty}(M^4)$ and a compact set $\mathcal{K}\subset M^4$ such that 
\begin{equation}\label{conformal-def}
g'=f^2 g_0 \quad\text{ on }\E \setminus\mathcal{K},\quad\quad\quad f-\left(1+\frac{m'}{6{r}}\right)\in C^{2, \alpha}_{1+q'}(\E),
\end{equation}
where $q'=\min\{1,q\}$, $\alpha\in(0,1)$, and $m'$ is the mass of $(\E,g')$.
\end{itemize}
\end{theorem}

In order to establish this result, we will first investigate the existence of solutions for certain linear elliptic equations on ALF manifolds, in preparation for conformal deformation.  We begin with a brief technical lemma that will be employed in what follows.

\begin{lemma} \label{Schauder1} 
Let $(M^4,g,\E)$ be a complete ALF manifold having an almost free $U(1)$ action with respect to a designated end $\E$.
Suppose that $u$, $b$, and $c$ are smooth $U(1)$-invariant functions on $\E$ satisfying
\begin{equation}\label{ahf09ahj-09fj-a9}
\Delta_g u+cu=b.
\end{equation}
If $u$ is bounded and $b,c\in C^{0, \alpha}_{3+\varsigma}(\E)$ for some $\alpha, \varsigma\in (0, 1)$,
then there are $a_0, a_1 \in\mathbb{R}$ such that
\begin{equation}
u-\left(a_{0}+\frac{a_1}{r}\right)\in C^{2,\alpha}_{1+\varsigma'}
\end{equation}
where $\varsigma'=\min\{q, \varsigma\}$.
\end{lemma}

\begin{proof}
Near any point of $\E$ coordinates may be introduced so that the metric takes the form
\begin{equation}
g=\pi^* \bar{g}+\frac{\eta^2}{|\eta|^2}=\bar{g}_{ij}dx^i dx^j +|\eta|^2 (dt+A_i dx^i)^2,
\end{equation}
where $\pi:M^4 \rightarrow \bar{M}^3 =M^4 /U(1)$ is the quotient map, $\bar{g}$ is the quotient metric, and $\eta$ is the 1-form dual of the $U(1)$ generator $T=\partial_t$. By $U(1)$-invariance there exist functions $\bar{u}$, $\bar{b}$, and $\bar{c}$ on the quotient $\bar{\E}=\E/U(1)$ such that $u=\pi^* \bar{u}$, $b=\pi^* \bar{b}$, and $c=\pi^* \bar{c}$. We then have
\begin{equation}
\Delta_g u=\sum_{i,j=1}^{3}\frac{1}{\sqrt{|\eta|\det \bar{g}}}\partial_{x^i}\left(\sqrt{|\eta|\det\bar{g}} \text{ }\bar{g}^{ij}\partial_{x^j} u\right)=\pi^*\left(\Delta_{\bar{g}}\bar{u}+\frac{1}{2}\bar{g}(\bar{\nabla}\log|\eta|,\bar{\nabla}\bar{u})\right).
\end{equation}
Thus, equation \eqref{ahf09ahj-09fj-a9} may be rewritten on $\bar{\mathcal{E}}$ as
\begin{equation}
\Delta_{\bar{g}}\bar{u}+\frac{1}{2}\bar{g}(\bar{\nabla}\log|\eta|,\bar{\nabla}\bar{u})+\bar{c}\bar{u}=\bar{b}.
\end{equation}
By Proposition \ref{AEend-decay} the quotient end $\bar{\E}$ is AE of order $q$, and Remark \ref{poanfoinaopinhpohj} shows that $|\bar{\nabla}\log|\eta||_{\bar{g}}\in C^{0,\alpha}_{1+q}(\bar{\E})$. The desired result now follows from \cite{Lee}*{Corollary A.32}.
\end{proof}

Throughout the remainder of this section, it will be assumed for convenience that $M^4$ has a single end $\E$. The case of additional ends may be treated with minor modifications.

\subsection{Existence of positive solutions}
Within the setting of Theorem \ref{ALF-density} consider the equation with prescribed asymptotics
\begin{equation}\label{foinaiofnioangf}
\Delta_g u -\mathbf{c} u=0 \quad\text{ on } M^4,\quad\quad\quad u\rightarrow 1 \quad\text{ as }r\rightarrow\infty,
\end{equation}
where $\mathbf{c}\in C^{\infty}(M^4)\cap C^{2,\alpha}_4(\E)$ with $\alpha\in (0,1)$. To study this problem we will need a Sobolev inequality on the asymptotic end.

\begin{lemma}\label{sobolev} 
There exists a constant $C_*$ depending only on the geometry of $(\E,g)$, such that for any $U(1)$ invariant function $u\in C^{\infty}_c (\E)$ it holds that
\begin{equation}\label{fhaoifnhiqwnhh}
\left(\int_\mathcal{E}u^6  \right)^{1/3}\leq C_* \int_{\mathcal{E}} |\nabla u|_g^2 .
\end{equation} 
\end{lemma}

\begin{proof}
By Proposition \ref{AEend-decay}, the asymptotic end of the quotient space $(\bar{M}^3, \bar{g})$ is AE. Moreover, if $u$ is $U(1)$ invariant then its descent to the quotient space $\bar{u}$ satisfies $|\bar{\nabla}\bar{u}|_{\bar{g}}=|\nabla u|_{g}$, where $\bar{\nabla}$ denotes the gradient of $\bar{M}^3$. The desired inequality \eqref{fhaoifnhiqwnhh} then follows from the corresponding inequality \cite{schoen-yau1979}*{Lemma 3.1} in the AE setting.
\end{proof}

We may now establish the existence of positive solutions to \eqref{foinaiofnioangf} under additional hypotheses on the coefficient $\mathbf{c}$.

\begin{theorem}\label{conformal} 
Let $(M^4, g, \mathcal{E})$ be a complete ALF manifold having an almost free $U(1)$ action with respect to $\E$. Assume that $\mathbf{c}\in C^{\infty}(M^4)\cap C^{0,\alpha}_4(\E)$, $\alpha\in (0,1)$ is a $U(1)$ invariant function supported in $\E$ such that
\begin{equation}\label{fpoainofinaoinghha}
\left(\int_{\mathcal{E}}|\mathbf{c}|^{\frac{3}{2}} \right)^{\frac{2}{3}}\leq \frac{1}{2C_*}.
\end{equation}
Then there exists a smooth positive solution of \eqref{foinaiofnioangf} satisfying $u-(1+\frac{\mathcal{C}}{r})\in C^{2, \alpha}_{1+q'}(\E)$,
where $q'=\min\{1,q\}$, $\mathcal{C}=-\frac{1}{2\pi\ell \omega_2}\int_{M^4}\mathbf{c}u$,  and  $C_*$ is a constant from Lemma \ref{sobolev}.
\end{theorem}

\begin{remark}\label{aofnoianfoianbgoianoignoiagn}
The same conclusions hold if only the negative part of the coefficient function, $\min\{\mathbf{c},0\}$, satisfies condition \eqref{fpoainofinaoinghha}.
\end{remark}

\begin{proof}
Let $\{\Omega_i\}_{i=1}^{\infty}$ be an exhaustion of $M^4$ by precompact open sets, each of which is $U(1)$-invariant and has a smooth boundary $\partial \Omega_i\subset \mathcal{E}$.  For each $i$, the kernel of $\Delta_g -\mathbf{c}$ in $W^{1,2}_0(\Omega_i)$ is trivial. To see this, observe that if a function $w\in\ker(\Delta_g-\mathbf{c})\cap W^{1,2}_0(\Omega_i)$ then utilizing Lemma \ref{sobolev} and H\"{o}lder's inequality produces
\begin{align}\label{gradient-end}
\begin{split}
\int_{\Omega_i} |\nabla w|_g^2=-\int_{\Omega_i}\mathbf{c}w^2  &=-\int_{\Omega_i\cap \mathcal{E}} \mathbf{c}w^2 \\
&\leq \left(\int_\mathcal{E} |\mathbf{c}|^{\frac{3}{2}} \right)^{\frac{2}{3}} \left(\int_{\mathcal{E}} w^6 \right)^{\frac{1}{3}}\\
&\leq \frac{1}{2C_*} \left(\int_{\mathcal{E}} w^6 \right)^{\frac{1}{3}} 
\leq \frac{1}{2} \int_{\E} |\nabla w|_g^2 = \frac{1}{2} \int_{\Omega_i \cap\E} |\nabla w|_g^2 ,
\end{split}
\end{align}
and hence $w=0$ in $\Omega_i$. It follows by elliptic theory that there is then a unique smooth solution to the Dirichlet problem
\begin{equation}\label{Prob1}
\Delta_g  w_i-\mathbf{c} w_i=\mathbf{c} \quad \text{ in } \Omega_i , \quad\quad\quad w_i =0 \quad \text{ on }\partial\Omega_i .
\end{equation}
Moreover, the uniqueness ensures that $w_i$ is also $U(1)$-invariant, as $g$, $\mathbf{c}$, and $\Omega_i$ have this property.

We will now make uniform estimates for the sequence $\{w_i\}_{i=1}^{\infty}$. The same manipulations as in \eqref{gradient-end} yield
\begin{equation}
\int_{M^4} |\nabla w_i|_g^2=-\int_{\mathcal{E}} \mathbf{c}w_i^2- \int_\mathcal{E} \mathbf{c}w_i 
\leq \frac{1}{2C_*}\left(\int_{\mathcal{E}} w_i^6 \right)^{\frac{1}{3}}+  \left(\int_\mathcal{E} \mathbf{c}^{\frac{6}{5}} \right)^{\frac{5}{6}} \left(\int_{\mathcal{E}} w_i^6 \right)^{\frac{1}{6}}. 
\end{equation}
This together with Lemma \ref{sobolev} implies that  
\begin{equation}\label{initial-iteration}
\left(\int_{\mathcal{E}}w^{6}_i \right)^{\frac{1}{6}}\leq 2C_* \left(\int_\mathcal{E} \mathbf{c}^{\frac{6}{5}} \right)^{\frac{5}{6}}. 
\end{equation}
Note that the right-hand side is finite since $\mathbf{c}\in C^{0,\alpha}_4(\E)$. 
Nash-Moser iteration then provides a uniform $L^{\infty}$ bound on any fixed compact subset of the end $\E$, and from this, Schauder estimates give $C^{k,\alpha}$ control for any $k$. In particular, these estimates hold on level sets of the coordinate function $r$, and thus standard elliptic estimates imply that this control extends to the interior of $M^4$. In sum, the sequence is uniformly bounded in $C^{k,\alpha}$ on any compact subset of $M^4$. Therefore, a diagonal argument may be used to extract a convergent subsequence (with notation unchanged for convenience) 
$w_{i} \rightarrow w$. The limit function $w$ is $U(1)$ invariant due to the invariance of $w_i$. Furthermore it is smooth, and from Lemma \ref{Schauder1} together with \eqref{initial-iteration} it satisfies $w-\frac{\mathcal{C}}{r}\in C^{2,\alpha}_{1+q'}(\E)$, for some constant $\mathcal{C}$ where $q'=\min\{1,q\}$. By setting $u=1+w$, we find that this function satisfies equation \eqref{foinaiofnioangf} along with all other desired properties, except perhaps positivity.

To establish the positivity of $u$ we proceed by contradiction. Assume that $\min_{M^4} u<0$, and let $-\epsilon\in (\min_{M^4}u , 0)$ be a small regular value; the existence of such a regular value is a consequence of Sard's theorem. Consider the domain $\Omega_{\epsilon}=\{p\in M^4 \mid u(p)<-\epsilon\}$.  The properties of $u$ imply that this domain is a precompact open set, and since $\epsilon$ is a regular value its boundary $\partial\Omega_{\epsilon}$ is smooth. We may then multiply equation \eqref{foinaiofnioangf} by $u$ and integrate by parts, noting that the boundary term has an advantageous sign and applying similar arguments found in \eqref{gradient-end}, to conclude that $u$ must be constant inside $\Omega_{\epsilon}$. This leads to a contradiction, since $u=-\epsilon$ on the boundary while it must also achieve the minimum value (which differs from $-\epsilon$) inside $\Omega$. We conclude that $u\geq 0$ on $M^4$, and by the Harnack inequality it must in fact be strictly positive.
\end{proof}

\subsection{Proof for Theorem \ref{ALF-density}} 
The first step is to glue-in by hand a harmonically ALF end. Thus, in $\E$ consider a smooth 1-parameter family of radial cut-off functions $0\leq \phi_s \leq 1$ with $s>>0$ and satisfying the following properties: 
\begin{equation}
\phi_s(r) =1 \quad\text{ for } r\leq 2s,\quad\quad \phi_s(r)=0 \quad\text{ for }r\geq 4s,\quad\quad |\nabla^k \phi_s|_g\leq C_k s^{-k}\quad\text{ for } 2s\leq r\leq 4s,
\end{equation}
for some constants $C_k$ depending only on $k$; these functions are then extended trivally to the rest of $M^4$. Then on $M^4$ define the 1-parameter family of Riemannian metrics
\begin{equation}
g_s=(1-\phi_s)\left(1+\frac{m}{6r}\right)^2 g_0+\phi_s g, 
\end{equation}
where $m$ is the mass of $(\E,g)$. Note that with the help of \eqref{decay-scal}, these metrics and their scalar curvatures obey
\begin{equation}\label{pre-scal-ineq}
g_s=\left(1+\frac{m}{6r}\right)^2 g_0 \quad\quad\text{ and }\quad  \quad R_{g_s}=-\frac{1}{4\ell^2}\left(1+\frac{m}{6r}\right)^{-3}|d\tau|^2=O(r^{-4}) 
\quad\text{ for } r\geq 4s.
\end{equation}
In particular, the mass of $(\E,g_s)$ is $m$ for each $s$. Furthermore, this family of metrics is clearly $U(1)$ invariant and $(M^4, g_s)$ admits an almost free $U(1)$ action with respect to $\E$.

The next step is concerned with reinstating the nonnegative scalar curvature condition via a conformal change. In $\E$ define another smooth 1-parameter family of radial cut-off functions $0\leq \psi_s \leq1$ with $s>>0$ and satisfying the following properties: 
\begin{equation}
\psi_s(r) =0 \quad\text{ for } r\leq s,\quad\quad \psi_s(r)=1 \quad\text{ for }r\geq 2s,\quad\quad |\nabla^k \psi_s|_g\leq C_k s^{-k}\quad\text{ for } s\leq r\leq 2s.
\end{equation}
These functions are then extended trivally to the rest of $M^4$.
Consider the equations with prescribed asymptotics
\begin{equation}\label{foianoifnoipqnhh}
\Delta_{g_s} u_s-\frac{1}{6}\psi_s R_{g_s}u_s=0 \quad\text{ on }M^{4},\quad\quad\quad u_s \rightarrow 1 \quad\text{ as }r\rightarrow\infty.
\end{equation}
Observe that $\psi_s R_{g_s}\in C^{\infty}(M^4)\cap C^{0,\alpha}_4(\E)$ for any $\alpha\in (0,1)$, and these $U(1)$ invariant functions are supported in $\E$.
Moreover, by \eqref{ALF-metric-decay} we have
\begin{equation}\label{pre-est2}
 |R_{g_s}|\leq C r^{-2-q} \quad\text{ for } s\leq r\leq 4 s,
\end{equation}
and therefore with the help of \eqref{pre-scal-ineq} it holds that
\begin{align}\label{L-1.5}
\begin{split}
\int_{\E} |\psi_s R_{g_s}|^{\frac{3}{2}}&=\int_{s\leq r\leq 4s} |\psi_s R_{g_s}|^{\frac{3}{2}}+\int_{r\geq 4s} |\psi_s R_{g_s}|^{\frac{3}{2}}\\
&\leq C\left(\int_{s\leq r\leq 4s} r^{-\frac{3}{2}(2+q)}+\int_{r\geq 4s}r^{-6}\right) \leq s^{-q}+s^{-2}<\left(\frac{1}{2C_*}\right)^{\frac{3}{2}}
\end{split}
\end{align}
for $s$ sufficiently large, where $C$ is a constant independent of $s$ and $C_*$ is from Lemma \ref{sobolev}. We may now apply Theorem \ref{conformal} to find positive $U(1)$ invariant solutions $u_s$ to \eqref{foianoifnoipqnhh}, with $u_s-(1+\frac{\mathcal{C}_s}{r})\in C^{2, \alpha}_{1+q'}(\E)$,
where $q'=\min\{1,q\}$ and $\mathcal{C}_s=-\frac{1}{2\pi\ell \omega_2}\int_{M^4}\psi_s R_{g_s} u_s$.
It follows that the scalar curvature of $u_s^2g_s$ is nonnegative, that is
\begin{equation}
R_{u_s^2g_s}=u_s^{-3}(R_{g_s}-6u_s^{-1}\Delta_{g_s}u_s)=u_s^{-3}R_{g_{s}}(1-\psi_s)\geq 0,
\end{equation}
where we have used that $\psi_s=1$ for $r\geq 2s$.  Furthermore
\begin{equation}
\Big|\mathring{\nabla}^l \left(u_s^2 g_s-\left(1+\frac{m+6\mathcal{C}_s}{6r}\right)^2g_0 \right)\Big|_{g_0}=O(r^{-1-q'-l}), \quad\quad l=0,1,2,
\end{equation}
showing that $(\E,u_s^2 g_s)$ is harmonically ALF with mass $m_s =m+6\mathcal{C}_s$. 

It will now be shown that $\mathcal{C}_s$ tends to zero as $s$ goes to infinity. Using computations similar to \eqref{L-1.5} gives $\parallel\! \!\psi_s R_{g_s}\!\!\parallel_{L^{6/5}(\E)}=o(1)$ as $s\rightarrow\infty$, and hence from \eqref{initial-iteration} it follows that $u_s$ admits uniform pointwise bounds. In fact, these considerations combined with elliptic estimates may be used to find that for sufficiently large $s$, the metrics $u_s^2 g_s$ are pointwise $\varepsilon$-close to $g$ globally. Thus, to complete the proof and establish closeness of the masses, it suffices to prove that $|\int_{M^4}\psi_s R_{g_s}|=o(1)$.
To accomplish this, we will consider two different cases. In the first case $\int_{M^4} \psi_s R_{g_s}\geq 0$ which yields
\begin{align}
\begin{split}
|\int_{M^4} \psi_s R_{g_s}|&=\int_{M^4} \psi_s R_{g_s}=\int_{r\geq s} \psi_s R_{g_s}=\int_{s\leq r\leq 2s } \psi_s R_{g_s}+\int_{r\geq 2s}  R_{g_s}\\
&\leq \int_{s\leq r\leq 2s}R_{g_s}+\int_{r\geq 2s}  R_{g_s}= \int_{r\geq s} R_{g_s}\\
&=\lim_{r\rightarrow \infty}\int_{\mathcal{S}^3_{r}} \star_{g_0}\left(\text{div}_{g_0} g_s-d \text{tr}_{g_{0}} g_s \right)-\int_{\mathcal{S}^3_{s}}\star_{g_0} \left(\text{div}_{g_0} g_s-d \text{tr}_{g_{0}} g_s \right)
+ O(s^{1-2q})\\
&=o(1)+O(s^{1-2q}), 
\end{split}
\end{align} 
where to obtain the second line we have used that $R_{g_s}\geq 0$ for $r\in[s,2s]$, and the third line follows from \eqref{foaibfoiuabifo} and \eqref{mass-formula}.
In the second case $\int_{M^4} \psi_s R_{g_s}<0$, which after similar considerations produces
\begin{align}
\begin{split}
|\int_{M^4} \psi_s R_{g_s}|&=-\int_{r\geq s} \psi_s R_{g_s}=-\int_{s\leq r\leq 2s } \psi_s R_{g_s}-\int_{r\geq 2s}  R_{g_s}\leq |\int_{r\geq 2s}  R_{g_s}|\\
&\leq \Big|\lim_{r\rightarrow \infty}\int_{\mathcal{S}^3_{r}} \star_{g_0}\left(\text{div}_{g_0} g_s-d \text{tr}_{g_{0}} g_s \right) -\int_{\mathcal{S}^3_{2s}} \star_{g_0}\left(\text{div}_{g_0} g_s-d \text{tr}_{g_{0}} g_s \right)\Big|
+ O(s^{1-2q})\\
&=o(1)+O(s^{1-2q}).
\end{split}
\end{align} 
Since $q>1/2$, given $\varepsilon>0$ there exists an $s_{\varepsilon}>>0$ such that $|m-m_{s_{\varepsilon}}|<\varepsilon$. Therefore, setting $g'=u_{s_{\varepsilon}}^2g_{s_{\varepsilon}}$ and $f=u_{s_{\varepsilon}}\cdot\left(1+\frac{m}{6r}\right)$ gives the desired result.

\section{Positivity of the Mass} 
\label{sec5} \setcounter{equation}{0}
\setcounter{section}{5} 
 
In this section we will prove the inequality portion of Theorem \ref{B}. Consider a complete ALF manifold  $(M^4, g)$ with nonnegative scalar curvature, having an almost free $U(1)$ action with respect to a designated end $\E$. According to Proposition \ref{AEend-decay} the quotient space has a corresponding AE end, and as will be seen below the quotient space scalar curvature has a positivity property allowing for a conformal change to zero scalar curvature. The AE conformal metric has a mass no greater than that of the original ALF manifold (up to a positive multiple), and thus the desired ALF positive mass theorem will follow from reduction to the AE positive mass theorem with singularities. It will be shown that the singular points within the quotient are sufficiently mild, ensuring that the AE positive mass result is applicable.

\subsection{Scalar curvature positivity property} 
Let $\{p_1,\dots,p_k\}$ be the singular points of the $U(1)$ action. The principal orbit theorem implies that the quotient map 
\begin{equation}
\pi :M^4 \setminus\{p_1 ,\dots,p_k \}\rightarrow \bar{M}^3 \setminus \{\bar{p}_1 ,\dots,\bar{p}_k\}
\end{equation}
yields a principal $U(1)$ bundle and in particular gives a Riemannian submersion, where $\bar{p}_i=\pi(p_i)$ and $\bar{M}^3 =M^4 /U(1)$.
For any $p\in M^4 \setminus\{p_1 ,\dots,p_k \}$ and $b=\pi(p)$, we will denote by $\mathcal{V}_p$ the tangent space of the fiber $\pi^{-1}(b)$ at $p$, and we will denote by $\mathcal{H}_p$ the orthogonal complement to $\mathcal{V}_p$ in $T_p M^4$. This yields two distributions, the vertical $\mathcal{V}=\cup_p \mathcal{V}_p$ and the horizontal $\mathcal{H}=\cup_p\mathcal{H}_p$. Define (2,1)-tensor fields $\mathcal{T}$ and $\mathcal{A}$ on $M^4$ which evaluate on vector fields $E_1$, $E_2$ by
\begin{equation}
\mathcal{T}(E_1 ,E_2) =\mathcal{H}\nabla_{\mathcal{V}E_1}\mathcal{V}E_2 +\mathcal{V}\nabla_{\mathcal{V}E_1}\mathcal{H}E_2 ,\quad\quad\quad
\mathcal{A}(E_1 ,E_2) =\mathcal{H}\nabla_{\mathcal{H}E_1}\mathcal{V}E_2 +\mathcal{V}\nabla_{\mathcal{H}E_1}\mathcal{H}E_2,
\end{equation}
where the notation $\mathcal{V}E_1$ and $\mathcal{H}E_1$ represents projection onto the vertical and horizontal distributions, respectively. Note that for two vertical vectors $U$, $V$ at $p$ the horizontal vector $\mathcal{T}(U,V)$ is minus the second fundamental form of the fiber, and for two horizontal vectors $X$, $Y$ at $p$ a computation \cite{Besse}*{Proposition 9.24} shows that $\mathcal{A}(X,Y) =\frac{1}{2}\mathcal{V}[X,Y] $, so that it may be interpreted as a measure of the lack of integrability of the horizontal distribution. Since the fibers are 1-dimensional, their mean curvature vector is given by $\mathcal{N}=\mathcal{T}(U_1,U_1)$ where $U_1 =T/|T|_g$ and $T$ is the Killing field generator of the $U(1)$ symmetry.

Let $\bar{X}_i$, $i=1,2,3$ be a local orthonormal frame on $\bar{M}^3$, then there exists a unique set of horizontal orthonormal vectors $X_i$, $i=1,2,3$ such that $\pi_{*}(X_i)=\bar{X}_i$. Moreover $[T,X_i]\in \mathcal{V}$ since
\begin{equation}
\pi_{*}[T,X_i]=[\pi_{*}(T),\pi_{*}(X_i)]=0,
\end{equation}
and
\begin{equation}\label{fapjpajpfonjapogjpoqajg}
g([T,X_i],T)=g(\mathcal{L}_{T}X_i , T)=-g(X_i,\mathcal{L}_{T}T)=0,
\end{equation}
so that $\mathcal{L}_{T}X_i =0$ for $i=1,2,3$. With this frame we define the partial divergence operation for vector fields $E$ on $M^4$ by
\begin{equation}
\check{\delta}E=-\sum_{i=1}^{3}g(\nabla_{X_i}E,X_i).
\end{equation}
According to \cite{Besse}*{Corollary 9.37} the relation between the scalar curvature of $M^4$ and its quotient takes the form
\begin{equation}\label{ofinaoifnoiapoip}
\pi^{*}R_{\bar{g}}=R_g -R_{F}+|\mathcal{A}|^2+|\mathcal{T}|^2+|\mathcal{N}|^2 +2\check{\delta}\mathcal{N},
\end{equation}
where $R_{F}$ is the fiber scalar curvature. Note that $R_F =0$ since the fibers are 1-dimensional.

\begin{proposition}\label{mean-scalar-formula} 
Each individual quantity in \eqref{ofinaoifnoiapoip} is $U(1)$ invariant, and thus descends to the quotient. In particular, there exist functions $\bar{R}_{g}$ and $\bar{\mathcal{A}}^2$ as well as a vector field $\bar{\mathcal{N}}$ on $\bar{M}^3 \setminus \{\bar{p}_1,\dots,\bar{p}_k\}$, such that $\pi^* \bar{R}_g =R_g$, $\pi^* \bar{\mathcal{A}}^2 =|\mathcal{A}|^2$, and $\pi_{*}\mathcal{N}=\bar{\mathcal{N}}$. Moreover, the quotient space scalar curvature may be expressed as
\begin{equation}\label{fapijfpapp}
R_{\bar{g}}=\bar{R}_g +\bar{\mathcal{A}}^2 +2|\bar{\mathcal{N}}|^2_{\bar{g}} -2\mathrm{div}_{\bar{g}}\bar{\mathcal{N}}.
\end{equation}
\end{proposition}

\begin{proof}
First note that since the $U(1)$ action is by isometries, $R_g$ is $U(1)$ invariant. Consider now the tensor field $\mathcal{A}$ and observe that
\begin{align}\label{foanoifnpqaoinfpoiqnif}
\begin{split}
|\mathcal{A}|^2=\sum_{i=1}^{3}|\mathcal{A}(X_i , U_1)|_g^2 &=\sum_{i,j=1}^{3}g(\mathcal{A}(X_i,U_1),X_j)^2\\
&=\sum_{i,j=1}^{3}g(U_1,\mathcal{A}(X_i , X_j))^2 =\frac{1}{4}\sum_{i,j=1}^3 \left(\mathcal{V}[X_i , X_j]\right)^2.
\end{split}
\end{align}
By the Jacobi identity
\begin{equation}
\mathcal{L}_T [X_i, X_j]=[T,[X_i , X_j]]=-[X_i, [X_j , T]]-[X_j, [T,X_i]]=0,
\end{equation}
and thus $|\mathcal{A}|^2$ is $U(1)$ invariant. Next compute
\begin{equation}
|\mathcal{T}|^2 =\sum_{i=1}^{3}|\mathcal{T}(U_1 , X_i)|_g^2 =\sum_{i=1}^{3}\frac{g(\nabla_{T} X_i , T)^2}{|T|^4_g}
=\sum_{i=1}^{3}\frac{g(\nabla_{X_i} T , T)^2}{|T|^4_g}=\sum_{i=1}^{3}\left[X_i (\log |T|_g)\right]^2,
\end{equation}
where we have used that $[T,X_i]=0$. Notice that this function is also $U(1)$ invariant. Similarly
\begin{equation}
\mathcal{N}=\mathcal{T}(U_1,U_1)=\sum_{i=1}^3 \frac{g(\nabla_{T}T,X_i)}{|T|_g^2}X_i =-\sum_{i=1}^3 \frac{g(T,\nabla_{T}X_i)}{|T|_g^2}X_i
=-\sum_{i=1}^3 X_i (\log |T|_g) X_i ,
\end{equation}
so that
\begin{equation}\label{mean-local-T}
\pi_* \mathcal{N} =-\sum_{i=1}^3 \bar{X}_i (\log |T|_g) \bar{X}_i =-\bar{\nabla}\log |T|_g =:\bar{\mathcal{N}}
\end{equation}
where $\bar{\nabla}$ denotes the connection with respect to $\bar{g}$. In particular, $|\mathcal{T}|^2 =|\mathcal{N}|^2=|\bar{\mathcal{N}}|^2_{\bar{g}}$. Lastly
\begin{equation}
\check{\delta}\mathcal{N}=-\sum_{i=1}^{3} g(\nabla_{X_i}\mathcal{N},X_i)
=-\sum_{i=1}^{3} \bar{g}(\pi_* \nabla_{X_i}\mathcal{N},\pi_* X_i)=-\sum_{i=1}^{3}\bar{g}(\bar{\nabla}_{\bar{X}_i}\bar{\mathcal{N}},\bar{X}_i)
=-\mathrm{div}_{\bar{g}}\bar{\mathcal{N}}.
\end{equation}
The desired result now follows from \eqref{ofinaoifnoiapoip}.
\end{proof}

\begin{remark}\label{poanfoinaopinhpohj}
Let $\bar{r}$ denote the distance function in $\bar{M}^3$ from a singular point $\bar{p}_i$. Then by Remark \ref{singular-beh} and the proof of Proposition \ref{mean-scalar-formula}, we find that
\begin{equation}
|\bar{\mathcal{N}}|_{\bar{g}}=|\bar{\nabla}\log|T|_g|_{\bar{g}}=\frac{1}{\bar{r}}+O(1) \quad\text{ as } \bar{r}\rightarrow 0.
\end{equation}
Furthermore, the asymptotics in the end $\E$ may be obtained from \eqref{aonfoinh} and the discussion following this equation,
namely $|\bar{\mathcal{N}}|_{\bar{g}}\in C^{0,\alpha}_{1+q}(\E)$ for any $\alpha\in(0,1)$.
\end{remark}

\subsection{Conformal change on the quotient space}\label{foinfoainoiphghhag}
Theorem \ref{ALF-density} asserts that any ALF manifold having an almost free $U(1)$ action with respect to a distinguished end, can be approximated by one in which the end is harmonically ALF. In what follows, we will assume that there is a single end and that $(M^4, g, \E)$ is the approximation manifold in Theorem \ref{ALF-density}.
In particular, there exists a function $f\in C^{\infty}(M^4)$ and a compact set $\mathcal{K}\subset M^4$ such that 
\begin{equation}\label{conformal-def1}
g=f^2 g_0 \quad\text{ on }\E \setminus\mathcal{K},\quad\quad\quad f-\left(1+\frac{m}{6{r}}\right)\in C^{2, \alpha}_{1+q'}(\E),
\end{equation}
where $q'=\min\{1,q\}$, $\alpha\in(0,1)$, and $m$ is the mass of $(\E,g)$. Moreover, the scalar curvature of $g$ is nonnegative and has asymptotics
\begin{equation}
R_{g}=f^{-3}\left(R_{g_0}f-6\Delta_{g_0}f\right) \in C^{0, \alpha}_{3+q'}(\E). 
\end{equation}
The manifold $(M^4 ,g)$ admits an almost free $U(1)$ action with respect to $\E$ in which the Killing field generator $T$ agrees with the vector field $V$ associated with the model metric $g_0$. 
Combining this with Theorem \ref{local} yields the following asymptotics for the quotient $(\bar{M}^3,\bar{g},\bar{\E})$ near singular points $\bar{p}_i$, $i=1,\dots,k$ and at infinity 
\begin{equation}\label{Taylor-exp}
\bar{g}=d\bar{r}^2+ \frac{1}{4}\bar{r}^2 g_{S^2}+\pmb{\epsilon}_{\bar{r}}  \quad\text{ on } \bar{U}_i \setminus\{\bar{p}_i\},\quad\quad\quad 
\bar{g}-\left(1+\frac{m}{6{r}}\right)^2 \delta\in C^{2,\alpha}_{1+q'}(\bar{\E}),
\end{equation}
where $\pmb{\epsilon}_{\bar{r}}=O_2(\bar{r}^3)$. In particular
\begin{equation}\label{fpainpfinapinfpian}
R_{\bar{g}}=\frac{6}{\bar{r}^2}+O(\bar{r}^{-1}) \quad\text{ as }\bar{r}\rightarrow 0,\quad\quad\quad
R_{\bar{g}}=O(r^{-3-q'}) \quad \text{ as }r\rightarrow\infty.
\end{equation}
Furthermore, with the help of Remark \ref{poanfoinaopinhpohj} we find that in this setting
\begin{equation}\label{mean-scalar-special}
|\bar{\mathcal{N}}|_{\bar{g}}=\frac{1}{\bar{r}}+O(1) \quad\text{ as }\bar{r}\rightarrow 0,\quad\quad\quad
\bar{\mathcal{N}}=-\bar{\nabla}\log(\ell f)=\frac{m}{6 r^2}\partial_r +O(r^{-3}) \quad \text{ as }r\rightarrow\infty.
\end{equation}

\begin{theorem}\label{factor} 
Let $(M^4, g, \mathcal{E})$ have the properties described above. Then there exists a function $\bar{u}\in W^{1,2}_{loc}(\bar{M}^3)\cap C^{\infty}(\bar{M}^3 \setminus \{\bar{p}_1,\dots,\bar{p}_k\})$ which is positive away from the singular points $\bar{p}_i$ and satisfies the equation (weakly across singular points)
\begin{equation}
\Delta_{\bar{g}}\bar{u}-\frac{1}{8}R_{\bar{g}}\bar{u}=0 \quad\text{ on }\bar{M}^3,
\end{equation}
such that for some constants $c$ and $\bar{\mathcal{C}}$ the following asymptotics hold 
\begin{equation}
c^{-1}\bar{r}^{1/2}\leq \bar{u}\leq c \bar{r}^{1/2} \quad\text{ on } \bar{U}_i,\quad\quad\quad\quad\quad
\bar{u}-\left(1+\frac{\bar{\mathcal{C}}}{r}\right)\in C^{2, \alpha}_{1+q'}(\bar{\E}), 
\end{equation}
where
\begin{equation}\label{fpnapin0gjq9j9hh}
\bar{\mathcal{C}}\leq -\frac{1}{8\pi}\int_{\bar{M}^3}|\bar{\nabla}\bar{u}|_{\bar{g}}^2 \mathrm{dV}_{\bar{g}}.
\end{equation}
\end{theorem}



The first step in establishing this result consists of showing that a related differential operator on the quotient space is positive. This will be a consequence of the structure in O'Neill's scalar curvature formula \eqref{fapijfpapp}.
 

\begin{lemma}\label{positive1} 
For any function $\varphi\in C^{0,1}(\bar{M}^3)\cap C^{\infty}(\bar{M}^3 \setminus\{\bar{p}_1,\dots,\bar{p}_k\})$
which vanishes outside a compact subset of $\bar{M}^3$, it holds that
\begin{equation}\label{ineq0}
\int_{\bar{M}^3}\left(|\bar{\nabla}\varphi|_{\bar{g}}^2+\frac{1}{2} R_{\bar{g}} \varphi^2 \right)\geq 0. 
\end{equation}
\end{lemma}

\begin{proof} 
Let $r>0$ be sufficiently large, and let $\varepsilon>0$ be sufficiently small. Denote the region of $\bar{M}^3$ that lies within the coordinate sphere $S_{r}\subset\bar{\E}$ and outside the geodesic balls $B_{\varepsilon}(\bar{p}_i)$, by $\bar{M}^3_{\varepsilon,r}$. Then using Proposition \ref{mean-scalar-formula} combined with an integration by parts produces
\begin{align}
\begin{split}
\int_{\bar{M}^3_{\varepsilon,r}}\left( |\bar{\nabla} \varphi |_{\bar{g}}^2+\frac{1}{2}R_{\bar{g}}\varphi^2 \right)
&\geq \int_{\bar{M}^3_{\varepsilon,r}} \left(|\bar{\nabla}\varphi|_{\bar{g}}^2+|\bar{\mathcal{N}}|_{\bar{g}}^2 \varphi^2-(\mathrm{div}_{\bar{g}} \bar{\mathcal{N}}) \varphi^2 \right)\\
&= \int_{\bar{M}^3_{\varepsilon,r}} \left(|\bar{\nabla} \varphi|_{\bar{g}}^2 + |\bar{\mathcal{N}}|_{\bar{g}}^2\varphi^2 +2\varphi \bar{g}(\bar{\mathcal{N}}, \bar{\nabla} \varphi ) \right)
 - \int_{\partial\bar{M}^3_{\varepsilon,r}} \bar{g}(\bar{\mathcal{N}},\bar{\nu}) \varphi^2 \\ 
&=\int_{\bar{M}^3_{\varepsilon,r}} |\bar{\nabla} \varphi +\bar{\mathcal{N}}\varphi|_{\bar{g}}^2 +o(1),
\end{split}
\end{align}
where $\bar{\nu}$ is the unit outer normal to $\partial\bar{M}^3_{\varepsilon,r}$. In the last step, we have used \eqref{mean-scalar-special} to find that the boundary integrals may be made arbitrarily small by taking $\varepsilon\rightarrow 0$ and $r\rightarrow\infty$. Thus, the desired result follows by taking these limits since the starting integrand is globally integrable (in light of \eqref{fpainpfinapinfpian}).
\end{proof}

Although the coercivity inequality \eqref{ineq} can be used to produce solutions to the zero scalar curvature equation, these will not necessarily have the desired asymptotic behavior near the singular points $\bar{p}_i$. In order to achieve the desired asymptotics, we will conceal the singular points with a preemptive conformal change, similar to the approach of Schoen-Yau \cite{schoen-yau1981}*{pg. 257} .
Consider a positive function $\psi\in C^{\infty}(\bar{M}^3 \setminus\{\bar{p}_1,\dots,\bar{p}_k\})$ with $\psi=1$ on $\bar{\E}$, and $\psi=\bar{r}^{\frac{1}{2}}$ on each $\bar{U}_i \setminus \{\bar{p}_i \}$. Define a new metric $\tilde{g}=\psi^4\bar{g}$ on $\bar{M}^3$. Using \eqref{Taylor-exp} and the change of radial coordinate $\rho=\frac{1}{2}r^2$, this metric takes the following form near singular points
\begin{equation}
\tilde{g}=d\rho^2+\rho^2 g_{S^2}+\tilde{\pmb{\epsilon}}_{\rho},  
\end{equation}
where $\tilde{\pmb{\epsilon}}_{\rho}=O_2 (\rho^{\frac{5}{2}})$ is a family of symmetric 2-tensors on $S^2$.
This implies that $\tilde{g}$ is uniformly equivalent to the Euclidean metric near each $\bar{p}_i$. Moreover, the scalar curvature of this metric may be expressed in terms of the conformal Laplacian by
\begin{equation}\label{scal-conf}
R_{\tilde{g}}\!=\!-8\psi^{-6}L_{\bar{g}}\psi =\!-8\psi^{-5}\!\left(\!\Delta_{\bar{g}}\psi \!-\!\frac{1}{8}R_{\bar{g}}\psi \!\right)\!\!=\!O(\rho^{-\frac{3}{2}}) \text{ as }\rho\rightarrow 0,\quad\text{ } R_{\tilde{g}}=O(r^{-3-q'}) \text{ as }r\rightarrow\infty.
\end{equation}
Recall the transformation law for the conformal Laplacian, namely $L_{\tilde{g}}(\psi^{-1}\phi)=\psi^{-5}L_{\bar{g}}\phi$ for any smooth function $\phi$. We will use this together with \eqref{ineq0} to obtain an improved coercivity inequality with respect to the metric $\tilde{g}$.


\begin{lemma}\label{positive2} 
For any function $\phi\in C^{\infty}_c(\bar{M}^3)$ it holds that
\begin{equation}\label{ineq}
\int_{\bar{M}^3}\left(|\tilde{\nabla}\phi|_{\tilde{g}}^2+\frac{1}{8} R_{\tilde{g}} \phi^2 \right)\mathrm{dV}_{\tilde{g}}
\geq \frac{1}{2}\int_{\bar{M}^3}|\bar{\nabla}(\psi\phi)|_{\bar{g}}^2 \mathrm{dV}_{\bar{g}}\geq C_* \left(\int_{\bar{M}^3}\phi^6 \mathrm{dV}_{\tilde{g}}\right)^{1/3},
\end{equation}
for some constant $C_* >0$ depending on $(\bar{M}^3 ,\bar{g})$.
\end{lemma}

\begin{proof}
Using Lemma \ref{positive1} and notation from its proof we have
\begin{align}\label{5.28}
\begin{split}
\int_{\bar{M}^3_{\varepsilon,r}}\left( |\tilde{\nabla} \phi |_{\tilde{g}}^2+\frac{1}{8}R_{\tilde{g}}\phi^2 \right)\mathrm{dV}_{\tilde{g}}
&= -\int_{\bar{M}^3_{\varepsilon,r}}\phi L_{\tilde{g}}\phi \mathrm{dV}_{\tilde{g}} +\int_{\partial\bar{M}^3_{\varepsilon,r}}\phi\tilde{\nu}(\phi)\mathrm{dA}_{\tilde{g}}\\
&=-\int_{\bar{M}^3_{\varepsilon,r}}(\psi\phi) L_{\bar{g}}(\psi\phi) \mathrm{dV}_{\bar{g}} +\int_{\partial\bar{M}^3_{\varepsilon,r}}\phi\tilde{\nu}(\phi)\mathrm{dA}_{\tilde{g}}\\
&=\int_{\bar{M}^3_{\varepsilon,r}}\left(|\bar{\nabla}(\psi\phi)|_{\bar{g}}^2 +\frac{1}{8}R_{\bar{g}}(\psi\phi)\right)\mathrm{dV}_{\bar{g}}\\
&\qquad+\int_{\partial\bar{M}^3_{\varepsilon,r}}\phi\tilde{\nu}(\phi)\mathrm{dA}_{\tilde{g}}
-\int_{\partial\bar{M}^{3}_{\varepsilon,r}}(\psi\phi)\bar{\nu}(\psi\phi)\mathrm{dA}_{\bar{g}}\\
&\geq \frac{1}{2}\int_{\bar{M}^3_{\varepsilon,r}}|\bar{\nabla}(\psi\phi)|_{\bar{g}}^2 \mathrm{dV}_{\bar{g}} +o(1),
\end{split}
\end{align}
where $\tilde{\nu}$ and $\bar{\nu}$ are the unit outer normals to $\partial\bar{M}^3_{\varepsilon,r}$ with respect to $\tilde{g}$ and $\bar{g}$, respectively. The desired first inequality is then obtained by taking the limit as $\varepsilon\rightarrow 0$ and $r\rightarrow\infty$, since the starting integrand is globally integrable. 

The desired second inequality follows from the Sobolev inequality for the singular manifold $(\bar{M}^3 ,\bar{g})$ since
\begin{equation}
\int_{\bar{M}^3}|\bar{\nabla}(\psi\phi)|_{\bar{g}}^2 \mathrm{dV}_{\bar{g}}\geq C_{*}\left(\int_{\bar{M}^3}(\psi\phi)^6 \mathrm{dV}_{\bar{g}}\right)^{1/3}=C_* \left(\int_{\bar{M}^3}\phi^6 \mathrm{dV}_{\tilde{g}}\right)^{1/3}.
\end{equation}
Note that the Sobolev inequality on this space may be established in the same manner as \cite{schoen-yau1979}*{Lemma 3.1}, since the metric $\bar{g}$ is uniformly equivalent to the Euclidean metric near each singular point. In particular, near these points
\begin{equation}
C |\nabla_{\delta}(\psi\phi)|_{\delta}^2 \geq |\bar{\nabla}(\psi\phi)|_{\bar{g}}^2 \leq C^{-1} |\nabla_{\delta}(\psi\phi)|_{\delta}^2,\quad\quad\quad
C\mathrm{dV}_{\delta}\geq \mathrm{dV}_{\bar{g}}\geq C^{-1}\mathrm{dV}_{\delta},
\end{equation}
for some constant $C>0$.
\end{proof}

\begin{proof}[Proof of Theorem \ref{factor}]
As in the proof of Theorem \ref{conformal}, inequality \eqref{ineq} allows us to find weak solutions $v_i \in W_{0}^{1,2}(\tilde{\Omega}_i)\cap C^{\infty}(\tilde{\Omega}_i \setminus\{\bar{p}_1,\dots,\bar{p}_k\})$ of the Dirichlet problems
\begin{equation}
\Delta_{\tilde{g}}  v_i-\frac{1}{8}R_{\tilde{g}} v_i=\frac{1}{8}R_{\tilde{g}} \quad \text{ in } \tilde{\Omega}_i , \quad\quad\quad v_i =0 \quad \text{ on }\partial\tilde{\Omega}_i ,
\end{equation}
where $\{\tilde{\Omega}_i\}_{i=1}^{\infty}$ is a sequence of precompact exhaustion domains for $\bar{M}^3$. In particular, the weak solution property combined with Lemma \ref{positive2} and H\"{o}lder's inequality produces
\begin{align}
\begin{split}
C_*\left(\int_{\tilde{\Omega}_i} v_i^6 \mathrm{dV}_{\tilde{g}}\right)^{1/3} &\leq \int_{\tilde{\Omega}_i}\left(|\tilde{\nabla}v_i|_{\tilde{g}}^2+\frac{1}{8} R_{\tilde{g}} v_i^2 \right)\mathrm{dV}_{\tilde{g}}\\
&=-\int_{\tilde{\Omega}_i}\frac{1}{8}R_{\tilde{g}} v_i \mathrm{dV}_{\tilde{g}}
\leq\frac{1}{8}\left(\int_{\tilde{\Omega}_i}|R_{\tilde{g}}|^{\frac{6}{5}}\mathrm{dV}_{\tilde{g}}\right)^{5/6}
\left(\int_{\tilde{\Omega}_i}v_i^6\mathrm{dV}_{\tilde{g}}\right)^{1/6}.
\end{split}
\end{align}
It follows that $v_i$ is uniformly controlled in $L^6(\tilde{\Omega}_i)$, since $R_{\tilde{g}}\in L^{2-\varsigma}(\bar{M}^3)$ for any small $\varsigma>0$ by \eqref{scal-conf}. Standard elliptic theory now yields uniform pointwise bounds for $v_i$ on $\bar{\E}\cap\tilde{\Omega}_i$. Moreover, if $\mathbf{r}$ is a smooth extension of the radial coordinate function in the asymptotic end to all of $\bar{M}^3$, with $\mathbf{r}=1$ on $\bar{M}^3 \setminus \bar{\E}$, then the weak solution property produces
\begin{align}
\begin{split}
\int_{\tilde{\Omega}_i}|\tilde{\nabla}v_i|_{\tilde{g}}^2 \mathrm{dV}_{\tilde{g}}&=\frac{1}{8}\int_{\tilde{\Omega}_i}(v_i -v_i^2)R_{\tilde{g}}\mathrm{dV}_{\tilde{g}}\\
&\leq \frac{1}{8}\left(\int_{\tilde{\Omega}_i}|R_{\tilde{g}}|^{\frac{3}{2}}\mathbf{r}^{\frac{3}{2}(1+\varsigma)}\mathrm{dV}_{\tilde{g}}\right)^{2/3}\left(\int_{\tilde{\Omega}_i}(|v_i|+v_i^2)^{3}\mathbf{r}^{-3(1+\varsigma)}\mathrm{dV}_{\tilde{g}}\right)^{1/3}.
\end{split}
\end{align}
The right-hand side is uniformly bounded in light of the decay \eqref{scal-conf}, and thus $v_i$ is controlled in $W^{1,2}_{loc}(\tilde{\Omega}_i)$ independent of $i$. Higher order uniform estimates may be obtained by boot-strap on compact subsets away from singular points. A diagonal argument may now be employed to extract a convergence subsequence (denoted without change) $v_i\rightarrow v$. The limit function is globally bounded, lies in $W^{1,2}_{loc}(\bar{M}^3)$, is smooth away from points $\bar{p}_i$, and admits asymptotics $v-\frac{\bar{\mathcal{C}}}{r}\in C^{2,\alpha}_{1+q'}(\bar{\E})$ by \cite{Lee}*{Corollary A.37} 
for some constant $\bar{\mathcal{C}}$. Define $u=1+v$ and observe that it satisfies the equation (weakly across singular points)
\begin{equation}\label{fpaoinapinfpioaqnhpg}
\Delta_{\tilde{g}}u-\frac{1}{8}R_{\tilde{g}}u=0 \quad\text{ on }\bar{M}^3, \quad\quad\quad\quad
u-\left(1+\frac{\bar{\mathcal{C}}}{r}\right)\in C^{2, \alpha}_{1+q'}(\bar{\E}).
\end{equation}
The desired function is then obtained by setting $\bar{u}=\psi u$.

It remains to show that $u$ is positive, and to estimate $\bar{\mathcal{C}}$. We first note that the regularity of $u$ near singular points may be improved. Namely, by writing equation \eqref{fpaoinapinfpioaqnhpg} in normal coordinates at $\bar{p}_i$ as $\tilde{g}^{ij}\partial_{ij}u=F$, we may boot-strap the regularity of the inhomogeneous term to find $F\in L^{2-\varsigma}_{loc}(\bar{M}^3)$ and hence by the interior $L^p$-estimates it follows that $u\in W^{2,2-\varsigma}_{loc}(\bar{M}^3)$ for any $\varsigma>0$. In particular, $u$ is continuous at singular points. Assume that $\min_{\bar{M}^3} u<0$, and let $-\epsilon\in (\min_{\bar{M}^3}u , 0)$ be a small regular value which is distinct from the values $u(\bar{p}_i)$, $i=1,\dots,k$; the existence of such a regular value is a consequence of Sard's theorem. Consider the domain $\Omega_{\epsilon}=\{p\in \bar{M}^3 \mid u(p)<-\epsilon\}$.  The properties of $u$ imply that this domain is a precompact open set, and since $\epsilon$ is a regular value its boundary $\partial\Omega_{\epsilon}$ is smooth. We may then apply Lemma \ref{positive2} with an approximating sequence $\phi_n \in C^{\infty}_c(\Omega)$ converging to $u+\varepsilon$ in $W^{2,2-\varsigma}(\Omega)$ and converging smoothly away from singular points and $\partial\Omega$, to find
\begin{align}
\begin{split}
C_{*}\left(\int_{\Omega}(u+\epsilon)^6\mathrm{dV}_{\tilde{g}}\right)^{1/3} \!\!\!\!&\leq \lim_{n\rightarrow\infty}\int_{\Omega}\left(|\tilde{\nabla}\phi_n |_{\tilde{g}}^2 +\frac{1}{8}R_{\tilde{g}}\phi_n^2 \right)\mathrm{dV}_{\tilde{g}}\\
&= \lim_{n\rightarrow\infty}\left[\int_{\Omega}\phi_n\left(\!-\Delta_{\tilde{g}}\phi_n +\frac{1}{8}R_{\tilde{g}}\phi_n \right)\mathrm{dV}_{\tilde{g}}\!+\!\int_{\partial\Omega}\phi_n \tilde{\nu}(\phi_n) \mathrm{dA}_{\tilde{g}}\right]\\
&=\int_{\Omega}(u+\epsilon)\left(\!-\Delta_{\tilde{g}}u +\frac{1}{8}R_{\tilde{g}}(u+\epsilon) \right)\mathrm{dV}_{\tilde{g}}\!+\!\int_{\partial\Omega}\!(u+\epsilon) \tilde{\nu}(u) \mathrm{dA}_{\tilde{g}}=O(\epsilon),
\end{split}
\end{align}
where $\tilde{\nu}$ is the unit outer normal to $\partial\Omega$. By taking a sequence of $\epsilon\rightarrow 0$ a contradiction arises. We conclude that $u\geq 0$ globally, and since $R_{\tilde{g}}\in L^{2-\varsigma}(\bar{M}^3)$ we may deduce from the Harnack-Moser inequality \cite{Stamp}*{Theorem 8.1} that $u$ is strictly positive everywhere.

Consider now the monopole at infinity. Let $B_{\varepsilon}$ be an $\varepsilon$-normal neighborhood around the singular points, and let $\varphi_n \in C^{\infty}_c(\bar{M}^3)$ be an approximating sequence converging to $u$ in $W^{1,2}(\bar{M}^3)\cap W^{2,2-\varsigma}_{loc}(\bar{M}^3)$ and smoothly away from singular points, then by Lemma \ref{positive2} it holds that
\begin{align}\label{5.36}
\begin{split}
\frac{1}{2}\int_{\bar{M}^3}|\bar{\nabla}(\psi u)|^2_{\bar{g}}\mathrm{dV}_{\bar{g}} &\leq \lim_{n\rightarrow\infty}\left(\int_{B_{\varepsilon}}+\int_{\bar{M}^3 \setminus B_{\varepsilon}}\right)\left(|\tilde{\nabla}\varphi_n |_{\tilde{g}}^2 +\frac{1}{8}R_{\tilde{g}}\varphi_n^2 \right)\mathrm{dV}_{\tilde{g}}\\
&= \lim_{n\rightarrow\infty}\left[\int_{B_{\varepsilon}}\varphi_n\left(-\Delta_{\tilde{g}}\varphi_n +\frac{1}{8}R_{\tilde{g}}\varphi_n \right)\mathrm{dV}_{\tilde{g}}+\int_{\partial B_{\varepsilon}}\varphi_n \tilde{\nu}(\varphi_n) \mathrm{dA}_{\tilde{g}}\right]\\
&\quad\quad +\lim_{n\rightarrow\infty}\int_{\bar{M}^3 \setminus B_{\varepsilon}}\left(|\tilde{\nabla}\varphi_n |_{\tilde{g}}^2 +\frac{1}{8}R_{\tilde{g}}\varphi_n^2 \right)\mathrm{dV}_{\tilde{g}}\\
&=\int_{\partial B_{\varepsilon}}u \tilde{\nu}(u) \mathrm{dA}_{\tilde{g}}
+\int_{\bar{M}^3 \setminus B_{\varepsilon}}\left(|\tilde{\nabla}u |_{\tilde{g}}^2 +\frac{1}{8}R_{\tilde{g}}u^2 \right)\mathrm{dV}_{\tilde{g}}\\
&=\lim_{r\rightarrow\infty}\int_{S_r}u\tilde{\nu}(u)\mathrm{dA}_{\tilde{g}}=-4\pi\bar{\mathcal{C}},
\end{split}
\end{align}
which yields the desired inequality \eqref{fpnapin0gjq9j9hh}.
\end{proof}

\subsection{Proof of the inequality in Theorem \ref{B}}
Let $(M^4, g,\E)$ be a complete ALF manifold with nonnegative scalar curvature, mass $m$, and having an almost free $U(1)$ action with respect to end $\mathcal{E}$. It will be assumed that there is a single end for convenience; the case of multiple ends may be treated similarly. Given $\varepsilon>0$, Theorem \ref{ALF-density} provides an approximation manifold $(M^4, g' ,\E)$ with nonnegative scalar curvature, a mass satisfying $|m -m'|<\varepsilon$, and having an almost free $U(1)$ action with respect to $\E$ in which the Killing field generator agrees with the vector field $V$ associated with the model metric $g_0$. Moreover,
there exists a function $f\in C^{\infty}(M^4)$ and a compact set $\mathcal{K}\subset M^4$ such that 
\begin{equation}\label{conformal-def11}
g'=f^2 g_0 \quad\text{ on }\E \setminus\mathcal{K},\quad\quad\quad f-\left(1+\frac{m'}{6{r}}\right)\in C^{2, \alpha}_{1+q'}(\E),
\end{equation}
where $q'=\min\{1,q\}$, $\alpha\in(0,1)$. By Proposition \ref{AEend-decay} the quotient $(\bar{M}^3,\bar{g},\bar{\E})$ of the approximation manifold is AE, and \eqref{fpainpfinapinfpian} shows that its scalar curvature is integrable, $R_{\bar{g}}\in L^{1}(\bar{M}^3)$. Then Corollary \ref{quo-sep} yields the relation between the masses of these two manifolds
\begin{equation}
m' =4\bar{m}+\lim_{r\rightarrow \infty}\frac{1}{2\pi}\int_{\bar{S}_{r} } \langle\bar{\mathcal{N}},\bar{\nu}\rangle =4\bar{m} +\frac{m'}{3},
\end{equation}
where we have used \eqref{mean-scalar-special}. Now apply Theorem \ref{factor} to find a conformal metric $\hat{g}=\bar{u}^4 \bar{g}$ on $\bar{M}^3$ with zero scalar curvature such that $(\bar{\E},\hat{g})$ is AE; its mass will be denoted by $\hat{m}$. Furthermore, the conformal factor has the following expansion
\begin{equation}
\bar{u}=1+\frac{\bar{\mathcal{C}}}{r}+O(r^{-\frac{3}{2}})\quad\text{ as } r\rightarrow\infty,
\end{equation}
where $\bar{\mathcal{C}}\leq 0$. Since $\hat{g}\in W^{1,6-\varsigma}_{loc}(\bar{M}^3)$ for any $\varsigma>0$ in light of the discussion from Section \ref{foinfoainoiphghhag}, and the singular set is discrete, we may apply the AE positive mass theorem of \cite{ShiTamSing}*{Theorem 7.2} to find that $\hat{m}\geq 0$. Moreover, a calculation shows that $\hat{m}=\bar{m}+2\bar{\mathcal{C}}$,
and hence
\begin{equation}
0\leq \hat{m}=\bar{m}+2\bar{\mathcal{C}}\leq \bar{m}=\frac{m'}{6}\leq \frac{m+\varepsilon}{6}.
\end{equation}
Since $\varepsilon$ is arbitrary, we conclude that $m\geq 0$.

\begin{remark}\label{angle}
For later use, we note here that by combining the computations in \eqref{5.28} and \eqref{5.36} an alternate
expression for the monopole is given by
\begin{equation}
-4\pi\bar{\mathcal{C}}=\int_{\bar{M}^3}\left(|\bar{\nabla}\bar{u}|_{\bar{g}}^2 +\frac{1}{8}R_{\bar{g}}\bar{u}^2 \right)\mathrm{dV}_{\bar{g}}.
\end{equation}
Therefore, from the proof of Theorem \ref{B}, the mass of the approximating manifold admits the lower bound
\begin{equation}
m' \geq -12\bar{\mathcal{C}}=\frac{3}{\pi}\int_{\bar{M}^3}\left(|\bar{\nabla}\bar{u}|_{\bar{g}}^2 +\frac{1}{8}R_{\bar{g}}\bar{u}^2 \right)\mathrm{dV}_{\bar{g}}.
\end{equation}
\end{remark}

\section{Rigidity of the Mass} 
\label{sec6} \setcounter{equation}{0}
\setcounter{section}{6}

The purpose of this section is to prove the case of equality statement in Theorem \ref{B}. Since the proof of the corresponding inequality relied on a density result (Theorem \ref{ALF-density}), it is difficult to use those arguments to establish rigidity.
We will therefore employ an alternative approach based on harmonic coordinates, and an associated mass formula involving these functions that was originally observed by Bartnik in the AE setting, in his proof of \cite{bartnik1986}*{Theorem 4.4}.

\subsection{Harmonic coordinates}
Consider a complete ALF manifold $(M^4 ,g,\E)$ with almost free $U(1)$ action. Recall that the metric on $\E$ may be expressed in Riemannian submersion format $g=\pi^* \bar{g}+\frac{\eta^2}{|\eta|^2}$, 
where $\eta=g(T,\cdot)$ is the dual 1-form to the $U(1)$ generator $T$. According to
Proposition \ref{AEend-decay}, there exists a coordinate system $\bar{x}=(\bar{x}^1, \bar{x}^2, \bar{x}^3)$ on the quotient end $(\bar{\E},\bar{g})$ that yields an AE structure. By pulling these functions back to $\E$ and denoting them by $x^i =\pi^* \bar{x}^i$ we obtain
\begin{equation}\label{098}
\pi^* \bar{g} =\bar{g}_{ij} dx^i dx^j =(\delta_{ij}+O_2(r^{-q}))dx^i dx^j,
\end{equation}
where $r=|x|$. Combining this with the flow parameter $t$ associated with the vector field $T$, yields local coordinates $(x,t)$ on $\E$ such that
\begin{equation}\label{0987}
\frac{\eta}{|\eta|}=\ell (1+O_2(r^{-q}))dt +A_i dx^i,\quad\quad A_i =O_2(r^{-q}).
\end{equation}
These observations yield an ALF structure in which $\Delta_g x^i =O_1(r^{-1-q})$ for $i=1,2,3$.

\begin{lemma}\label{exist-harmoinic}
Let $(M^4, g,\E)$ be a complete ALF manifold having an almost free $U(1)$ action with respect its single end $\E$. Assume that the
order of decay of this end satisfies $q\in(\frac{1}{2},1]$. Then there exist $U(1)$-invariant functions $y^i \in C^{\infty}(M^4)$, $i=1,2,3$ such that
\begin{equation}\label{harm-est-cord}
\Delta_g y^i=0 \quad\text{ on }M^4,\quad \qquad\quad\quad y^i -x^i \in C^{2,\alpha}_{q'-1}(\E),
\end{equation}
for any $q' <q$ and $\alpha\in(0,1)$.
\end{lemma}

\begin{proof} 
Since $\Delta_g x^i =O_1(r^{-1-q})$, the existence of smooth functions $y^i$ satisfying \eqref{harm-est-cord} is guaranteed by the proof of \cite{CLSZ}*{Proposition 4.12}, in the AF case. The proof relies solely on the asymptotic analysis carried out in \cite{minerbe}*{Section 2}, which applies in the general ALF regime. Thus, the desired existence result is valid in the current setting. 

It remains to show that the $y^i$ are $U(1)$-invariant. To do this, we will first show that the solution of \eqref{harm-est-cord} is unique up to addition of constants.
Note that according to the proof of \cite{CLSZ}*{Proposition 4.12} these functions lie in a weighted Sobolev space, namely $y^i -x^i \in H^{2}_{\delta}(M^4)$ for any $\delta>\frac{5}{2}-q$; here we are using the notation and definition of such spaces from \cite{minerbe}*{Section 2}. If $y^i_{*}$ is another solution of \eqref{harm-est-cord}, then $y^i - y^i_*$ is harmonic and lies in $H^2_{\delta}(M^4)$. By applying \cite{minerbe}*{Proposition 4} with $u=y^i -y^i_{*}$, $f=0$, and $\delta'= 1-\varsigma$ for arbitrarily small $\varsigma>0$, it follows that $y^i -y^i_*=v^i + c^i$ for some functions $v^i \in L^2_{\delta'}(M^4)$ and some constants $c^i$, $i=1,2,3$. Since $v^i$ is harmonic and lies in $L^2_1(M^4)$, it must in fact be zero as is shown in the proof of \cite{minerbe}*{Corollary 2}. Therefore, $y^i$ and $y^i_*$ differ at most by a constant. Since $x^i$ is $U(1)$-invariant, we may set $y^i_* =\varphi_t^* y^i$ where $\varphi_t$ is the flow for $T$, to find $y^i -\varphi_t^* y^i =c^i_t$ for a 1-parameter family of constants $c^i_t$. By taking the time derivative of this equation and using the decay of \eqref{harm-est-cord}, we have $\frac{dc^i_t}{dt}=-T(y^i)=O(r^{-q'})$. It follows that $\frac{dc^i_t}{dt}=0$ after letting $r\rightarrow \infty$, and hence $T(y^i)=0$ showing that the solutions are $U(1)$-invariant.
\end{proof}

Asymptotically linear harmonic functions have been used to obtain formulae for the ADM mass in the AE setting by Bartnik, in the proof of \cite{bartnik1986}*{Theorem 4.4}. This was extended to the AF setting by Chen-Liu-Shi-Zhu \cite{CLSZ}*{Proposition 4.12}, and Minerbe used related harmonic 1-forms to find a different definition of mass with a similar formula \cite{minerbe}*{(26)}. Here we will generalize the Chen-Liu-Shi-Zhu result to the ALF case.




\begin{lemma}\label{harmonic}  
Let $(M^4, g,\E)$ be a complete ALF manifold having an almost free $U(1)$ action with respect its single end $\E$. Assume that the
order of decay of this end satisfies $q\in(\frac{1}{2},1]$, and
\begin{equation}\label{Ricci-decay}
|\Ric(g)|_{g_0}+ r|\mathring{\nabla}\Ric(g)|_{g_0}=O(r^{-3-\epsilon}) \quad\text{ as }r\rightarrow\infty
\end{equation}
for some $\epsilon\in(0,1)$. If $y^i$, $i=1,2,3$ are the functions constructed in Lemma \ref{exist-harmoinic}, then the mass of $\E$ is given by
\begin{equation}
m=\frac{1}{12 \pi^2 \ell}\sum_{i=1}^3 \int_{M^4}\left( |\nabla^2 y^i|^2_g+ \Ric (\nabla y^i, \nabla y^i) \right).
\end{equation}
\end{lemma}

\begin{proof} 
Observe that $(y,t)$ form local coordinates on $\E$, and with the aid of \eqref{098} and \eqref{0987} the metric may be expressed as
\begin{equation}
g=(\delta_{ij}+O_2(r^{-q'}))dy^i dy^j +\left(\ell(1+O_2(r^{-q'}))dt +B_i dy^i \right)^2,\quad\quad B_i =O_2(r^{-q'}),
\end{equation}
with $r=|y|$ and $q'\in(\frac{1}{2},q)$. Define functions $g_{ij}=g(\partial_{y^i},\partial_{y^j})$ and $g^{ij}=g(dy^i, dy^j)$, then by Bochner's formula
\begin{equation}\label{equ2}
\frac{1}{2} \Delta_g g^{ii}=\frac{1}{2}\Delta_g|\nabla y^i|_g^2=|\nabla^2 y^i|^2 + \Ric( \nabla y^i, \nabla y^i) \in  C^{0,\alpha}_{2+2q_1}(\E)
\end{equation}
where $q_1=\min\{q',(1+\epsilon)/2\}$, $\alpha\in(0,1)$, and we have used \eqref{harm-est-cord} as well as \eqref{Ricci-decay}. 
Therefore, Lemma \ref{Schauder1} implies that $g^{ij}-(\delta_{ij}-\frac{c_{ij}}{r}) \in C^{2, \alpha}_{1+q_2}(\E)$ for some set of constants $c_{ij}=c_{ji}$,
where $q_2=\min\{q, 2q_1-1\}>0$. 
By performing a rotation of the $y$-coordinates if necessary, it may be assumed without loss of generality that $c_{ij}=\delta_{ij}c_i$ for some constants $c_i$, $i=1,2,3$. The components of the metric now satisfy the following asymptotics
\begin{equation}\label{metric-under-harmonic}
g_{ij}=\left(1+\frac{c_i}{r}\right)\delta_{ij}+O_2(r^{-1-q_2}),\quad 
 g_{4i}=g(\ell^{-1}\partial_t, \partial_{y^i})=O_2(r^{-q'}),\quad g_{44}
 =1+O_2(r^{-q'}).
\end{equation} 
Moreover, the unit outer normal to the surface $\mathcal{S}^3_{r}$ admits the asymptotics
$\nu=\tfrac{y^i}{r}\partial_{y^i}+O(r^{-q'})$. It follows that the mass flux density takes the form  
\begin{equation}
\sum_{a,b=1}^{4}(g_{ab,a}-g_{aa,b})\nu^b
=\sum_{i,j=1}^{3}({g}_{ij,i}-{g}_{ii,j}){\nu}^j +\sum_{b=1}^{4}(g_{4b,4}-g_{44,b})\nu^b +O(r^{-2-q_2}).
\end{equation}

Computations similar to those in the proof of Corollary \ref{quo-sep} produce
\begin{equation}
\sum^3_{i, j=1} (g_{ij,i}-g_{ii,j})\nu^j =\sum_{i=1}^3c_ir^{-2}-\sum^3_{i=1}
c_i (\nu^i)^2r^{-2}+O(r^{-2-q_2}),
\end{equation}
and
\begin{equation}\label{f0rhq09hryg}
\sum_{b=1}^{4}(g_{4b,4}-g_{44,b})\nu^b=\ell^{-2}\partial_t g(\partial_t ,\nu) -\sum^3_{i=1}\ell^{-2}\nu^i\partial_{y^i} g(\partial_t,\partial_t)+ O(r^{-2-q_2}). 
\end{equation}
Furthermore, using direct calculation together with \eqref{metric-under-harmonic}, as well as $\det g=g(\partial_t ,\partial_t)\det g_{ij}$, and the fact that $y^i$ is harmonic yields
\begin{align}
\begin{split}
0=\Delta_g y^i &= \frac{1}{\sqrt{\det g}}\partial_{a}\left(g^{ai}\sqrt{\det g}\right)\\
&=\frac{1}{2}\ell^{-2}\partial_{y^i}g(\partial_t ,\partial_t)+\ell^{-1}\partial_t g^{4i}
+c_i\nu^ir^{-2}-\frac{\nu^i}{2}\sum^3_{j=1}c_jr^{-2}+O(r^{-2-q_2}).
\end{split}
\end{align}
Solving for the first term on the right-hand side and inserting this into \eqref{f0rhq09hryg} gives
\begin{equation}
\sum_{b=1}^{4}(g_{4b,4}-g_{44,b})\nu^b=\ell^{-2}\partial_t g(\partial_t,\nu) +2\partial_t \left(\sum_{i=1}^3 \nu^i g^{4i}\right)
+2\sum_{i=1}^3c_i(\nu^i)^2r^{-2}-\sum^3_{i=1}c_ir^{-2}+O(r^{-2-q_2}).
\end{equation}
Noting that the terms involving $\partial_t$ integrate to zero in the definition of mass \eqref{massdef}, it follows that
\begin{equation}
m=\frac{1}{3}\sum_{i=1}^3 c_i .
\end{equation}
On the other hand, integrating the Bochner formula \eqref{equ2} by parts and employing \eqref{metric-under-harmonic} leads to
\begin{equation}
\frac{1}{3}\sum_{i=1}^3 c_i= \frac{2}{3}\sum_{i=1}^3\lim_{r\rightarrow\infty} \frac{1}{2\pi \ell\omega_2}\int_{\mathcal{S}^3_r}  \frac{1}{2}\nu(g^{ii})=\frac{2}{3}\sum_{i=1}^3 \frac{1}{2\pi \ell \omega_2 }\int_{M^4} \left(|\nabla^2 y^i|_g^2+ \Ric(\nabla y^i, \nabla y^i) \right).
\end{equation}
The desired result is obtained by combining the last two equations.
\end{proof}

\subsection{Proof of rigidity in Theorem \ref{B}}
Let $(M^4, g,\E)$ be a complete ALF manifold with nonnegative scalar curvature, zero mass, and having an almost free $U(1)$ action with respect to end $\mathcal{E}$. It will be assumed that there is a single end for convenience. The case of multiple ends may be treated similarly, by extending Lemma \ref{exist-harmoinic} so that the harmonic functions asymoptotically vanish in the additional ends.

We will follow the strategy of \cite{schoen-yau1979}*{pgs. 72-74} used in the AE setting. In particular, the scalar curvature must vanish $R_g =0$, otherwise Theorem \ref{conformal} and Remark \ref{aofnoianfoianbgoianoignoiagn} may be used to conformally transform $(M^4 ,g)$ to zero scalar curvature and negative mass, which violates the inequality of Theorem \ref{B} established in Section \ref{sec5}. Moreover, the Ricci curvature must also vanish, otherwise an infinitesimal Ricci flow combined with a conformal change (Remark \ref{aofnoianfoianbgoianoignoiagn}) can be used to find nearby metrics with zero scalar curvature and negative mass, again violating the positive mass inequality. At this point, we may apply Lemma \ref{harmonic} to find functions $y^i$, $i=1,2,3$ with parallel differentials $dy^i$.
Consequently, the 1-form $\star_g(dy^1 \wedge d y^2 \wedge dy^3)$ is also parallel, yielding a globally parallel frame for the cotangent bundle $T^* M^4$. 
It follows that the metric $g$ is flat. Arguments at the end of the proof of \cite{minerbe}*{Theorem 3} now show that $(M^4, g)$ is isometric to $\mathbb{R}^3 \times S^1$. 

\section{Mass Lower Bounds in Terms of Degree}
\label{sec7} \setcounter{equation}{0}
\setcounter{section}{7}

The purpose of this section is to establish Theorem \ref{C}. Let $(M^4, g)$ be a complete ALF manifold having an almost free $U(1)$ action. We will begin by studying the dual 1-form $\eta$ of the $U(1)$-generator $T$, and how it enters into a lower bound for the quotient space scalar curvature. 
Recall that the principal orbit theorem implies that the quotient map 
\begin{equation}
\pi :M^4 \setminus\{p_1 ,\dots,p_k \}\rightarrow \bar{M}^3 \setminus \{\bar{p}_1 ,\dots,\bar{p}_k\}
\end{equation}
yields a principal $U(1)$ bundle and in particular gives a Riemannian submersion, where $\{p_1,\dots,p_k\}$ are the singular points of the $U(1)$ action, $\bar{p}_i=\pi(p_i)$, and $\bar{M}^3 =M^4 /U(1)$. In this context, Proposition \ref{mean-scalar-formula} yields the following expression for the scalar curvature of the quotient
\begin{equation}\label{fapijfpapp11}
R_{\bar{g}}=\bar{R}_g +\bar{\mathcal{A}}^2 +2|\bar{\mathcal{N}}|^2_{\bar{g}} -2\mathrm{div}_{\bar{g}}\bar{\mathcal{N}}.
\end{equation}
This expression yields an advantageous lower bound for the scalar curvature in terms of $\eta$.

\begin{proposition}\label{form-express}
Under the setting and notation of Proposition \ref{mean-scalar-formula}, the following identities 
\begin{equation}\label{tensor-eta}
\bar{\mathcal{A}}^2 =\frac{|d\eta|^2}{2|\eta|^2}-\frac{2|\bar{\nabla}|\eta||_{\bar{g}}^2}{|\eta|^2}, \quad \quad \quad
\bar{\mathcal{N}}=-\bar{\nabla}\log|\eta|,
\end{equation}
are valid on $\bar{M}^3 \setminus \{\bar{p}_1 ,\dots,\bar{p}_k\}$. Furthermore, if the total space scalar curvature is nonnegative $R_g \geq 0$ then
\begin{equation}
R_{\bar{g}}\geq \frac{|d\eta|^2}{2|\eta|^2}+2\Delta_{\bar{g}}\log|\eta| .
\end{equation}
\end{proposition}

\begin{proof} 
Let $\bar{X}_i$, $i=1,2,3$ be a local orthonormal frame on $\bar{M}^3$, then there exists a unique set of horizontal orthonormal vectors $X_i$, $i=1,2,3$ such that $\pi_{*}(X_i)=\bar{X}_i$. We also set $U_1 =T/|T|_g$. Then with the help of \eqref{foanoifnpqaoinfpoiqnif} it holds that
\begin{equation}\label{fq9rj877733rhh}
\bar{\mathcal{A}}^2 = \frac{1}{4}\sum_{i,j=1}^{3}g(U_1 ,[X_i , X_j])^2
=\frac{1}{4}\sum_{i,j=1}^{3}\frac{|d\eta(X_i ,X_j)|^2}{|\eta|^2}
=\frac{|d\eta|^2}{2|\eta|^2} -\frac{1}{2}\sum_{i=1}^3 \frac{|d\eta(X_i ,U_1)|^2}{|\eta|^2},
\end{equation}
where we have used
\begin{equation}
d\eta(X_i,X_j)=X_i \left(\eta(X_j)\right) -X_j \left(\eta(X_i)\right) -\eta([X_i,X_j])=-\eta([X_i ,X_j]).
\end{equation}
Furthermore, \eqref{fapjpajpfonjapogjpoqajg} implies that
\begin{equation}\label{28hgh39}
d\eta(X_i ,T)=X\left(\eta(T)\right)-T\left(\eta(X_i)\right)-\eta([X_i,T])=X_i(|\eta|^2).
\end{equation}
Combining \eqref{fq9rj877733rhh} and \eqref{28hgh39} produces the first identity of \eqref{tensor-eta}. The second desired identity follows immediately from \eqref{mean-local-T}. Moreover, inserting the formulas of \eqref{tensor-eta} into \eqref{fapijfpapp11}
then gives the scalar curvature lower bound.
\end{proof}

\subsection{Proof of Theorem \ref{C}}
Let $(M^4, g,\E)$ be a complete ALF manifold with nonnegative scalar curvature, mass $m$, and having an almost free $U(1)$ action with respect to end $\mathcal{E}$. It will be assumed that there is a single end for convenience; the case of multiple ends may be treated similarly. Given $\varepsilon>0$, Theorem \ref{ALF-density} provides an approximation manifold $(M^4, g' ,\E)$ with nonnegative scalar curvature, a mass satisfying $|m -m'|<\varepsilon$, and having an almost free $U(1)$ action with respect to $\E$ in which the Killing field generator agrees with the vector field $V$ associated with the model metric $g_0$. In particular,
there exists a function $f\in C^{\infty}(M^4)$ and a compact set $\mathcal{K}\subset M^4$ such that 
\begin{equation}\label{conformal-def11234}
g'=f^2 g_0 \quad\text{ on }\E \setminus\mathcal{K},\quad\quad\quad f-\left(1+\frac{m'}{6{r}}\right)\in C^{2, \alpha}_{1+q'}(\E),
\end{equation}
where $q'=\min\{1,q\}$, $\alpha\in(0,1)$. 
Moreover, with the help of Remark \ref{singular-beh} there exists a neighborhood $\bar{U}\subset \bar{M}^3$ of the singular set such that 
\begin{equation}\label{fonaf09hq0-9hg}
|\eta|=\bar{r}+O(\bar{r}^2) \quad  \text{ on } \bar{U},\quad\quad\quad\quad  |\eta|-\ell \left(1+\frac{m'}{6r}\right)\in C^{2,\alpha}_{1+q'}(\E),
\end{equation}
where $\bar{r}$ denotes the $\bar{g}$-distance to the singular set; we are slightly abusing notation in the first equation by continuing to denote by $|\eta|$ the descent of this function to the quotient space. Here and in what follows, $\eta$ and $\bar{g}$ are with respect to $g'$.

According to Theorem \ref{factor}, there exists a function $\bar{u}\in W^{1,2}_{loc}(\bar{M}^3)\cap C^\infty(\bar{M}^3 \setminus \{\bar{p}_1,\cdots, \bar{p}_k\})$ which is positive away from the singular points $\bar{p}_i$ and satisfies the following equation with asymptotics
\begin{equation}\label{factordecaysingular}
\Delta_{\bar{g}}\bar{u}-\frac{1}{8}R_{\bar{g}}\bar{u}=0 \quad \text{ on }  \bar{M}^3, \quad\quad\quad \bar{u}=O(\bar{r}^{1/2})\quad \text{ on } \bar{U}, \quad\quad
 \quad  \bar{u}-\left(1+\frac{\bar{\mathcal{C}}}{r}\right)\in C^{2,\alpha}_{1+q'}(\bar{\E}),
\end{equation}
for some constant $\bar{\mathcal{C}}\leq 0$. Then Remark \ref{angle} and Proposition \ref{form-express} may be combined to yield the mass lower bound
\begin{align}\label{theorem-c-step1}
\begin{split}
m'&\geq 
\frac{3}{\pi}\int_{\bar{M}^3}\left(|\bar{\nabla}\bar{u}|_{\bar{g}}^2 +\frac{1}{8}R_{\bar{g}}\bar{u}^2 \right)\mathrm{dV}_{\bar{g}}\\
&\geq \frac{3}{8\pi}\int_{\bar{M}^3}\left( 8|\bar{\nabla}\bar{u}|^2_{\bar{g}}+\frac{|d\eta|^2 \bar{u}^2}{2|\eta|^2}+2\bar{u}^2 \Delta_{\bar{g}}\log |\eta| \right)\mathrm{dV}_{\bar{g}}\\
&=\frac{3}{8\pi}\int_{\bar{M}^3} \left(8|\bar{\nabla}\bar{u}|^2_{\bar{g}}+\frac{|d\eta|^2 u^2}{2|\eta|^2}-4\bar{u}\bar{g}(\bar{\nabla}\bar{u},\bar{\nabla}\log |\eta|) \right)\mathrm{dV}_{\bar{g}} -\frac{m'}{2},
\end{split}
\end{align}
where in the last step \eqref{fonaf09hq0-9hg} was used to compute the boundary terms both at the singularities and at infinity arising from an integration by parts.
Next observe that \eqref{28hgh39} implies
\begin{equation}
|d\eta|^2 \geq \sum_{i=1}^3 |d\eta(X_i, U_1)|^2 =4|\nabla|\eta| |_{g'}^2,
\end{equation}
and therefore 
by transforming the mass lower bound into an integral over the total space, we obtain
\begin{equation}\label{findualterm1}
m'\geq\frac{1}{8\pi^2}\int_{M^4} \left(\frac{8|\nabla u|_{g'}^2}{|\eta|}+\frac{|d\eta|^2 u^2}{4|\eta|^3}+\frac{|\nabla|\eta||_{g'}^2 u^2}{|\eta|^3}-\frac{4ug'(\nabla |\eta|, \nabla u)}{|\eta|^2} \right)\mathrm{dV}_{\!\!g'},
\end{equation} 
where $u=\pi^*\bar{u}$.

In order to establish the relation between mass and degree, we will make use of a particular feature of dimension 4 when treating the term involving $d\eta$, namely $|d\eta|^2 \mathrm{dV}_{\!\!g'}\geq d\eta \wedge d\eta$. By combining this with Stoke's theorem, together with \eqref{local-killing} and \eqref{factordecaysingular} for the decay near singular points,
and the fact that the Cauchy-Schwarz inequality holds for wedge products involving 1-forms \cite{Fed}*{Section 1.7.5}, we find that
\begin{align}\label{detadegree}
\begin{split}
\int_{M^4}\!\!\frac{|d\eta|^2 u^2}{|\eta|^3}\mathrm{dV}_{\!\!g'} &\geq
\frac{1}{2}\int_{M^4} \frac{u^2}{|\eta|^{3}} d\eta\wedge d \eta +\frac{1}{2}\int_{M^4}\frac{|d\eta|^2 u^2}{|\eta|^3}\mathrm{dV}_{\!g'}  \\
&=\lim_{r\rightarrow\infty}\frac{1}{2}\int_{\mathcal{S}_r^3} \frac{u^2}{|\eta|^{3}}\eta\wedge d\eta
-\frac{1}{2}\int_{M^4}d(u^2|\eta|^{-3})\wedge \eta\wedge d\eta +\frac{1}{2}\int_{M^4}\frac{|d\eta|^2 u^2}{|\eta|^3}\mathrm{dV}_{\!\!g'}\\
&\geq \lim_{r\rightarrow\infty}\frac{\ell}{2}\int_{\mathcal{S}_r^3}\frac{\eta}{|\eta|^2}\wedge d\left(\!\frac{\eta}{|\eta|^2}\!\right)\!+\!\int_{M^4}\!\!\left( \frac{|d\eta|^2u^2}{2|\eta|^3}-|\eta||\nabla(u|\eta|^{-\frac{3}{2}})|_{g'} \cdot \frac{|d\eta||u|}{|\eta|^{\frac{3}{2}}}\right)\!\mathrm{dV}_{\!\!g'} \\
&\geq 2\pi^2 \ell\deg(\E)-\frac{1}{2}\int_{M^4} |\eta|^2|\nabla (u|\eta|^{-\frac{3}{2}})|_{g'}^2 \mathrm{dV}_{\!\!g'},
\end{split}
\end{align}
where in the last step we have used Proposition \ref{degree-eta} to identify the bundle degree. Note that the degree is an invariant of the background bundle structure of the end, and thus does not depend on whether it is computed with respect to $g$ or $g'$. If $\deg(\E)\geq 0$, then inserting \eqref{detadegree} into \eqref{findualterm1} produces
\begin{align}\label{pahfihqit933}
\begin{split}
m' 
&\geq \frac{\ell}{16}|\deg(\E)|+\frac{1}{8\pi^2}\int_{{M}^4} \left(\frac{63}{8}\frac{|{\nabla}{u}|_{g'}^2}{|\eta|}+\frac{23}{32}\frac{|\nabla|\eta||_{g'}^2u^2}{|\eta|^3}-\frac{29}{8}\frac{ug'({\nabla} |\eta|, {\nabla}u)}{|\eta|^2} \right)\mathrm{dV}_{\!\!g'}\\
&\geq \frac{\ell}{16}|\deg(\E)|+\frac{1}{8\pi^2}\int_{{M}^4} \frac{|{\nabla}{u}|_{g'}^2}{|\eta|}\mathrm{dV}_{\!\!g'} .
\end{split}
\end{align}
If $\deg(\E)<0$ then in \eqref{detadegree} we may use the inequality $|d\eta|^2 \mathrm{dV}_{\!\!g'}\geq -d\eta \wedge d\eta$, and follow the same manipulations to achieve \eqref{pahfihqit933}. Since $\varepsilon$ was arbitrary, this implies the desired lower bound \eqref{foainfinapfinpqinp} for the original mass $m$.

\begin{remark}\label{notsharp}
The inequality \eqref{foainfinapfinpqinp} is never saturated, and is thus not sharp. To see this, assume to the contrary that equality is achieved. As $\varepsilon \rightarrow 0$, the methods of Theorem \ref{factor} can be used to find a subsequence of functions $u$ converging strongly on compact subsets of $M^4$ to some $u_0$ which vanishes at singular points, and converges to 1 a infinity. Then applying Fatou's lemma to \eqref{pahfihqit933} shows that $|\nabla u_0|_g =0$, yielding a contradiction.
We note that the multi-Taub NUT instanton satisfies $m=\frac{1}{2}\deg(\E)$ with $\ell=1$, and so one may speculate that the best constant could be $1/2$.
\end{remark}

\section{Stable Minimal Hypersurfaces in AF Manifolds}
\label{sec8} \setcounter{equation}{0}
\setcounter{section}{8}

In this section we will construct stable minimal hypersurfaces under the assumptions of Theorem \ref{A}. Consider a complete AF manifold $(M^n, g, \E)$ with distinguished end $\E$. The AF structure diffeomorphism 
\begin{equation}
\Psi: (\mathbb{R}^{n-1}\setminus B_{1})\times S^1 \rightarrow \E
\end{equation}
yields local coordinates $(x,\theta)$ in the asymptotic end. We will denote by $\mathcal{S}_{r,\theta}\subset \E$ the intersection of the $r$ and $\theta$ level set, and will refer to it as a \textit{coordinate sphere}.
In what follows, we will assume that $\mathcal{S}_{r, \theta}$ is homologically trivial in $H_{n-2}(M^n)$, and $3\leq n\leq 7$. According to Federer-Fleming theory \cite{FF}*{Corollary 9.6} (see also \cite{Simon83}*{Lemma 34.1}) and the regularity result \cite{Simon83}*{Theorem 37.7}, for each $r \geq 1$ and $\theta\in S^1$ there exists an integral sum of smooth oriented minimal hypersurfaces $\S_{r, \theta}\subset M^n$ with $\partial\S_{r,\theta}=\mathcal{S}_{r,\theta}$ that minimizes volume among all smooth $(n-1)$-chains $\Sigma$ for which $\partial\Sigma=\mathcal{S}_{r,\theta}$, that is
\begin{equation}
\mathcal{H}^{n-1} (\S_{ r, \theta})=\inf_{\partial \S=\mathcal{S}_{r,\theta}}\mathcal{H}^{n-1}(\S)
\end{equation}
where $\mathcal{H}^{n-1}$ denotes $(n-1)$-dimensional Hausdorff measure.
Note that $\Sigma_{r,\theta}$ has finite volume regardless of the existence of multiple ends in $M^n$, and is smooth since $n\leq 7$. We will study the topological and geometric properties of $\Sigma_{r,\theta}$, and then make use of these to establish the existence of a limiting minimal surface with desirable properties as $r\rightarrow \infty$. 

\subsection{Local volume estimates}
In order to study the convergence of the surfaces $\Sigma_{r,\theta}$ it is essential to obtain local volume estimates.
We will make use of the following regions within the asymptotic end
\begin{equation}
\Pi_{r_0, r_1}=\Psi(\{(x,\theta)\mid r_0\leq |x|\leq r_1 \}), \quad\quad\quad Z_{p, r}=\Psi(\{(x,\theta)\mid |x-p|\leq r\}),
\end{equation}
where $p\in \bb{R}^3\setminus B_{1}$. The next result is a consequence of the volume-minimizing property.

\begin{proposition}\label{vol1} 
For any $1<r_1<r$ and $\theta\in S^1$ there exists a constant $C$ independent of these values such that 
\begin{equation}
\mathcal{H}^{n-1}(\Sigma_{r, \theta}\cap (M^n\setminus \mathcal{E}))\leq C , \quad\quad\quad
\mathcal{H}^{n-1} (\Sigma_{r, \theta}\cap \Pi_{1, r_1})\leq C r_1^{n-1}.
\end{equation}
\end{proposition}

\begin{proof} 
Let $\mathcal{S}_r \subset\E$ denote the $r$ level set. We may assume without loss of generality that $\S_{ r, \theta}$ intersects $\mathcal{S}_{r_1}$ transversely, so that $\Sigma^{r_1}_{r,\theta}:=\S_{r,\theta} \cap \mathcal{S}_{r_1}$ is an $(n-2)$-dimensional submanifold. To see this transversality, consider the function $(x,\theta)\mapsto |x|$ restricted to $\Sigma_{r,\theta}$. This is a smooth map, and thus by Sard's theorem, arbitrarily close to any $r_1$ we may find a regular value. Since critical values indicate when a point of tangency to the vertical coordinate surface occurs, a given $r_1$ may always be perturbed to ensure transversality. 

We will next show that $\Sigma^{r_1}_{r,\theta}$ and $\mathcal{S}_{r_1, \theta}$ are homologous in $H_{n-2}(\mathcal{S}_{r_1})$. Consider the $(n-1)$-dimensional submanifolds 
\begin{equation}
\Sigma'_{r, \theta}:=\Sigma_{r,\theta}\cap \Pi_{r_1, \infty},\quad\quad \quad  P^\theta_{r_1, r}:=P^\theta\cap \Pi_{r_1, r},
\end{equation}
where $P^\theta$ denotes the $\theta$ level set.  Observe that $\partial \Sigma'_{r, \theta}=\mathcal{S}_{r, \theta}-\Sigma^{r_1}_{r,\theta}$ and $\partial P^\theta_{r_1, r}=\mathcal{S}_{r, \theta}-\mathcal{S}_{r_1, \theta}$, which implies that their difference is an $(n-2)$-cycle in the coordinate surface at radius $r_1$, more precisely 
\begin{equation}\label{hologvanish}
\partial(\S'_{r, \theta}-P^\theta_{r_1, r})=\mathcal{S}_{r_1, \theta}-\Sigma^{r_1}_{r,\theta}. 
\end{equation} 
The chain $\S'_{r, \theta}-P^\theta_{r_1, r}$ is thus a representative of an element in $H_{n-1}(\Pi_{r_1, \infty}, \mathcal{S}_{r_1})$. Moreover, a basic computation yields
\begin{equation}\label{foapjf-ja-}
H_{n-1}(\Pi_{r_1, \infty}, \mathcal{S}_{r_1})\cong H_{n-1}(\bb{R}^{n-1}\times \bb{S}^1, B_{r_1}\times \bb{S}^1)=0,  
\end{equation} 
where we have used the excision property \cite{HA1}*{Theorem 2.20} in the first isomorphism. It follows that the image of the boundary homomorphism $H_{n-1}(\Pi_{r_1, \infty}, \mathcal{S}_{r_1})\rightarrow H_{n-2}(\mathcal{S}_{r_1})$ is trivial. Combining this with \eqref{hologvanish} gives the claimed relation between $\Sigma^{r_1}_{r,\theta}$ and $\mathcal{S}_{r_1, \theta}$.

In light of the previous paragraph, we may appeal again to Federer-Fleming theory to find a volume-minimizing chain $W$ with $\p W=\Sigma^{r_1}_{r,\theta}-\mathcal{S}_{r_1,\theta}$, in particular $\mathcal{H}^{n-1}(W)\leq \mathcal{H}^{n-1}(\mathcal{S}_{r_1})$. Consider a $(n-1)$-chain
\begin{equation}
\hat{\S}_{r,\theta}:=\S'_{r,\theta}+ W+P^{\theta}_{1, r_1}+ \Sigma_{1,\theta},
\end{equation}
and note that $\partial\hat{\S}_{r,\theta}=\mathcal{S}_{r,\theta}$. Set $\Sigma''_{r,\theta}=\Sigma_{r,\theta}\setminus\Sigma'_{r,\theta}$ and use the volume-minimizing property of $\Sigma_{r,\theta}$ to find
\begin{align}
\begin{split}
\mathcal{H}^{n-1}(\Sigma'_{r,\theta})+\mathcal{H}^{n-1}(\Sigma''_{r,\theta})&
=\mathcal{H}^{n-1}(\Sigma_{r,\theta})\\
&\leq\mathcal{H}^{n-1}(\hat{\Sigma}_{r,\theta})\\
&=\mathcal{H}^{n-1}(\Sigma'_{r,\theta})+\mathcal{H}^{n-1}(W)
+\mathcal{H}^{n-1}(P^{\theta}_{1,r_1})+\mathcal{H}^{n-1}(\Sigma_{1,\theta}),
\end{split}
\end{align}
so that
\begin{equation}
\mathcal{H}^{n-1}(\S''_{r,\theta})\leq \mathcal{H}^{n-1}(W)+\mathcal{H}^{n-1}(P^{\theta}_{1, r_1})+\mathcal{H}^{n-1}(\S_{1,\theta}).
\end{equation}
In addition, we have the following estimates
\begin{align}\label{fpaif]hq09htg9q3g}
\begin{split}
\mathcal{H}^{n-1}(W)&\leq\mathcal{H}^{n-1}(\mathcal{S}_{r_1})\leq C_0 r_1^{n-2} ,\\ 
\mathcal{H}^{n-1}(P^{\theta}_{1, r_1})&\leq C_0 r^{n-1}_1 ,\\ 
\mathcal{H}^{n-1}(\S_{1,\theta})&\leq \sup_{\theta\in S^1}\mathcal{H}^{n-1}(\S_{1,\theta})\leq C_0,
\end{split}
\end{align} 
for some constant $C_0$ independent of $r$, $r_1$, and $\theta$.
Note that the validity of the last inequality is due to the fact that the volume of any surface $\Sigma_{1,\theta}$ can be estimated by a competitor consisting of the $\theta=0$ surface and a vertical portion of the coordinate surface $\mathcal{S}_1$ (which has volume less than $\mathcal{H}^{n-1}(\mathcal{S}_1)$). Therefore
\begin{equation}
\mathcal{H}^{n-1} (\Sigma_{r,\theta}\cap \Pi_{1, r_1})\leq \mathcal{H}^{n-1}(\Sigma''_{\a,r})\leq C_1 r^{n-1}_1.
\end{equation}
Lastly, similar arguments as those used for the last inequality of \eqref{fpaif]hq09htg9q3g} produce
\begin{equation}
\mathcal{H}^{n-1}(\Sigma_{r,\theta}\cap (M^n\setminus \mathcal{E}))\leq C_{2}.
\end{equation}
Setting $C =\max\{C_1, C_2\}$ yields the desired result.
\end{proof}

As a consequence of the arguments in the proof of the previous proposition, one may also establish
the following additional volume control.

\begin{corollary}\label{Vol2} 
For any $1<r_1<r$, $\theta\in S^1$, and $p\in\E$ such that $\min\{r(p), r-r(p)\}\geq 2r_1$, there
exists a constant $C$ independent of these quantities such that 
\begin{equation}
\mathcal{H}^{n-1}(\Sigma_{r,\theta}\cap Z_{p, r_1})\leq C r^{n-1}_1.
\end{equation}
\end{corollary}
 
\subsection{Convergence of the minimal hypersurfaces}

A complete minimal surface may be extracted from the Plateau solutions by letting the radial parameter tend to infinity.

\begin{proposition}\label{exist}
Let $(M^n, g)$ be a complete AF manifold with $3\leq n\leq 7$. If some coordinate sphere $\mathcal{S}_{r,\theta}$ of an end $\mathcal{E}$ is trivial in homology $H_{n-2}(M^n;\mathbb{Z})$, then there exists a complete properly embedded minimal hypersurface $\Sigma_\infty$ which is stable against compactly supported variations. Furthermore, for any $p\in \E$ and $r_1 >1$ with $r(p)\geq 2r_1$, there exists a constant $C$ independent of these quantities such that
\begin{equation}\label{fpanpfinpanpion}
\mathcal{H}^{n-1}(\Sigma_\infty\cap(M^n\setminus \mathcal{E}))\leq C,\quad\quad\quad
\mathcal{H}^{n-1}(\Sigma_\infty\cap Z_{p, r_1})\leq C r^{n-1}_1 .
\end{equation}
\end{proposition}

\begin{proof}
We first claim that $\Sigma_{r,\theta}\cap (M^n \setminus \E)$ is nonempty for all $r>1$ and $\theta\in S^1$.
If not, then there is some $\S_{r,\theta}\subset \mathcal{E}\cong (\mathbb{R}^{n-1}\setminus B_1)\times S^1$. It follows that the image of the homology class $[\mathcal{S}_{r,\theta}]$ vanishes in $H_{n-2}((\mathbb{R}^{n-1}\setminus B_1)\times S^1)$, since $\partial \S_{r,\theta}=\mathcal{S}_{r,\theta}$. However, if $n\geq 4$ then $[\mathcal{S}_{r,\theta}]$ is a generator of $H_{n-2}((\mathbb{R}^{n-1}\setminus B_1)\times S^1)\cong \bb{Z}$, a contradiction. If $n=3$, then the same argument holds except that the homology is $\mathbb{Z}\oplus\mathbb{Z}$, and the coordinate circle generates the first factor.

Consider a sequence of radial values $r_j \rightarrow\infty$, and choose a corresponding sequence of points $\theta_j \in S^1$ such that
\begin{equation}
\mathcal{H}^{n-1} (\S_{r_j , \theta_j})=\inf_{\theta\in S^1} \mathcal{H}^{n-1} (\S_{r_j , \theta}).
\end{equation}
Note that according to the arguments of \cite{LUY}*{Lemma 4.1}, the map $\theta\mapsto \mathcal{H}^{n-1}(\Sigma_{r,\theta})$ is continuous for each $r$, ensuring the existence of $\theta_j$ for each $j$. In light of the volume estimates
Proposition \ref{vol1} and Corollary \ref{Vol2}, as well as the above nonempty intersection property, standard minimal surface theory shows that after passing to a subsequence $\Sigma_{r_j , \theta_j}$ converges in the pointed $C^k$ topology (for any $k$) to a complete properly embedded stable minimal surface $\S_\infty$. Furthermore, the strong convergence together with Proposition \ref{vol1} and Corollary \ref{Vol2} imply the estimates \eqref{fpanpfinpanpion}.
\end{proof}
 
\subsection{Topological aspects of the end of the minimal surface}
For any $r>1$ consider the projection map
\begin{equation}
\mathbf{p}: \Pi_{r,\infty}\subset\mathcal{E}\rightarrow \mathbb{R}^{n-1}\setminus B_r , 
\end{equation}
given by $\mathbf{p}\circ \Psi(x,\theta)=x$. Denote by $\mathbf{p}_{\Sigma}$ the restriction of the projection to the hypersurface $\Sigma_{\infty}$. The next result states in particular that this map is surjective.

\begin{lemma}\label{degree} 
The degree of the restricted projection map is $\mathrm{deg } (\mathbf{p}_{\Sigma})=\pm 1$. 
\end{lemma}

\begin{proof} 
Let $y\in \mathbb{R}^{n-1} \setminus B_r$ be a regular value of the map $\mathbf{p}_{\Sigma}$, and observe that the degree
is equal to the intersection number of the surface and the fiber over $y$, that is $\mathbf{I}(\Sigma_{\infty},\mathbf{p}^{-1}(y))$.
As in Proposition \ref{vol1}, let $\Sigma_{r' ,\theta}$ denote a solution of the Plateau problem with $\partial\Sigma_{r',\theta}=\mathcal{S}_{r',\theta}$ and $r'>|y|>r$. Consider the $(n-1)$-chain 
\begin{equation}
(\Sigma_{r',\theta}\cap \Pi_{r,\infty})-P^{\theta}_{r,r'},
\end{equation}
and note that its boundary lies in $\mathcal{S}_r$. Hence
\begin{equation}
[(\Sigma_{r',\theta}\cap \Pi_{r,\infty})-P^{\theta}_{r,r'}]\in H_{n-1}(\Pi_{r,\infty},\mathcal{S}_r)=0,
\end{equation}
where the homology computation was accomplished in \eqref{foapjf-ja-}. It follows that the intersection number of  
$(\Sigma_{r',\theta}\cap \Pi_{r,\infty})-P^{\theta}_{r,r'}$ and $\mathbf{p}^{-1}(y)$ is zero.
Furthermore, since $\mathbf{p}^{-1}(y)$ intersects $P^{\theta}_{r, r'}$ transversely at a unique point we have $\mathbf{I}(P^{\theta}_{r, r'},\mathbf{p}^{-1}(y))=\pm 1$, which then implies that $\mathbf{I}(\Sigma_{r',\theta}\cap \Pi_{r,\infty},\mathbf{p}^{-1}(y))=\pm 1$.
Now use that $\Sigma_{\infty}=\lim_{j\rightarrow\infty}\Sigma_{r_j ,\theta_j}$ for some sequence $(r_j ,\theta_j)$ with $r_j \rightarrow\infty$, to conclude that $\mathbf{I}(\Sigma_{\infty},\mathbf{p}^{-1}(y))=\pm 1$ as well.
\end{proof}

In order to lift the minimal surface to the universal cover of the end, we will need the following fact concerning fundamental groups.

\begin{proposition}\label{group}
The induced inclusion homomorphism $i_* :\pmb{\pi}_1(\Sigma_\infty\cap \Pi_{r,\infty})\rightarrow \pmb{\pi_1}(\Pi_{r,\infty})$ is trivial for any $r\geq 1$. 
\end{proposition}

\begin{proof} 
We will argue by contradiction. Suppose that $i_*$ is nontrivial, then there is a nonzero integer $k$ and a closed curve $c\subset \S_\infty \cap\Pi_{r,\infty}$ such that $i_*[c] =k[\mathbf{p}^{-1}(y)]$ in $\pmb{\pi_1}(\Pi_{r,\infty})$, where $y\in \mathbb{R}^{n-1} \setminus B_r$. It follows from Lemma \ref{degree} that the intersection number of $c$ and $\Sigma_{\infty}\cap \Pi_{r,\infty}$ is $\pm k\neq 0$. 
On the other hand, if $N(\Sigma_{\infty} \cap \Pi_{r,\infty})$ is a tubular neighborhood of $\Sigma_\infty\cap \Pi_{r,\infty}$ having sufficiently small radius, then we may find a closed curve $c'\subset \partial N(\Sigma_{\infty}\cap \Pi_{r,\infty}) \cap \mathrm{Int }\text{ }\!\Pi_{r,\infty}$ which is homotopic to $c$ in $\Pi_{r,\infty}$. 
However, $\partial N(\Sigma_{\infty}\cap \Pi_{r,\infty})\cap \mathrm{Int} \text{ }\!\Pi_{r,\infty}$ has empty intersection with $\Sigma_{\infty}$. Hence $\mathbf{I}(c,\Sigma_{\infty}\cap \Pi_{r,\infty})=0$, yielding a contradiction.
\end{proof}

In light of Proposition \ref{group}, we may lift the surface $\Sigma_{\infty}\cap \Pi_{r,\infty}$ to the universal cover $(\hat{\E},\hat{\mathbf{p}}^{*}g)$ of the end $\E$, where $\hat{\mathbf{p}}:\hat{\E}\rightarrow\E$ is the covering map.
Let $\hat{\Sigma}_{r,\infty}$ denote such a lifting, and let $\hat{\mathbf{p}}_{\hat{\Sigma}}$ be the restriction of the covering map to this lifted surface. We observe that the following is an immediate consequence, where the volume estimate arises from Proposition \ref{exist}.

\begin{corollary}\label{cover} 
The map $\hat{\mathbf{p}}_{\hat{\Sigma}}:\hat{\Sigma}_{r,\infty}\rightarrow\Sigma_{\infty}\cap \Pi_{r,\infty}$ is an isometry. Moreover, for any $p\in E$ and $r_1 >1$ there exists a constant $C$ independent of these values such that
\begin{equation}\label{foajbnfoinqoh454}
\mathcal{H}^{n-1}(\hat{\mathbf{p}}^{-1}(Z_{p, r_1})\cap \hat{\Sigma}_{r,\infty})\leq C r_1^{n-1} .
\end{equation}
\end{corollary}

\section{Asymptotics of Minimal Hypersurfaces in AF Manifolds}
\label{sec9} \setcounter{equation}{0}
\setcounter{section}{9} 

In this section we will analyze the asymptotics of the stable minimal surface $\Sigma_{\infty}$ arising from Proposition \ref{exist}, and prove Theorem \ref{A}. In order to facilitate the analysis, we will assume initially that the AF end possesses strong asymptotics. More precisely, in analogy with Definition \ref{foanofinapinp} an end $(\E,g)$ of an AF manifold will be called \textit{harmonically AF} if the following decay holds
\begin{equation}
\Big|\mathring{\nabla}^l \left(g-\left(1+\frac{m}{c_n r^{n-3}}\right)^{\frac{4}{n-2}} g_0 \right)\Big|_{g_0}=O(r^{-\mathring{q}-l}), \quad\quad l=0,1,2,
\end{equation}
for some $m\in\mathbb{R}$ and $\mathring{q}>n-3$, where $g_0 =dr^2 +r^2 g_{S^{n-2}} +\ell^2 d\theta^2$ and $c_n =\tfrac{4(n-1)(n-3)}{n-2}$. Note that the parameter $m$ is the mass of the end $(\E,g)$.

\begin{theorem}\label{asy} 
Let $(M^n, g)$ be a complete AF manifold with $4\leq n\leq 7$, such that some coordinate sphere $\mathcal{S}_{r,\theta}$ of an end $\mathcal{E}$ is trivial in homology $H_{n-2}(M^n)$. Let $\S_\infty$ be the corresponding complete minimal hypersurface of Proposition \ref{exist}. If $(\E,g)$ is harmonically AF with mass $m$, then $\Sigma_\infty$ has a single end
$(\mathcal{E}_{\Sigma},g_{\Sigma})$ that is asymptotically Euclidean (with integrable scalar curvature), and has ADM mass $m_{\Sigma}$ that satisfies the relation
$m=\tfrac{2(n-1)^2}{n-2}m_{\Sigma}$.
\end{theorem}

\subsection{Graphical properties of the minimial surface} 
In order to establish Theorem \ref{asy}, we will first show that $\Sigma_{\infty}\cap \E$ may be realized as a graph with fast decay in the universal cover $(\mathbb{R}^{n-1}\setminus B_1)\times \mathbb{R}$ of the end. Consider the universal cover
$\hat{\mathbf{p}}:\hat{\E}\rightarrow\E$ with pullback metric $\hat{g}=\hat{\mathbf{p}}^{*}g$ given in local coordinates $(x,t)$ by
\begin{equation}
\hat{g}(x, t)=\left(1+\frac{m}{c_n r^{n-3}}\right)^{\frac{4}{n-2}}\left(\sum_{i=1}^{n-1}(dx^i)^2+dt^2\right)+O_2 (r^{-\mathring{q}}),   
\end{equation}
where $r=|x|$. Proposition \ref{group} implies the existence of a lift of the surface $\Sigma_{\infty}\cap \E$ to the universal cover, which will be denoted by $\hat{\Sigma}_{1,\infty}$. This is a stable minimal surface itself, and Corollary \ref{cover}  shows that the covering map restriction $\hat{\mathbf{p}}_{\hat{\Sigma}}:\hat{\Sigma}_{1,\infty}\rightarrow\Sigma_{\infty}\cap\E$ is an isometry.


\begin{lemma}\label{graph} 
There exists a radius $r_0 >1$ such that $\hat{\S}_{r_0 ,\infty}$ may be represented as the graph of a smooth function $t=f(x)$ over $\mathbb{R}^{n-1}\setminus B_{r_0}$. Furthermore, the graph function satisfies the following decay properties 
\begin{align}\label{fainf09a090a-9-}
f-a \in C_{n-4}^{3,\alpha}(\mathbb{R}^{n-1}\setminus B_{r_0})
\end{align}
for some $\alpha\in (0,1)$ and $a\in \mathbb{R}$.
\end{lemma}

\begin{proof}
This result can be established with the arguments of \cite{LUY}*{Theorem 4.5}, in which a stronger ambient metric decay is assumed. We only briefly sketch the proof here, to indicate that the weaker hypothesis of the current lemma is sufficient.
In light of the volume estimate of Corollary \ref{cover}, the curvature estimates of 
Schoen, Simon, and Yau \cites{SS81,SSY75} apply to $\hat{\Sigma}_{1,\infty}$ to yield a basic decay of the second fundamental form
$|\hat{\mathbf{A}}|_{g_{\hat{\Sigma}}}=O(r^{-1})$, where $g_{\hat{\Sigma}}$ is the induced metric on the surface.
Using this decay, a computation shows that
\begin{equation}
|\hat{\mathbf{A}}-\hat{\mathbf{A}}_{\delta}|_{g_{\hat{\Sigma}}}=O(r^{2-n}),
\end{equation}
where $\hat{\mathbf{A}}_{\delta}$ is the second fundamental form with respect to the background Euclidean metric. It follows that the corresponding Euclidean mean curvature satisfies $\hat{H}_{\delta}=O(r^{2-n})$.
In particular with the help of \eqref{foajbnfoinqoh454}, for any $s>n-1$ and $\hat{p}\in\hat{\E}$ it holds that
\begin{equation}
\left(\int_{\hat{B}_\rho(\hat{p})\cap \hat{\S}_{1,\infty}} |\hat{H}_{\delta}|^s d\mathcal{H}^{n-1}\right)^{\frac{1}{s}}\rho^{1-\frac{n-1}{s}}\leq C \rho^{3-n}.
\end{equation}
Thus, since $n>3$ the left-hand side may be made smaller than any given value by choosing a sufficiently large radius, allowing for an application of the Allard regularity theorem
\cite{Simon83}*{Theorems 23.1 \& 24.2} to find that the end of the minimal surface may be written as the graph of a function $f\in C^{1,\alpha}_0(\mathbb{R}^{n-1}\setminus B_{r_0})$
for some large $r_0$ and $\alpha\in (0,1)$. 
Then since $\hat{H}_{\delta}\in C_{n-2}^{0,\alpha}(\mathbb{R}^{n-1}\setminus B_{r_0})$, we may treat the Euclidean mean curvature as a linear operator and combine standard asymptotic analysis (e.g. \cite{Lee}*{Corollary A.37}) along with a boot-strap to find
that $f-a\in C_{n-4}^{3,\alpha}(\mathbb{R}^{n-1}\setminus B_{r_0})$ for some constant $a$.
\end{proof}
 

\begin{proof}[Proof of Theorem \ref{asy}] 
First note that $\Sigma_{\infty}$ can have only one end, labeled $(\E_{\Sigma},g_{\Sigma})$, which extends into $\E$. This follows from the fact that all ends of $M^n$ are foliated by positive mean curvature surfaces that prevent $\Sigma_{\infty}$ from entering any auxiliary ends, together with proper embeddedness which precludes an end from developing in compact regions of $M^n$. According to Corollary \ref{cover} and Lemma \ref{graph}, we may take $\E_{\Sigma}$ to be diffeomorphic to $\mathbb{R}^{n-1}\setminus B_{r_0}$.
Furthermore, in the coordinates provided by this diffeomorphism the induced metric becomes
\begin{align}
\begin{split}
g_{\Sigma}&=\left(1+\frac{m}{c_n r^{n-3}}\right)^{\frac{4}{n-2}}\sum_{i=1}^{n-1}(\delta_{ij}+\partial_{x^i}f\partial_{x^j}f)dx^i dx^j+O_2 (r^{-\mathring{q}})\\
&=\left(1+\frac{4m}{(n-2)c_n r^{n-3}}\right)\sum_{i=1}^{n-1}(dx^i)^2 +O_2(r^{-\mathring{q}}),
\end{split}
\end{align}
where we have used the decay \eqref{fainf09a090a-9-} for the derivatives of $f$. It follows that $\Sigma_{\infty}$ is asymptotically Euclidean with scalar curvature fall-off $R_{g_\Sigma}=O(r^{-\mathring{q}-2})$, which is integrable.
Moreover, its mass is given by $m_{\Sigma}=\tfrac{n-2}{2(n-1)^2}m$. 
\end{proof}


\subsection{Stability inequality} 
As stated in Proposition \ref{exist}, the minimal surface $\Sigma_{\infty}$ is stable against compactly supported variations.
In fact, this statement may be strengthened by using smooth approximations and the enhanced minimizing property of the chosen sequence of Plateau solutions obtained by `height picking'. The resulting stability inequality holds for test functions that differ by a constant from weighted Sobolev spaces\footnote{The convention for the sign of the weight used here is opposite of that used in \cite{EHLR}.}. The proof of the next result is analogous to that of \cite{EHLR}*{Lemma 17} and \cite{LUY}*{Theorem 4.21}, noting that $R_g, R_{g_\Sigma}=O(r^{-\mathring{q}-2})$ and $|\mathbf{A}|_{g_{\Sigma}}=O(r^{2-n})$ where 
$\mathbf{A}$ denotes the second fundamental form of $\Sigma_{\infty}$. 


\begin{proposition}\label{strong} 
Assume the hypotheses of Theorem \ref{asy}. Let $c\in\mathbb{R}$, then for any $\varphi -c\in W^{1,2}_{\frac{n-3}{2}}(\Sigma_{\infty})$ it holds that
\begin{equation}\label{stability}
\int_{\S_\infty} \left(|\nabla \varphi |_{g_{\Sigma}}^2+\frac{1}{2}R_{g_\Sigma}\varphi^2 \right) \mathrm{dV}_{\! g_{\Sigma}}\geq \int_{\S_\infty}\frac{1}{2}(R_{g}+|\mathbf{A}|_{g_{\Sigma}}^2)\varphi^2 \mathrm{dV}_{\! g_\Sigma}.
\end{equation}
\end{proposition}

This may then be used to obtain an appropriate conformal transformation to zero scalar curvature on the minimal surface, when the ambient space is of nonnegative scalar curvature. Then an application of the AE positive mass theorem yields a sign restriction for the mass of the minimal surface.

\begin{lemma}\label{pmt-stable-hyper} 
Assume the hypotheses of Theorem \ref{asy}. If in addition the ambient scalar curvature is nonnegative $R_g \geq 0$, then the ADM mass of $\Sigma_{\infty}$ is nonnegative $m_{\Sigma}\geq 0$.
\end{lemma}


\begin{proof} 
We seek a conformal factor $w>0$ on $\Sigma_{\infty}$ with $w\rightarrow 1$ along the end $\E_{\Sigma}$ such that $w^{\frac{4}{n-3}}g_{\Sigma}$ is scalar flat. It thus suffices to set $w=1+v$ and solve
\begin{equation}\label{fpajf9aj0f9h093}
L_{g_{\Sigma}}v:=\frac{2(n-2)}{n-3}\Delta_{g_{\Sigma}}v-\frac{1}{2}R_{g_{\Sigma}}v=\frac{1}{2}R_{g_{\Sigma}} \quad\text{ on }\Sigma_{\infty},\quad\quad\quad v-\frac{a}{r^{n-3}} \in C^{2,\alpha}_{n-3+\varepsilon}(\Sigma_{\infty}),
\end{equation}
for some $\epsilon,\alpha\in(0,1)$ and $a\in\mathbb{R}$ where $L_{g_{\Sigma}}$ is the conformal Laplacian.
Since $R_{g_{\Sigma}}=O(r^{-\mathring{q}-2})$, the scalar curvature of $\Sigma_{\infty}$ lies in $L^p_{n-1+\epsilon}(\Sigma_{\infty})$ for any $p\geq 1$. We claim that 
\begin{equation}
L_{g_{\Sigma}}: W^{2,p}_{n-3+\epsilon}(\Sigma_{\infty})\rightarrow L^p_{n-1+\epsilon}(\Sigma_{\infty})
\end{equation}
is invertible for any $p>1$.
To see this, let $\varphi\in W^{2,p}_{n-3+\epsilon}(\Sigma_{\infty})$ and apply Proposition \ref{strong} to find
\begin{equation}
\int_{\S_\infty} \left(\frac{2(n-2)}{n-3}|\nabla \varphi |_{g_{\Sigma}}^2+\frac{1}{2}R_{g_{\Sigma}}\varphi^2 \right)\mathrm{dV}_{g_{\Sigma}}\geq
\int_{\S_\infty}\left(\frac{n-1}{n-3}|\nabla \varphi|_{g_{\S}}^2 +\frac{1}{2}(R_{g}+|\mathbf{A}|_{g_{\S}}^2)\varphi^2\right)\mathrm{dV}_{g_{\S}}.
\end{equation}
If $L_{g_{\S}}\varphi=0$ then integrating by parts shows that the left-hand side vanishes, and hence since $R_g \geq 0$ we have that $\varphi\equiv 0$. Since $L_{g_{\S}}$ is a Fredholm operator of index zero between the indicated weighted space \cite{Lee}*{Corollary A.42}, we conclude that a unique solution $v\in W^{2,p}_{n-3+\epsilon}(\Sigma_{\Sigma})$ of \eqref{fpajf9aj0f9h093} exists.
Moreover, by elliptic regularity and Sobolev embedding it follows that $v\in C^{2,\alpha}_{n-3+\epsilon}(\Sigma_{\infty})$ and \cite{Lee}*{Corollary A.37} shows that there exists $a\in\mathbb{R}$ such that $v$ satisfies the desired partial expansion \eqref{fpajf9aj0f9h093}. Furthermore, similar arguments to those in the proof of Theorem \ref{factor} combined with the stability inequality may be used to show that $w>0$.

Consider now the mass $\tilde{m}_{\Sigma}$ of the conformal metric $\tilde{g}_{\Sigma}:=w^{\frac{4}{n-3}}g_{\Sigma}$, which is related to the mass of the original metric by $\tilde{m}_{\Sigma}=m_{\Sigma}+2a$. Furthermore, multiplying equation \eqref{fpajf9aj0f9h093} by $w$ and integrating by parts produces
\begin{equation}\label{aofnaif-9aj-gjqa}
-2(n-2)\omega_{n-2} a=\int_{\Sigma}\left(\frac{2(n-2)}{n-3}|\nabla w|_{g_{\Sigma}}^2 +\frac{1}{2}R_{g_{\Sigma}}w^2 \right)\mathrm{dV}_{g_{\Sigma}}\geq 0,
\end{equation}
where in the last inequality we used the stability inequality \eqref{stability}. According to the AE positive mass theorem \cite{EHLR} and \eqref{aofnaif-9aj-gjqa}, we have $0\leq \tilde{m}_{\Sigma}\leq m_{\Sigma}$.
\end{proof}

\subsection{Proof of Theorem \ref{A}}
Let $(M^n, g)$ be a complete AF manifold with nonnegative scalar curvature and $4\leq n\leq 7$. According to the density result \cite{CLSZ}*{Proposition 4.11}, given $\varepsilon>0$ and an end $\E\subset M^n$ there exists a smooth complete metric $g'$ on $M^n$ satisfying the following properties: the scalar curvature is nonnegative $R_{g'}\geq 0$, the end $(\mathcal{E}, g')$ is harmonically AF, and the mass $m'$ of this end is close to the original mass $|m-m'|<\varepsilon$.
Now assume that some coordinate sphere $\mathcal{S}_{r,\theta}$ of the end $\mathcal{E}$ is trivial in homology $H_{n-2}(M^n)$. By Proposition \ref{exist} there exists a stable minimal surface $\Sigma'_{\infty}\subset (M^n ,g')$. According to Theorem \ref{asy}, this minimal surface has a single AE end whose mass $m'_{\Sigma}$ is related to that of the ambient metric via $m'=\tfrac{2(n-1)^2}{n-2}m'_{\Sigma}$. Furthermore, Lemma \ref{pmt-stable-hyper} implies that $m'_{\Sigma}\geq 0$. We conclude that $m> m' -\varepsilon\geq-\varepsilon$. Since $\varepsilon$ may be chosen arbitrarily small, it follows that $m\geq 0$. Lastly, the desired rigidity statement may be established in the same manner as the proof of rigidity in Theorem \ref{B}.

\bibliography{ALF}

\end{document}